\documentclass[11pt]{article}
\pdfoutput=1

\usepackage[dvipsnames]{xcolor}
\def\blue#1{\textcolor{blue}{#1}}
\def\blue#1{\textcolor{black}{#1}}

\def\beq{\begin{equation} }\def\eeq{\end{equation} }\def\ep{\varepsilon}\def\1{\mathbf{1}}

\usepackage[framemethod=default]{mdframed}
\usepackage{caption}

\usepackage{indentfirst}
\usepackage{bm, mathrsfs, graphics,float,amssymb,amsmath,subeqnarray,setspace,graphicx,amsthm,epstopdf,subfigure, enumerate, color}
\usepackage[utf8]{inputenc}
\usepackage[colorlinks,
            linkcolor=red,
            anchorcolor=blue,
            citecolor=blue
            ]{hyperref}
\usepackage{natbib}

\usepackage{fullpage}
\numberwithin{equation}{section}

\newtheorem{lemma}{Lemma}
\newtheorem{theorem}{Theorem}
\newtheorem{proposition}{Proposition}
\newtheorem{definition}{Definition}
\newtheorem{corollary}[theorem]{Corollary}
\newtheorem{remark}{Remark}

\ifx\assumption\undefined
\newtheorem{assumption}{Assumption}
\fi

\newcommand{\bT}{\mathrm{T}}

\newcommand{\cO}{\mathcal{O}}
\newcommand{\cH}{\bm{\mathcal{H}}}
\newcommand{\cV}{\mathcal{V}}
\newcommand{\EE}{\mathbb{E}}
\newcommand{\RR}{\mathbb{R}}

\newcommand{\Ib}{\mathbf{I}}

\newcommand{\cF}{{\bm{\mathcal{F}}}}

\newcommand{\PP}{\mathbb{P}}

\newcommand{\minimize}{\mathop{\mathrm{minimize}}}

\usepackage{bbm}
\newcommand{\cI}{\mathcal{I}}
\newcommand{\bone}{\mathbbm{1}}
\newcommand{\tC}{\tilde{C}}

\def\cK{\mathcal{K}}
\newcommand{\cep}{\tilde{\epsilon}}

\newcommand{\tH}{\tilde{\bm{\mathcal{H}}}}

\newcommand{\sK}{\mathscr{K}}

\usepackage{multirow}
\usepackage{tablefootnote}
\usepackage{colortbl}
\usepackage{hhline}

\usepackage{algorithm}
\usepackage{algorithmic}

\def\FSP{first-order stationary point}
\def\SSP{second-order stationary point}
\def\ep{\epsilon}

\newcommand{\ww}{\mathbf{w}}

\newcommand{\vv}{\mathbf{v}}
\newcommand{\x}{\mathbf{x}}
\newcommand{\hx}{\hat{\x}}

\newcommand{\y}{\mathbf{y}}

\renewcommand{\b}{\mathbf{b}}

\newcommand{\Z}{\mathbf{Z}}
\newcommand{\A}{\mathbf{A}}
\newcommand{\cA}{\mathcal{A}}

\newcommand{\bxi}{\bm{\xi}}

\newcommand{\B}{\mathbf{B}}

\newcommand{\UU}{\mathbf{U}}

\newcommand{\uu}{\mathbf{u}}
\newcommand{\E}{\mathbb{E}}

\newcommand{\tx}{\tilde{\x}}

\newcommand{\hf}{\hat{f}}
\newcommand{\cP}{\bm{\mathcal{P}}}

\renewcommand{\a}{\mathbf{a}}

\newcommand{\e}{\mathbf{e}}

\renewcommand{\u}{\mathbf{u}}
\newcommand{\tf}{\tilde{f}}
\newcommand{\C}{\mathbf{C}}
\newcommand{\D}{\mathbf{D}}

\newcommand{\I}{\mathbf{I}}
\newcommand{\<}{\left\langle}
\renewcommand{\>}{\right\rangle}

\usepackage{mathtools}

\newcommand{\hB}{\mathcal{B}}
\newcommand{\hV}{\mathcal{V}}

\newcommand{\bcV}{\bm{\mathcal{ V}}}
\newcommand{\bcU}{\bm{\mathcal{U}}}
\newcommand{\bcG}{\bm{\mathcal{G}}}

\def\Ei{\EE}
\newcommand{\cS}{{\mathcal{S}}}
\def\NEON{\textsc{Neon}}
\def\SPIDER{\textsc{Spider}}

\begin{document}
\title{
\SPIDER:
Near-Optimal Non-Convex Optimization via Stochastic Path Integrated Differential Estimator
}

\author{
Cong Fang
\thanks{
Peking University;
email: fangcong@pku.edu.cn; zlin@pku.edu.cn
}
\thanks{This work was done while  Cong Fang was a Research Intern with Tencent AI Lab.}
\qquad
Chris Junchi Li
\thanks{
Tencent AI Lab;
email: junchi.li.duke@gmail.com; tongzhang@tongzhang-ml.org
}
\qquad
Zhouchen Lin
\footnotemark[1]
\qquad
Tong Zhang
\footnotemark[3]
}
\date{July 4, 2018~ (Initial)\\ \today~ (Current)}

\maketitle

\begin{abstract}
In this paper, we propose a new technique named \textit{Stochastic Path-Integrated Differential EstimatoR} (\SPIDER), which can be used to track many deterministic quantities of interest with significantly reduced computational cost. 
We apply \SPIDER\ to two tasks, namely the stochastic first-order and zeroth-order methods.
For stochastic first-order method, combining \SPIDER\ with normalized gradient descent, we propose two new algorithms, namely \SPIDER-SFO and \SPIDER-SFO\textsuperscript{+}, that solve non-convex stochastic optimization problems using stochastic gradients only. 
We provide sharp error-bound results on their convergence rates.
In special, we prove that the \SPIDER-SFO and \SPIDER-SFO\textsuperscript{+} algorithms achieve a  record-breaking gradient computation cost of $\mathcal{O}\left(  \min( n^{1/2} \epsilon^{-2}, \epsilon^{-3} ) \right)$ for finding an $\epsilon$-approximate first-order and  $\tilde{\mathcal{O}}\left(  \min( n^{1/2} \epsilon^{-2}+\epsilon^{-2.5}, \epsilon^{-3} ) \right)$  for  finding an $(\epsilon, \mathcal{O}(\ep^{0.5}))$-approximate second-order stationary point, respectively. 
In addition, we prove that \SPIDER-SFO nearly matches the algorithmic lower bound for finding approximate first-order stationary points under the gradient Lipschitz assumption in the finite-sum setting.
For stochastic zeroth-order method, we prove a cost of $\mathcal{O}( d \min( n^{1/2} \epsilon^{-2}, \epsilon^{-3}) )$ which outperforms all existing results.
\end{abstract}

\tableofcontents

\section{Introduction}

In this paper, we study the optimization problem
\beq\label{opt_eq0}
\minimize_{\x\in \RR^d} ~~~~
  f(\x)
\equiv
 \EE\left[ F(\x; \bm{\zeta}) \right]
\eeq
where the stochastic component $F(\x; \bm\zeta)$, indexed by some random vector $\bm\zeta$, is smooth and possibly \textit{non-convex}.
Non-convex optimization problem of form \eqref{opt_eq0} contains many large-scale statistical learning tasks.
Optimization methods that solve \eqref{opt_eq0} are gaining tremendous popularity due to their favorable computational and statistical efficiencies \citep{bottou2010large, bubeck2015convex, bottou2016optimization}.  
Typical examples of form \eqref{opt_eq0} include principal component analysis, estimation of graphical models, as well as training deep neural networks \citep{GOODFELLOW-BENGIO-COURVILLE}.
The expectation-minimization structure of stochastic optimization problem \eqref{opt_eq0} allows us to perform iterative updates and minimize the objective using its stochastic gradient $\nabla F(\x; \bm\zeta)$ as an estimator of its deterministic counterpart.

A special case of central interest is when the stochastic vector $\bm\zeta$ is finitely sampled.
In such \textit{finite-sum} (or \textit{offline}) case, we denote each component function as $f_i(x)$ and \eqref{opt_eq0} can be restated as
\beq\label{opt_eq}
\minimize_{\x\in \RR^d} ~~~~
  f(\x)
=
\frac{1}{n} \sum_{i=1}^n f_i(\x)
\eeq
where $n$ is the number of individual functions.
Another case is when $n$ is reasonably large or even infinite, running across of the whole dataset is exhaustive or impossible.
We refer it as the \textit{online} (or \textit{streaming}) case.
For simplicity of notations we will study the optimization problem of form \eqref{opt_eq} in both finite-sum and on-line cases till the rest of this paper.

One important task for non-convex optimization is to search for, given the precision accuracy $\ep > 0$, an \textit{$\ep$-approximate first-order stationary point} $\x \in \RR^d$ or $\| \nabla f(\x) \| \le \ep$.
In this paper, we aim to propose a new technique, called the \textit{Stochastic Path-Integrated Differential EstimatoR} (\SPIDER), which enables us to construct an estimator that tracks a deterministic quantity with significantly lower sampling costs.
As the readers will see, the \SPIDER\ technique further allows us to design an algorithm with a faster rate of convergence for non-convex problem \eqref{opt_eq}, in which we utilize the idea of \textit{Normalized Gradient Descent} (NGD) \citep{NESTEROV,hazan2015beyond}.
NGD is a variant of Gradient Descent (GD) where the stepsize is picked to be inverse-proportional to the norm of the full gradient.
Compared to GD, NGD exemplifies faster convergence, especially in the neighborhood of stationary points \citep{levy2016power}.
However, NGD has been less popular due to its requirement of accessing the full gradient and its norm at each update.
In this paper, we estimate and track the gradient and its norm via the \SPIDER\ technique and then hybrid it with NGD.
Measured by \textit{gradient cost} which is the total number of computation of stochastic gradients, our proposed \SPIDER-SFO algorithm achieves a faster rate of convergence in $\cO(\min(n^{1/2} \ep^{-2},  \ep^{-3} ) )$ which outperforms the previous best-known results in both finite-sum \citep{allen2016variance}\citep{reddi2016stochastic} and on-line cases \citep{lei2017non} by a factor of $\cO(\min(n^{1/6},  \ep^{-0.333} ) )$.

For the task of finding stationary points for which we already achieved a faster convergence rate via our proposed \SPIDER-SFO algorithm, a follow-up question to ask is:
\textit{is our proposed \SPIDER-SFO algorithm \textit{optimal} for an appropriate class of smooth functions?}
In this paper, we provide an \textit{affirmative} answer to this question in the finite-sum case.
To be specific, inspired by a counterexample proposed by \citet{carmon2017lower} we are able to prove that the gradient cost upper bound of \SPIDER-SFO algorithm matches the \textit{algorithmic lower bound}.
To put it differently, the gradient cost of \SPIDER-SFO \textit{cannot} be further improved for finding stationary points for some particular non-convex functions.

Nevertheless, it has been shown that for machine learning methods such as deep learning, approximate stationary points that have at least one negative Hessian direction, including saddle points and local maximizers, are often \textit{not} sufficient and need to be avoided or escaped from \citep{dauphin2014identifying,ge2015escaping}.
Specifically, under the smoothness condition for $f(\x)$ and an additional Hessian-Lipschitz condition for $\nabla^2 f(\x)$, we aim to find an \textit{$(\ep, O(\ep^{0.5}))$-approximate second-order stationary point} which is a point $\x\in\RR^d$ satisfying $\|\nabla f(\x)\| \le \ep$ and $\lambda_{\min}( \nabla^2 f(\x) ) \ge - \cO( \ep^{0.5})$ \citep{nesterov2006cubic}.
As a side result, we propose a variant of our \SPIDER-SFO algorithm, named \SPIDER-SFO\textsuperscript{+} (Algorithm \ref{algo:SPIDER-SFOplus}) for finding an approximate second-order stationary point, based a so-called \textit{Negative-Curvature-Search} method.
Under an additional Hessian-Lipschitz assumption, \SPIDER-SFO\textsuperscript{+} achieves an $(\epsilon,\cO(\ep^{0.5}))$-approximate second-order stationary point at a gradient cost of $\tilde\cO(\min(n^{1/2} \ep^{-2} +\ep^{-2.5},  \ep^{-3} ) )$.
In the on-line case, this indicates that our \SPIDER-SFO algorithm improves upon the best-known gradient cost in the on-line case by a factor of $\tilde{\cO}(\ep^{-0.25})$ \citep{allen2017neon2}. 
For the finite-sum case, the gradient cost of \SPIDER\ is sharper than that of the state-of-the-art \NEON+FastCubic/CDHS algorithm in \citet{agarwal2017finding,carmon2016accelerated} by a factor of $\tilde{\cO}(n^{1/4}\epsilon^{0.25})$ when $n \geq \epsilon^{-1}$.%
\footnote{%
In the finite-sum case, when $n \leq \epsilon^{-1}$ \SPIDER-SFO has a slower rate of $\tilde\cO(\ep^{-2.5})$ than the state-of-art $\tilde{\cO}(n^{3/4} \ep^{-1.75})$ rate achieved by \NEON+FastCubic/CDHS \citep{allen2017neon2}. 
\NEON+FastCubic/CDHS has exploited appropriate acceleration techniques, which has \textit{not} been considered for \SPIDER.
}

\subsection{Related Works}
In the recent years, there has been a surge of literatures in machine learning community that analyze the convergence property of non-convex optimization algorithms.
Limited by space and our knowledge, we have listed all literatures that we believe are mostly related to this work.
We refer the readers to the monograph by \citet{jain2017non} and the references therein on recent general and model-specific convergence rate results on non-convex optimization.

\paragraph{First- and Zeroth-Order Optimization and Variance Reduction}
For the general problem of finding approximate stationary points, under the smoothness condition of $f(\x)$, it is known that vanilla Gradient Descent (GD) and Stochastic Gradient Descent (SGD), which can be traced back to \citet{cauchy1847methode} and \citet{robbins1951stochastic} and achieve an $\ep$-approximate stationary point with a gradient cost of $\cO(\min(n \ep^{-2},  \ep^{-4} ) )$ \citep{NESTEROV,ghadimi2013stochastic,nesterov2011random,ghadimi2013stochastic,shamir2017optimal}.

Recently, the convergence rate of GD and SGD have been improved by the variance-reduction type of algorithms \citep{SVRG, SAG}.
In special, the finite-sum Stochastic Variance-Reduced Gradient (SVRG) and on-line Stochastically Controlled Stochastic Gradient (SCSG), to the gradient cost of $\tilde\cO(\min(n^{2/3} \ep^{-2},  \ep^{-3.333} ) )$ \citep{allen2016variance,reddi2016stochastic,lei2017non}.

\paragraph{First-order method for finding approximate stationary points}
Recently, many literature study the problem of how to avoid or escape saddle points and achieve an approximate second-order stationary point at a polynomial gradient cost
\citep{ge2015escaping,jin2017escape,xu2017first,allen2017neon2,hazan2015beyond,levy2016power,allen2017natasha2,reddi2018generic, tripuraneni2017stochastic, jin2017accelerated, lee2016gradient,agarwal2017finding, carmon2016accelerated, paquette2018catalyst}.
Among them, the group of authors \citet{ge2015escaping,jin2017escape} proposed the noise-perturbed variants of Gradient Descent (PGD) and Stochastic Gradient Descent (SGD) that escape from all saddle points and achieve an $\ep$-approximate second-order stationary point in gradient cost of $\tilde\cO(\min(n \ep^{-2}, poly(d)\ep^{-4}) )$ stochastic gradients.
\citet{levy2016power} proposed the noise-perturbed variant of NGD which yields faster evasion of saddle points than GD.

The breakthrough of gradient cost for finding second-order stationary points were achieved in 2016/2017, when the two recent lines of literatures, namely FastCubic \citep{agarwal2017finding} and CDHS \citep{carmon2016accelerated} as well as their stochastic versions \citep{allen2017natasha2,tripuraneni2017stochastic}, achieve a gradient cost of $\tilde\cO(\min( n\epsilon^{-1.5} + n^{3/4}\epsilon^{-1.75},    \ep^{-3.5}) )$ which serve as the best-known gradient cost for finding an $(\epsilon,\cO(\ep^{0.5}))$-approximate second-order stationary point before the initial submission of this paper.\footnote{%
\citet{allen2017natasha2} also obtains a gradient cost of $\tilde\cO(\epsilon^{-3.25})$ to achieve a (modified and weakened) $(\ep,\cO(\ep^{0.25}))$-approximate second-order stationary point.
}
\footnote{%
Here and in many places afterwards, the gradient cost also includes the number of stochastic Hessian-vector product accesses, which has similar running time with computing per-access stochastic gradient.
}
In particular, \citet{agarwal2017finding,tripuraneni2017stochastic} converted the cubic regularization method for finding second-order stationary points \citep{nesterov2006cubic} to stochastic-gradient based and stochastic-Hessian-vector-product-based methods, and \citet{carmon2016accelerated,allen2017natasha2} used a Negative-Curvature Search method to avoid saddle points.
See also recent works by \citet{reddi2018generic} for related saddle-point-escaping methods that achieve similar rates for finding an approximate second-order stationary point.

\paragraph{Online PCA and the NEON method}
In late 2017, two groups \citet{xu2017first,allen2017neon2} proposed a generic saddle-point-escaping method called \NEON, a Negative-Curvature-Search method using stochastic gradients.
Using such \NEON\ method, one can convert a series of optimization algorithms whose update rules use stochastic gradients and Hessian-vector products (GD, SVRG, FastCubic/CDHS, SGD, SCSG, Natasha2, etc.) to the ones using \textit{only} stochastic gradients without increasing the gradient cost.
The idea of \NEON\ was built upon Oja's iteration for principal component estimation \citep{oja1982simplified}, and its global convergence rate was proved to be near-optimal \citep{li2016near,jain2016streaming}.
\citet{allen2017first} later extended such analysis to the rank-$k$ case as well as the gap-free case, the latter of which serves as the pillar of the \NEON\ method.

\paragraph{Other concurrent works}
As the current work is carried out in its final phase, the authors became aware that an idea of resemblance was earlier presented in an algorithm named the \textit{StochAstic Recursive grAdient algoritHm} (SARAH) \citep{nguyen2017sarah,nguyen2017stochastic}. 
Both our \SPIDER-type of algorithms and theirs adopt the recursive stochastic gradient update framework.
Nevertheless, our techniques essentially differ from the works \citet{nguyen2017sarah,nguyen2017stochastic} in two aspects:

\begin{enumerate}[(i)]
\item
The version of SARAH proposed by \citet{nguyen2017sarah,nguyen2017stochastic} can be seen as a variant of gradient descent, while ours hybrids the \SPIDER\ technique with a stochastic version of NGD.

\item
\citet{nguyen2017sarah,nguyen2017stochastic} adopt a large stepsize setting (in fact their goal was to design a memory-saving variant of SAGA \citep{SAGA}), while our algorithms adopt a small stepsize that is proportional to $\ep$;
\end{enumerate}

Soon after the initial submission to NIPS and arXiv release of this paper, we became aware that similar convergence rate results for stochastic first-order method were also achieved independently by the so-called SNVRG algorithm \citep{zhou2018stochastic,zhou2018finding}.\footnote{%
To our best knowledge, the work by \citet{zhou2018stochastic,zhou2018finding} appeared on-line on June 20, 2018 and June 22, 2018, separately. SNVRG \citep{zhou2018stochastic} obtains a gradient complexity of $\tilde{\cO}(\min(n^{1/2} \ep^{-2}, \ep^{-3}))$ for finding an  approximate first-order stationary point, and  achieves    $\tilde{\cO}(\ep^{-3})$  gradient complexity  for finding an approximate second-order stationary point \citep{zhou2018finding} for a wide range of $\delta$. 
By exploiting the third-order smoothness condition, SNVRG can also achieve an $(\epsilon, \cO(\epsilon^{0.5}))$-approximate second-order stationary point in  $\tilde{\cO}(\ep^{-3})$ gradient costs.
}

\subsection{Our Contributions}
In this work, we propose the Stochastic Path-Integrated Differential Estimator (\SPIDER) technique, which significantly avoids excessive access of stochastic oracles and reduces the time complexity.
Such technique can be potential applied in many stochastic estimation problems.

\begin{enumerate}[(i)]
\item
As a first application of our \SPIDER\ technique, we propose the \SPIDER-SFO algorithm (Algorithm \ref{algo:SPIDER-SFO}) for finding an approximate first-order stationary point for non-convex stochastic optimization problem \eqref{opt_eq}, and prove the optimality of such rate in at least one case. 
Inspired by recent works \citet{SVRG, carmon2016accelerated, carmon2017lower} and independent of \citet{zhou2018stochastic,zhou2018finding}, this is the \textit{first} time that the gradient cost of $\cO(\min(n^{1/2} \ep^{-2},  \ep^{-3} ))$ in both upper and lower (finite-sum only) bound for finding first-order stationary points for problem \eqref{opt_eq} were obtained.

\item  Following   \citet{carmon2016accelerated, allen2017neon2,xu2017first}, we propose  \SPIDER-SFO\textsuperscript{+} algorithm (Algorithm \ref{algo:SPIDER-SFOplus})   for finding an approximate second-order stationary point for non-convex stochastic optimization problem. To best of our knowledge, this is also the \textit{first} time that  the gradient cost of $\tilde{\cO}(\min(n^{1/2}\epsilon^{-2}+\epsilon^{-2.5},\ep^{-3}))$ achieved with  standard assumptions. 

\item
As a second application of our \SPIDER\ technique, we apply it to zeroth-order optimization for problem (\ref{opt_eq}) and achieves  individual function accesses of $\cO(\min (dn^{1/2}\ep^{-2}, d\ep^{-3}))$.  To best of our knowledge, this is also the \textit{first} time that using Variance Reduction technique \citep{SAG, SVRG} to  reduce the individual function accesses for non-convex problems to the aforementioned complexity.

\item  We propose a much simpler analysis for proving convergence  to a stationary point.  One can flexibly  apply our proof techniques to analyze others algorithms, e.g. SGD, SVRG \citep{SVRG}, and SAGA \citep{SAGA}.

\end{enumerate}

\vspace{.1in}
\noindent\textbf{Organization.}
The rest of this paper is organized as follows.
\S\ref{sec:idea} presents the core idea of stochastic path-integrated differential estimator that can track certain quantities with much reduced computational costs.
\S\ref{sec:SFO} provides the \SPIDER\ method for stochastic first-order methods and convergence rate theorems of this paper for finding  approximate first-order stationary and second-order stationary points, and details a comparison with concurrent works.
\S\ref{sec:SZO} provides the \SPIDER\ method for stochastic zeroth-order methods and relevant convergence rate theorems.
\S\ref{sec:summary} concludes the paper with future directions.
All the detailed proofs are deferred to the appendix in their order of appearance.

\vspace{.1in}
\noindent\textbf{Notation.}
Throughout this paper, we treat the parameters $L,\Delta, \sigma,$ and $\rho$, to be specified later as global constants.
Let $\|\cdot\|$ denote the Euclidean norm of a vector or spectral norm of a square matrix.
Denote $p_n = \cO(q_n)$ for a sequence of vectors $p_n$ and positive scalars $q_n$ if there is a global constant $C$ such that $|p_n| \le Cq_n$, and $p_n = \tilde{\cO}(q_n)$ such $C$ hides a poly-logarithmic factor of the parameters. Denote $p_n = \Omega (q_n)$ if there is a global constant $C$ such that $| p_n|\geq C q_n$.
Let $\lambda_{\min}(\A)$ denote the least eigenvalue of a real symmetric matrix $\A$.
For fixed $K \ge k \ge 0$, let $\x_{k:K}$ denote the sequence $\{\x^k ,\dots , \x^K \}$.
Let $[n] = \{1,\dots,n\}$ and $S$ denote the cardinality of a multi-set $\cS \subset [n]$ of samples (a generic set that allows elements of multiple instances).
For simplicity, we further denote the averaged sub-sampled stochastic estimator $\hB_{S} := (1/S) \sum_{i\in \cS} \hB_i$ and averaged sub-sampled gradient $\nabla f_S := (1 / S)\sum_{i\in \cS} \nabla f_i$.
Other notations are explained at their first appearance.

\section{Stochastic Path-Integrated Differential Estimator: Core Idea}\label{sec:idea}
In this section, we present in detail the underlying idea of our Stochastic Path-Integrated Differential Estimator (\SPIDER) technique behind the algorithm design.
As the readers will see, such technique significantly avoids excessive access of the stochastic oracle and reduces the complexity, which is of independent interest and has potential applications in many stochastic estimation problems. 

Let us consider an arbitrary deterministic vector quantity $Q(\x)$.
Assume that we observe a sequence $\hx_{0:K}$, and we want to dynamically track $Q(\hx^k )$ for $k = 0, 1, \dots, K.$
Assume further that we have an initial estimate $\tilde{Q}(\hx^0 ) \approx Q(\hx^0 )$, and an unbiased estimate $\bxi _k ( \hx_{0:k})$ of $Q(\hx^k ) - Q( \hx^{ k-1} )$ such that for each $k = 1,\dots, K$
\[
\EE\left[ \bxi _k ( \hx_{0:k})  \mid \hx_{0:k}  \right] = Q(\hx^k ) - Q(\hx^{k-1} )
.
\]
Then we can integrate (in the discrete sense) the stochastic differential estimate as
\beq\label{tildeQ}
\tilde{Q}(\hx_{0:K} ) := \tilde{Q}(\hx^0 ) + \sum_{k=1}^K \bxi _k (\hx_{0:k} )
.
\eeq
We call estimator $\tilde{Q}(\hx_{0:K})$ the \textit{Stochastic Path-Integrated Differential EstimatoR}, or \SPIDER\ for brevity.
We conclude the following proposition which bounds the error of our estimator $\|\tilde{Q}(\hx_{0:K} )- Q(\hx^K )\| $, in terms of both expectation and high probability:

\begin{proposition}\label{prop:aggregate}
We have 
\begin{enumerate}[(i)]
\item
The martingale variance bound has
\beq\label{mart_var}
\EE \|\tilde{Q}(\hx_{0:K} ) - Q(\hx^K )\| ^2 = 
\EE \|\tilde{Q}(\hx^0 ) - Q(\hx^ 0 )\| ^2 
+
\sum_{k=1}^K \EE \| \bxi_ k (\hx_{0:k} ) - (Q(\hx^k ) - Q(\hx^{k-1} )) \| ^2
 .
\eeq
\item
Suppose
\beq\label{Q0bdd}
\|\tilde{Q}(\hx^0 ) - Q(\hx^ 0 )\| \le b_0
\eeq
and for each $k=1,\dots,K$
\beq\label{Qtbdd}
 \| \bxi_ k (\hx_{0:k} ) - (Q(\hx^k ) - Q(\hx^{k-1} )) \| \le b_k, 
\eeq
Then  for any $\gamma > 0$ and  a given $k\in\{1,\dots,K\}$ we have with probability at least $1-4\gamma$
\beq\label{Qconcentrate}
\left\| \tilde{Q}(\hx_{0:k} )- Q(\hx^k ) \right\|
 \le
2 \sqrt{\sum_{s=0}^k b_s^2 \cdot \log \frac{1}{\gamma}}
.
\eeq
\end{enumerate}

\end{proposition}

Proposition \ref{prop:aggregate}(i) can be easily concluded using the property of square-integrable martingales.
To prove the high-probability bound in Proposition \ref{prop:aggregate}(ii), we need to apply an Azuma-Hoeffding-type concentration inequality \citep{pinelis1994optimum}.
See \S\ref{sec:pinelis} in the Appendix for more details.

Now, let $\hB$ map any $\x\in \RR^d$ to a random estimate $\hB_i(\x)$ such that, conditioning on the observed sequence $\x_{0:k}$, we have for each $k = 1,\dots, K$,
\beq\label{hAidist2}
\EE\left[
\hB_i(\x^{k})  - \hB_i(\x^{k-1})
\mid
\x_{0:k}
\right]
=
\hV^{k} -  \hV^{k-1}
.
\eeq
At each step $k$ let $S_*$ be a subset that samples $\cS_*$ elements in $[n]$ with replacement, and let the stochastic estimator $\hB_{S_*} = (1/\cS_*) \sum_{i\in S_*} \hB_i$ satisfy  
\beq\label{hAidist}
\Ei \|\hB_i (\x) -\hB_i (\y)\|^2 \leq L_\hB^2 \|\x -\y  \|^2
,
\eeq
and $\| \x^{k}-\x^{k-1} \| \leq  \epsilon_1$ for all $k = 1,\dots,K$.
Finally, we set our estimator $\hV^{k}$ of $\hB(\x^{k})$ as
\[
\hV^{k} = \hB_{S_*}(\x^{k})  - \hB_{S_*}(\x^{k-1}) + \hV^{k-1}
.
\]
Applying Proposition \ref{prop:aggregate} immediately concludes the following lemma,
which gives an error bound of the estimator $\hV^{k}$ in terms of the second moment of $\| \hV^{k} - \hB(\x^k)  \|$:

\begin{lemma}\label{lemm:aggregate}
We have under the condition \eqref{hAidist} that for all $k = 1,\dots, K$,
\begin{eqnarray}
\E\| \hV^{k} - \hB(\x^k) \|^2
 \leq  
\frac{k L_\hB^2 \epsilon_1^2}{\cS_*} + \EE \|\hV^0 - \hB(\x^0)  \|^2
.
\end{eqnarray}
\end{lemma}

It turns out that one can use \SPIDER\ to track many quantities of interest, such as stochastic gradient, function values, zero-order estimate gradient, functionals of Hessian matrices, etc.
Our proposed \SPIDER-based algorithms in this paper take $\hB_i$ as the stochastic gradient $\nabla f_i$ and the zeroth-order estimate gradient, separately.

\section{SPIDER for Stochastic First-Order Method}\label{sec:SFO}
In this section, we apply \SPIDER\ to the task of finding both first-order and second-order stationary points for non-convex stochastic optimization.
The main advantage of \SPIDER-SFO lies in using SPIDER to estimate the gradient with a low computation cots.
We introduce the basic settings and assumptions in \S\ref{ssec:SFO_assu} and propose the main error-bound theorems for finding approximate first-order  and second-order stationary points, separately in  \S\ref{ssec:SFO_upper} and \S\ref{ssec:SFO_SSP}.

\subsection{Settings and Assumptions}\label{ssec:SFO_assu}
We first introduce the formal definition of approximate first-order and second-order stationary points, as follows.
\begin{definition}
We call $\x\in \RR^d$ an \textit{$\epsilon$-approximate first-order stationary point}, or simply an \textit{FSP}, if
\begin{eqnarray}\label{SSP}
\|\nabla f(\x)\|\leq \epsilon.
\end{eqnarray}
Also, call $\x$ an \textit{$(\epsilon, \delta)$-approximate second-order stationary point}, or simply an \textit{SSP}, if
\begin{eqnarray}
\|\nabla f(\x)\|\leq \epsilon,
  \quad \quad   
\lambda_{\min}\left(  \nabla^2 f(\x)   \right)    \ge -\delta
.
\end{eqnarray}
\end{definition}

The definition of an $(\epsilon, \delta)$-approximate second-order stationary point generalizes the classical version where $\delta = \sqrt{\rho\ep}$, see e.g.~\citet{nesterov2006cubic}.
For our purpose of analysis, we also pose the following additional assumption:

\begin{assumption}\label{assu:main}
We assume the following

\begin{enumerate}[(i)]
	\item \label{a1} 
The $\Delta := f(\x^0)- f^* < \infty$ where  $f^*  =  \inf_{\x\in\RR^d} f(\x)$ is the global infimum value of $f(\x)$;
	
	\item \label{a2} 
The component function $f_i(\x)$ has an averaged $L$-Lipschitz gradient, i.e.~for all $\x,\y$,
\[
\Ei \| \nabla f_i(\x) -\nabla f_i(\y) \|^2 \leq L^2 \| \x-\y\|^2
;
\]
	
	\item \label{a3} 
(For on-line case only)
the stochastic gradient has a finite variance bounded by $\sigma^2 < \infty$, i.e.
\[
\Ei \left\|\nabla f_{i}(\x) -\nabla f(\x) \right\|^2\leq \sigma^2
.   
\]
		
\end{enumerate}
\end{assumption}

Alternatively, to obtain high-probability results using concentration inequalities, we propose the following more stringent assumptions:
\begin{assumption}\label{assu:main2}
We assume that Assumption \ref{assu:main} holds and, in addition,
\begin{enumerate}[(i)]
	\item[(ii')] \label{a4} (Optional)  
each component function $f_i(\x)$ has $L$-Lipschitz continuous gradient, i.e.~for all $i, \x, \y$,
\[
\left\|\nabla f_i(\x)-\nabla f_i(\y)\right\|\leq L\|\x-\y\|
.
\]
Note when $f$ is twice continuously differentiable, Assumption \ref{assu:main} (ii) is equivalent to  $ \Ei \|\nabla^2  f_i(\x)  \|^2\leq L^2 $ for all $\x$ and is  weaker than the additional Assumption \ref{assu:main2} (ii'), since the absolute norm squared bounds the variance for any random vector.
		\item[(iii')]\label{a5}
		(For on-line case only) the gradient of each component function $f_i(\x)$  has finite bounded  variance by $\sigma^2 < \infty$ (with probability $1$) , i.e. for all $i, \x$, 
		$$  \|\nabla f_i(\x) - \nabla f(\x)\|^2\leq \sigma^2   .$$
Assumption \ref{assu:main2} is  common in applying concentration laws to obtain high probability result\footnote{In this paper, we use Azuma-Hoeffding-type concentration inequality to obtain high probability results like \cite{xu2017first,allen2017neon2}. By applying Bernstein inequality, under the Assumption \ref{assu:main}, the parameters in the  Assumption \ref{assu:main2} are allowed to be $\tilde{\Omega} (\epsilon^{-1})$ larger  without hurting the convergence rate.}.
\end{enumerate}
\end{assumption}

For the problem of finding an $(\epsilon,\delta)$-approximate \SSP, we pose in addition to Assumption \ref{assu:main} the following assumption:

\begin{assumption}\label{assu:SSP}
We assume that Assumption \ref{assu:main2} (including (ii')) holds and, in addition, each component function $f_i(\x)$ has  $\rho$-Lipschitz continuous Hessian, i.e.~for all $i, \x, \y$,
\[
\|\nabla^2 f_i(\x)-\nabla^2 f_i(\y) \|\leq \rho \|\x-\y\|
.
\]
\end{assumption}

We emphasize that Assumptions \ref{assu:main}, \ref{assu:main2}, and \ref{assu:SSP} are standard for non-convex stochastic optimization \citep{agarwal2017finding, carmon2017lower, jin2017escape, xu2017first, allen2017neon2}.

\subsection{First-Order Stationary Point}\label{ssec:SFO_upper}

\begin{algorithm}[tb]
\caption{\SPIDER-SFO: Input $\x^0$, $q$, $S_1$, $S_2$, $n_0$, $\epsilon$, and $\cep$ (For finding first-order stationary point)}
\label{algo:SPIDER-SFO}
\begin{algorithmic}[1]
\FOR {$k=0$ to $K$}
\IF{$\mod(k,q)=0$}
		\STATE Draw $S_1$ samples (or compute the full gradient for the finite-sum case), let $\vv^k = \nabla f_{\cS_1}(\x^k)$
 \label{line3}
\ELSE
		\STATE Draw $S_2$ samples, and let $\vv^k = \nabla f_{\cS_2}(\x^k) - \nabla f_{\cS_2}(\x^{k-1}) + \vv^{k-1}$ 
 \label{line5}
\ENDIF
\newline
\STATE  \textbf{\hspace{-.2in} OPTION $\mbox{I}$}
  				  {\hfill $\diamond$ for convergence rates in high probability}
\IF{ $\|\vv^k\| \le 2\cep$}
\RETURN $\x^k$
\ELSE
	\STATE   $\x^{k+1} = \x^k - \eta \cdot (\vv^k / \|\vv^k\| ) $ \text{ where}$\quad \eta = \frac{\epsilon}{L n_0}$
	\label{line:update}
\ENDIF
\newline
\STATE  \textbf{\hspace{-.2in} OPTION $\mbox{II}$}
  				  {\hfill $\diamond$ for convergence rates in expectation}
\STATE   $\x^{k+1} = \x^k - \eta^k \vv^k$ \text{ where}$\quad \eta^k = \min \left( \frac{\epsilon}{Ln_0\|\vv^k\|}, \frac{1}{2L n_0} \right)$ 	\label{line:update2}
\newline
\ENDFOR
\newline
\STATE   \textbf{OPTION $\mbox{I}$}: Return $\x^K$
  				  {\hfill $\diamond$ however, this line is \textit{not} reached with high probability}
\newline
\STATE   \textbf{OPTION $\mbox{II}$}: Return $\tilde{\x}$ chosen uniformly at random from $\{\x^k \}_{k=0}^{K-1}$
\end{algorithmic}
\end{algorithm}

Recall that NGD has iteration update rule
\beq\label{NGD}
\x^{k+1} 
= 
\x^{k} -\eta \frac{\nabla f (\x^{k})}{\| \nabla f (\x^{k}) \|}
,
\eeq
where $\eta$ is a constant step size.
The NGD update rule \eqref{NGD} ensures $\| \x^{k+1} - \x^k\|$ being constantly equal to the stepsize $\eta$, and might  fastly escape from saddle points and converge to a second-order stationary point \citep{levy2016power}. We propose \SPIDER-SFO in Algorithm \ref{algo:SPIDER-SFO}, which is like a stochastic variant of NGD with the \SPIDER\ technique applied, so as to maintain an estimator in each epoch $\nabla f(\x^k)$ at a higher accuracy under limited gradient budgets.

To analyze the convergence rate of \SPIDER-SFO, let us first consider the on-line case for Algorithm \ref{algo:SPIDER-SFO}.
We let the input parameters be
\beq\label{inputOne}
S_1 = \frac{2\sigma^2}{\epsilon^2}
,\qquad
S_2 = \frac{2\sigma}{\epsilon n_0}
,\qquad
\eta =  \frac{\epsilon}{Ln_0}
,\qquad
\eta^k =  \min \left( \frac{\epsilon}{Ln_0\|\vv^k\|}, \frac{1}{2L n_0} \right)
,\qquad
q = \frac{ \sigma n_0}{\epsilon}
,
\eeq
where $n_0 \in [1,2\sigma/\epsilon]$ is a free parameter to choose.%
\footnote{%
When $n_0 = 1$,  the mini-batch size is $2\sigma / \epsilon$, which is the largest mini-batch size that Algorithm \ref{algo:SPIDER-SFO} allows to choose. 
}
In this case, $\vv^k$ in Line \ref{line5} of Algorithm \ref{algo:SPIDER-SFO} is a \SPIDER\ for $\nabla  f(\x^k)$. 
To see this, recall $\nabla f_i (\x^{k-1})$ is the stochastic gradient drawn at step $k$ and
\begin{eqnarray}
\Ei \left[
\nabla f_i(\x^k) - \nabla f_i(\x^{k-1})
 \mid
\x_{0:k}
\right]
  = 
  \nabla f(\x^k) - \nabla f(\x^{k-1})
  .
\end{eqnarray}
Plugging in $\cV^k = \vv^k$ and $\hB_i = \nabla f_i$ in Lemma \ref{lemm:aggregate} of \S\ref{sec:idea}, we can use $\vv^k$ in Algorithm \ref{algo:SPIDER-SFO} as the \SPIDER\ and conclude the following lemma that is pivotal to our analysis.
\begin{lemma}\label{111}
Set the parameters $S_1$, $S_2$, $\eta$, and $q$ as in \eqref{inputOne}, and $k_0 = \lfloor k/q \rfloor \cdot q $.
Then under the Assumption \ref{assu:main}, we have
\[
\E \left[ \| \vv^k -\nabla f(\x^k)  \|^2 \mid \x_{0:k_0} \right]
 \leq  \epsilon^2 
.
\]
Here we compute the conditional expectation over the randomness of $x_{(k_0+1): k}$. 
\end{lemma}
Lemma \ref{111} shows that our \SPIDER\ $\vv^k$ of $\nabla f(\x)$ maintains an error of $\cO(\epsilon)$.
Using this lemma, we are ready to present the following results for Stochastic First-Order (SFO) method for finding first-order stationary points of \eqref{opt_eq}.

\paragraph{Upper Bound for Finding First-Order Stationary Points, in Expectation}
\begin{theorem}[First-Order Stationary Point, on-line setting, expectation]\label{theo:FSOone}
For the on-line case, set the parameters $S_1$, $S_2$, $\eta$, and $q$ as in \eqref{inputOne}, and  $K = \left\lfloor(4L \Delta n_0)\epsilon^{-2}\right\rfloor+1$. 
Then under the Assumption \ref{assu:main}, for  Algorithm \ref{algo:SPIDER-SFO} with OPTION $\uppercase\expandafter{\romannumeral1}$,  after $K$ iteration, we have
\begin{eqnarray}
\E \left[\|\nabla f(\tx) \|\right]\leq 5\epsilon.
\end{eqnarray}
The  gradient cost is bounded by $ 	16 L\Delta \sigma \cdot \epsilon^{-3} +2\sigma^2 \epsilon^{-2} + 4\sigma n_0^{-1} \epsilon^{-1} $ for any choice of $n_0 \in [1,2\sigma/\epsilon]$.  Treating $\Delta$, $L$ and $\sigma$ as positive constants, the stochastic gradient complexity is $\cO(\epsilon^{-3})$.

\end{theorem}

The relatively reduced minibatch size serves as the key ingredient for the superior performance of \SPIDER-SFO. 
For illustrations, let us compare the sampling efficiency among SGD, SCSG and \SPIDER-SFO\ in their special cases.
With some involved analysis of these algorithms, we can conclude that to ensure a sufficient function value decrease of $\Omega(\epsilon^2/L)$ at each iteration,

\begin{enumerate}[(i)] 
\item
for SGD the choice of mini-batch size is $\cO\big( \sigma^2 \cdot  \epsilon^{-2}\big)$; 

\item
for SCSG \citep{lei2017non} and Natasha2 \citep{allen2017natasha2} the mini-batch size is $\cO\big( \sigma \cdot  \epsilon^{-1.333} \big)$;

\item
for our \SPIDER-SFO only needs a reduced mini-batch size of $\cO\big( \sigma \cdot \epsilon^{-1} \big)$
\end{enumerate}

Turning to the finite-sum case, analogous to the on-line case we let
\begin{eqnarray}\label{inputTwo}
S_2 = \frac{n^{1/2}}{n_0}
,\qquad
\eta = \frac{\epsilon}{Ln_0}
,\qquad
\eta^k =  \min \left( \frac{\epsilon}{Ln_0\|\vv^k\|}, \frac{1}{2L n_0} \right)
,\qquad
q = n_0n^{1/2}
,
\end{eqnarray}
where $n_0 \in [1,n^{1/2}]$.  
In this case, one computes the full gradient $\vv^k = \nabla f_{S_1}(\x^k)$ in Line \ref{line3} of Algorithm \ref{algo:SPIDER-SFO}.
We conclude our second upper-bound result:

\begin{theorem}[First-Order Stationary Point, finite-sum setting]\label{theo:FSOtwo}
In the finite-sum case, set the parameters $\cS_2$, $\eta$, and  $q$ as in \eqref{inputTwo}, $K = \left\lfloor (4L \Delta n_0)\epsilon^{-2} \right\rfloor+1$ and let $S_1 = [n]$, i.e.~we obtain the full gradient in Line \ref{line3}.
The  gradient cost is bounded by $n+	8(L\Delta) \cdot n^{1/2} \epsilon^{-2}+ 2n_0^{-1} n^{1/2} $ for any choice of $n_0 \in [1,n^{1/2}]$. 
Treating $\Delta$, $L$ and $\sigma$ as positive constants, the stochastic gradient complexity is $\cO(n+ n^{1/2}\epsilon^{-2})$. 
\end{theorem}

\paragraph{Lower Bound for Finding First-Order Stationary Points}
To conclude the optimality of our algorithm we need an algorithmic lower bound result \citep{carmon2017lower,woodworth2016tight}.
Consider the finite-sum case and any  random algorithm $\cA$ that maps functions $f: \RR^d \to \RR$ to a sequence of iterates in $\RR^{d+1}$, with 
\begin{eqnarray}\label{algor-need}
[\x^k; i_k]
 = 
\cA^{k-1} \big(\bxi, \nabla f_{i_0}(\x^0), \nabla f_{i_1}(\x^1), \ldots, \nabla f_{i_{k-1}}(\x^{k-1})   \big)
,\quad k\geq 1,
\end{eqnarray}
where $\cA^{k}$  are measure mapping into $\RR^{d+1}$, $i_k$ is the individual function chosen by $\cA$ at iteration $k$, and $\bxi$ is uniform random vector from $[0,1]$.  And  $[\x^0; i_0] = \cA^0(\bxi)$,  where  $\cA^0$ is a  measure mapping.
The lower-bound result for solving \eqref{opt_eq} is stated as follows:

\begin{theorem}[Lower bound for SFO for the finite-sum setting]
\label{theo:lowerbdd}
	For any $L>0$, $\Delta >0$, and \blue{$2\leq n \leq 	O\left( \Delta^2 L^2\cdot \epsilon^{-4} \right)$}, for any algorithm $\cA$ satisfying \eqref{algor-need}, there exists a dimension 
$
d = \tilde\cO\big(
 \Delta^2 L^2 \cdot n^2\epsilon^{-4}
\big)
, 
$
and a function $f$ satisfies Assumption \ref{assu:main} in the finite-sum case,  such that in order to find a point $\tx$ for which \blue{$\|\nabla f(\tx)\|\leq \epsilon$}, $\cA$ must cost at least $
\Omega\big(L \Delta \cdot n^{1/2}\epsilon^{-2}\big)
$ stochastic gradient accesses.
\end{theorem}

Note the condition $n\leq \cO(\epsilon^{-4})$ in Theorem \ref{theo:lowerbdd} ensures that our lower bound $\Omega(n^{1/2}\epsilon^{-2}) = \Omega(n+n^{1/2}\epsilon^{-2})$, and hence our upper bound in Theorem \ref{theo:FSOone} matches the lower bound in Theorem \ref{theo:lowerbdd} up to a constant factor of relevant parameters, and is hence \textit{near-optimal}.
Inspired by \citet{carmon2017lower}, our proof of Theorem \ref{theo:lowerbdd} utilizes a specific counterexample function that requires at least $\Omega(n^{1/2}\epsilon^{-2})$ stochastic gradient accesses.
Note \citet{carmon2017lower}  analyzed such counterexample in the deterministic case $n=1$ and we generalize such analysis to the finite-sum case $n\ge 1$.

\begin{remark}
Note by setting $n = \cO(\epsilon^{-4})$ the lower bound complexity in Theorem \ref{theo:lowerbdd} can be as large as $\Omega(\epsilon^{-4})$.
We emphasize that this does \textit{not} violate the $ \cO(\epsilon^{-3})$ upper bound in the on-line case [Theorem \ref{theo:FSOone}], since the counterexample established in the lower bound depends \textit{not} on  the stochastic gradient variance  $\sigma^2$ specified in Assumption \ref{assu:main}(\ref{a3}), but on  the component number $n$.
To obtain the lower bound result for the on-line case with the additional Assumption \ref{assu:main}(\ref{a3}), with more efforts one might be able to construct a second counterexample that requires $\Omega(\ep^{-3})$ stochastic gradient accesses with the knowledge of $\sigma$ instead of $n$. 
We leave this as a future work. 
\end{remark}

\paragraph{Upper Bound for Finding First-Order Stationary Points, in High-Probability}
We consider obtaining high-probability results.
With Theorem \ref{theo:FSOone} and Theorem \ref{theo:FSOtwo} in hand, by Markov Inequality, we have  $ \| \nabla f(\tx) \|\leq 15\epsilon$  with probability $\frac{2}{3}$.   Thus a straightforward way to obtain a high probability result is by  adding an additional verification step in the end  of Algorithm \ref{algo:SPIDER-SFO},  in which we  check  whether $\tx$ satisfies $ \| \nabla f(\tx) \|\leq 15\epsilon$ (for the on-line case  when $\nabla f(\tx)$ are  unaccessible, under Assumption \ref{assu:main2}~(iii'), we can draw $\tilde{O}(\epsilon^{-2})$ samples to estimate $\| \nabla f(\tx) \|$ in high accuracy). If not, we can  restart Algorithm \ref{algo:SPIDER-SFO} (at most  in $O(\log (1/p))$ times) until it find a desired solution. However, because the above way needs  running Algorithm \ref{algo:SPIDER-SFO} in multiple times, in the following, we show  with   Assumption \ref{assu:main2} (including (\ref{a4})),  original Algorithm \ref{algo:SPIDER-SFO} obtains a  solution with an additional polylogarithmic factor under high probability. 
\begin{theorem}[First-Order Stationary Point, on-line setting, high probability]\label{theo:FSOonehp}
	For the on-line case, set the parameters $S_1$, $S_2$, $\eta$ and $q$ in \eqref{inputOne}.  Set $\cep = 10\epsilon \log\left( \left(4\lfloor 4L \Delta n_0\epsilon^{-2}\rfloor+12\right)p^{-1}  \right) \sim \tilde{\cO}(\epsilon)$.	Then  under the Assumption \ref{assu:main2} (including (ii')), with  probability at least $1-p$, Algorithm \ref{algo:SPIDER-SFO} terminates before $K_0 = 	\lfloor  (4L\Delta n_0)\epsilon^{-2}\rfloor +2$ iterations and outputs an $\x^{\cK}$ satisfying
	\begin{eqnarray}\label{FSPonline}
	\|\vv^{\cK}\|\leq 2\cep
	\quad\text{and}\quad
	\| \nabla f(\x^{\cK})   \|\leq 3\cep.
	\end{eqnarray}
	The gradient costs to find a FSP satisfying \eqref{FSPonline} with probability $1-p$ are bounded by   $16 L\Delta \sigma \cdot \epsilon^{-3} +2\sigma^2\epsilon^{-2}+ 8\sigma n_0^{-1} \epsilon^{-1}  $ for any choice of of $n_0 \in [1,2\sigma/\epsilon]$.  Treating $\Delta$, $L$ and $\sigma$ as constants, the stochastic gradient complexity is $\tilde{\cO}(\epsilon^{-3})$.
\end{theorem}
\begin{theorem}[First-Order Stationary Point, finite-sum setting]\label{theo:FSOtwohp}
	In the finite-sum case, set the parameters $S_1$, $S_2$, $\eta$, and $q$ as \eqref{inputTwo}.    let $S_1 = [n]$, i.e.~we obtain the full gradient in Line \ref{line3}. 	Then under the Assumption \ref{assu:main2} (including (ii')), 	with  probability at least $1-p$, Algorithm \ref{algo:SPIDER-SFO} terminates before $K_0 = 	\lfloor 4L\Delta n_0/\epsilon^2\rfloor +2$ iterations and outputs an $\x^{\cK}$ satisfying
	\begin{eqnarray}\label{FSPoffline}
	\|\vv^{\cK}\|\leq 2\cep
	\quad\text{and}\quad
	\| \nabla f(\x^{\cK})   \|\leq 3\cep
	.
	\end{eqnarray}
	where $\cep = 16\epsilon \log\left(\left(4(L\Delta n_0 \epsilon^{-2} +12\right)p^{-1}\right)=\tilde{\cO}(\epsilon)$. So the gradient costs to find a FSP  satisfying \eqref{FSPoffline} with probability $1-p$  are bounded by   $n+	8L\Delta  n^{1/2} \epsilon^{-2}+(2n_0^{-1})n^{1/2}+ 4 n_0^{-1} n^{1/2}$ with any choice of  $n_0 \in [1,n^{1/2}]$.  Treating $\Delta$, $L$ and $\sigma$ as constants, the stochastic gradient complexity is $\tilde{\cO}(n+ n^{1/2}\epsilon^{-2})$.
	
\end{theorem}

\subsection{Second-Order Stationary Point}\label{ssec:SFO_SSP}

To find a second-order stationary point with \eqref{SSP},  we can fuse our \SPIDER-SFO in Algorithm \ref{algo:SPIDER-SFO} with a Negative-Curvature-Search (NC-Search) iteration that solves the following task: 
given a point $\x\in\RR^d$, decide if $\lambda_{\min}(\nabla^2 f(\x)) \geq -\delta $ or find a unit vector $\ww_1$ such that $\ww_1^\top \nabla^2 f(\x) \ww_1 \leq -\delta / 2$ (for numerical reasons, one has to leave some room between the two bounds). 
For the on-line case,  NC-Search can be efficiently solved by Oja's algorithm \citep{oja1982simplified, allen2017natasha2} and also by \NEON\ \citep{allen2017neon2,xu2017first} with the gradient cost of $\tilde{\cO}(\delta^{-2})$.%
\footnote{%
Recall that the NEgative-curvature-Originated-from-Noise method (or \NEON\ method for short) proposed independently by \citet{allen2017neon2,xu2017first} is a generic procedure that convert an algorithm that finds an  approximate first-order stationary points to the one that finds an  approximate second-order stationary point.
}
When $\ww_1$ is found, one can set $\ww_2 = \pm (\delta/\rho) \ww_1$ where $\pm$ is a random sign.
Then under Assumption \ref{assu:SSP}, Taylor's expansion implies that \citep{allen2017neon2}
\begin{eqnarray}\label{2de}
f(\x + \ww_2)
 \leq 
f(\x) 
+ 
[ \nabla f(\x)]^\top \ww_2  
+
\frac12 \ww_2^\top [\nabla^2 f(\x)] \ww_2
+ 
\frac{\rho}{6} \| \ww_2 \|^3
.
\end{eqnarray}
Taking expectation, one has $
\E f(\x + \ww_2)  \leq  f(\x) - \delta^3 / (2\rho^2) + \delta^3 / (6\rho^2) =  f(\x) -  \delta^3 / (3\rho^2)
.
$
This indicates that when we find a direction of negative curvature or Hessian, updating $\x \leftarrow \x + \ww_2$ decreases the function value by $\Omega(\delta^3)$ in expectation.  
Our \SPIDER-SFO algorithm fused with NC-Search is described in the following steps:

\begin{mdframed}[style=exampledefault]
	\begin{enumerate}[Step 1.]
		\item
		Run an efficient NC-Search iteration to find an $\cO(\delta)$-approximate negative Hessian direction $\ww_1$ using stochastic gradients, e.g.~\NEON2 \citep{allen2017neon2}.
		\item
	     If NC-Search find a $\ww_1$,  update $\x \leftarrow \x \pm (\delta / \rho) \ww_1$ in $\delta / (\rho\eta)$ mini-steps, and simultaneously use \SPIDER\ $\vv^k$ to maintain an estimate of $\nabla f(\x)$. Then Goto Step 1.
		\item
         If not, run \SPIDER-SFO  for $\delta/(\rho \eta)$ steps directly using the  \SPIDER ~$\vv^k$  (without restart) in Step 2. Then Goto Step 1.
         \item  During Step 3, if  we find $\| \vv^k\|\leq 2\cep$, return $\x^k$.
	\end{enumerate}
\end{mdframed}

\begin{algorithm}[tb!]
	\caption{\SPIDER-SFO\textsuperscript{+}: Input $\x^0$,  $S_1$, $S_2$, $n_0$, $q$, $\eta$,  $\sK$, $k=0$, $\epsilon$, $\cep$, 
		(For finding a second-order stationary point)}
	\label{algo:SPIDER-SFOplus}
	
	\begin{algorithmic}[1]
		\FOR{$j=0$ $\mathbf{to}$ $J$}
		\STATE  Run an efficient NC-search iteration, e.g.~\NEON2$(f,\x^k, 2\delta, \frac{1}{16J})$  and obtain $\ww_1$\label{sfo5}
		\IF{$\ww_1\neq\bot$}\STATE $\diamond$ Second-Order Descent:	
		\STATE    Randomly flip a sign, and set $\ww_2 =  \pm  \eta \ww_1$ and $
		j =	\delta / (\rho \eta)  -   1
		$ \label{sfo8}
		\FOR{$k$ $\mathbf{to}$ $k+ \sK$}
		\IF{$\mbox{mod}(k,q)=0$}\label{sfo1}
		\STATE Draw $\cS_1$ samples, $\vv^k = \nabla f_{S_1}(\x^k)$
		\ELSE	
		\STATE Draw $\cS_2$ samples, $\vv^k = \nabla f_{S_2}(\x^k) - \nabla f_{S_2}(\x^{k-1}) + \vv^{k-1}$
		\ENDIF\label{sfo2}
		\STATE $\x^{k+1} = \x^k - \ww_2$\label{sfo6}
		\ENDFOR	
		\ELSE 
		\STATE $\diamond$ First-Order Descent:
				\FOR{$k$ $\mathbf{to}$ $k+\sK$}
				\IF{$\mbox{mod}(k,q)=0$}\label{sfo3}
				\STATE Draw $\cS_1$ samples, $\vv^k = \nabla f_{S_1}(\x^k)$
				\ELSE	
				\STATE Draw $\cS_2$ samples, $\vv^k = \nabla f_{S_2}(\x^k) - \nabla f_{S_2}(\x^{k-1}) + \vv^{k-1}$
				\ENDIF	\label{sfo4}
				\IF{$\| \vv^k\|\leq 2\cep$}
				\RETURN $\x^k$
				\ENDIF
				\STATE $\x^{k+1} = \x^k - \eta \cdot (\vv^k / \|\vv^k\|)$\label{sfo7}
				\ENDFOR	
		\ENDIF
		\ENDFOR
	\end{algorithmic}
\end{algorithm}

The formal pseudocode of the algorithm described above, which we refer to as \SPIDER-SFO\textsuperscript{+}, is detailed in Algorithm \ref{algo:SPIDER-SFOplus}\footnote{In our initial version,  \SPIDER-SFO\textsuperscript{+} first find a FSP and then run NC-search iteration to find a SSP, which also ensures competitive $\tilde{\cO}(\epsilon^{-3})$ rate.  Our newly \SPIDER-SFO\textsuperscript{+} are easier to fuse momentum technique when $n$ is small. Please see the discussion later.}.
The core reason that \SPIDER-SFO\textsuperscript{+} enjoys a highly competitive convergence rate is that, instead of performing a single large step $\delta / \rho$ at the approximate direction of negative curvature as in \NEON2\citep{allen2017neon2}, we split such one large step into $\delta / (\rho\eta)$ small, equal-length mini-steps in Step 2, where each mini-step moves the iteration by an $\eta$ distance.
This allows the algorithm to successively maintain the \SPIDER\ estimate of the current gradient in Step 3 and avoid re-computing the gradient in Step 1.

Our final result on the convergence rate of Algorithm \ref{algo:SPIDER-SFOplus} is stated as:

\begin{theorem}[Second-Order Stationary Point]\label{theo:SSPT}
	Let Assumptions \ref{assu:SSP} hold.
	For the on-line case,  set $q, S_1, S_2, \eta$ in \eqref{inputOne}, $\sK = \frac{\delta L n_0}{\rho\epsilon}$ with any choice of $n_0 \in [1,2\sigma/\epsilon]$,   then with probability at least $1/2$\footnote{By multiple times (at most in $O(\log(1/p))$ times) of verification and restarting Algorithm \ref{algo:SPIDER-SFOplus} , one can also obtain a high-probability result.},  Algorithm \ref{algo:SPIDER-SFOplus} outputs an $\x^k$ with  $j \leq J = 4\left\lfloor\max\left(\frac{3\rho^2\Delta}{\delta^3}, \frac{4\Delta \rho}{\delta\epsilon}\right)\right\rfloor + 4$, and $k \leq K_0= \left(4 \left
	\lfloor\max\left(\frac{3\rho^2\Delta}{\delta^3}, \frac{4\Delta \rho}{\delta\epsilon}\right)\right\rfloor + 4\right) \frac{L n_0 \delta}{\rho \epsilon}$ satisfying
	\beq\label{ESSP}
	\|\nabla f(\x^k)\| \le \cep
	\quad\text{and}\quad 
	\lambda_{\min}(\nabla^2 f(\x^k)) \ge - 3\delta
	,
	\eeq
	with  $\cep = 10\epsilon\log\left(256 \left(\left\lfloor\max\left(\frac{3\rho^2\Delta}{\delta^3}, \frac{4\Delta \rho}{\delta\epsilon}\right)\right\rfloor + 1\right) \frac{\delta L n_0}{\rho \epsilon}+64 \right) = \tilde{\cO}(\epsilon) $.
	The  gradient cost to find a Second-Order Stationary Point with probability at least $1/2$ is upper bounded by 
	\[
	\tilde{\cO}\left( 
	\frac{\Delta L\sigma}{\epsilon^3}+\frac{\Delta \sigma L \rho}{\epsilon^2\delta^2}+\frac{\Delta L^2 \rho^2}{\delta^5}  + \frac{\Delta L^2 \rho}{\epsilon \delta^3}
	+\frac{\sigma^2}{\epsilon^2} +\frac{L^2}{\delta^2}  + \frac{ L \sigma \delta}{\rho \epsilon^2}  
	\right)
	.
	\]
	Analogously for the finite-sum case, under the same setting of Theorem \ref{theo:FSOtwo}, 
	set $q, S_1, S_2, \eta$ in \eqref{inputTwo}, $\sK = \frac{\delta L n_0}{\rho\epsilon}$,   $\cep = 16\epsilon\log\left(256 \left(\left\lfloor\max\left(\frac{3\rho^2\Delta}{\delta^3}, \frac{4\Delta \rho}{\delta\epsilon}\right)\right\rfloor + 1\right) \frac{\delta L n_0}{\rho \epsilon}+64 \right) = \tilde{\cO}(\epsilon) $,  with probability $1/2$,  Algorithm \ref{algo:SPIDER-SFOplus} outputs an $\x^k$ satisfying \eqref{ESSP} in $j \leq J$ and $k\leq K_0$
	with  gradients cost of 
	\[
	\tilde{\cO}\left( 
	\frac{\Delta Ln^{1/2}}{\epsilon^2} +\frac{\Delta \rho L n^{1/2}}{\epsilon\delta^2}+\frac{\Delta L^2 \rho^2}{\delta^5}+ \frac{\Delta L^2 \rho}{\epsilon \delta^3} + n +\frac{L^2}{\delta^2}+ \frac{ L n^{1/2} \delta}{\rho \epsilon}
	\right)
	.
	\]
\end{theorem}

\begin{corollary}
Treating $\Delta$, $L$, $\sigma$, and $\rho$ as positive constants, with high probability the gradient cost for finding an $(\epsilon, \delta)$-approximate second-order stationary point is $\tilde{\cO}(\epsilon^{-3} + \delta^{-2}\epsilon^{-2}+\delta^{-5} )$ for the on-line case and $\tilde{\cO}(n^{1/2}\ep^{-2} + n^{1/2}\delta^{-2}\epsilon^{-1}+\delta^{-3}\epsilon^{-1} + \delta^{-5}+n)$  for the finite-sum case, respectively.
When $\delta = \cO(\epsilon^{0.5})$, the  gradient cost  is $\cO(
 \min(
n^{1/2}\epsilon^{-2}+ \epsilon^{-2.5}, \epsilon^{-3}
 ) )$.
\end{corollary}

Notice that one may directly apply an on-line variant of the \NEON\ method to the \SPIDER-SFO Algorithm \ref{algo:SPIDER-SFO} which alternately does Second-Order Descent (but not maintaining \SPIDER) and First-Order Descent (Running a new \SPIDER-SFO).
Simple analysis suggests that the \NEON + \SPIDER-SFO algorithm achieves a gradient cost of
$
\tilde{\cO}\big(
  \epsilon^{-3} + \epsilon^{-2} \delta^{-3} + \delta^{-5}
\big)
$ 
for the on-line case and 
$
\tilde{\cO}\big(
  n^{1/2}\epsilon^{-2} + n^{1/2} \epsilon^{-1} \delta^{-3} + \delta^{-5}
\big)
$ 
for the finite-sum case \citep{allen2017neon2,xu2017first}.  
We discuss the differences in detail.

\begin{itemize}
\item
The dominate term in the gradient cost of \NEON + \SPIDER-SFO is the so-called \textit{coupling term} in the regime of interest: $\epsilon^{-2} \delta^{-3}$ for the on-line case and $n^{1/2} \epsilon^{-1} \delta^{-3}$ for the finite-sum case, separately.
Due to this term, most convergence rate results in concurrent works for the on-line case such as \citet{reddi2018generic,tripuraneni2017stochastic,xu2017first,allen2017neon2,zhou2018finding} have gradient costs that cannot break the $\cO(\epsilon^{-3.5})$ barrier when $\delta$ is chosen to be $\cO(\epsilon^{0.5})$.
Observe that we always need to run a new \SPIDER-SFO  which at least costs $\cO\big(\min(\epsilon^{-2}, n)\big)$ stochastic gradient accesses.

\item
Our analysis sharpens the seemingly non-improvable coupling term by modifying the single large \NEON\ step to many mini-steps.
Such modification enables us to maintain the \SPIDER\ estimates and obtain a coupling term $\cO\left( \min(n, \epsilon^{-2} ) \delta^{-2}\right)$ of \SPIDER-SFO\textsuperscript{+}, which improves upon the \NEON\ coupling term $\cO\left( \min(n, \epsilon^{-2} ) \delta^{-3}\right)$ by a factor of $\delta$.

\item
For the finite-sum case, \SPIDER-SFO\textsuperscript{+} enjoys a convergence rate that is faster than existing methods only in the regime $n= \Omega(\ep^{-1})$ [Table \ref{tab:comparison}].
For the case of $n = \cO(\ep^{-1})$, using \SPIDER\ to track the gradient in the \NEON\ procedure can be more costly than applying appropriate acceleration techniques \citep{agarwal2017finding,carmon2016accelerated}.%
\footnote{%
\SPIDER-SFO\textsuperscript{+} enjoys a faster rate than \NEON+\SPIDER-SFO where computing the ``full'' gradient dominates the gradient cost, namely $\delta = \cO(1)$ in the on-line case and $\delta = \cO(n^{1/2}\epsilon)$ for the finite-sum case. 
}
Beacause it is well-known that momentum technique \citep{nesterov1983method} provably ensures faster convergence rates when $n$ is sufficient small \citep{Shalev-Shwartz2016}.  One can also  apply momentum  technique to  solve the sub-problem in Step 1 and 3 like \citet{carmon2016accelerated,allen2017neon2} when $n\leq \cO(\epsilon^{-1})$, and thus can achieve  the state-of-the-art  gradient cost of
$$
\tilde{\cO}\left(\min\left(
n \ep^{-1.5} + n^{3/4} \ep^{-1.75}
,
n^{1/2} \ep^{-2} + n^{1/2} \ep^{-1} \delta^{-2}
\right)
+
\min\left(
n + n^{3/4} \delta^{-0.5}, \delta^{-2}
\right)
\delta^{-3}\right)
,
$$
in all scenarios.
\end{itemize}

\subsection{Comparison with Concurrent Works}\label{sec:comp}

\begin{table}[tb!]
\centering
\resizebox{\textwidth}{!}{%
\begin{tabular}{|c|ll|c|c|}
\hline
&
Algorithm
&
&
Online
&
Finite-Sum
\\ \hline
\multirow{5}{*}{
\begin{tabular}{@{}l@{}}
	First-order\\
Stationary \\ Point
\end{tabular}} 
&
GD / SGD
&
\citep{NESTEROV}
&
$\ep^{-4}$
&
$n\ep^{-2}$
\\ \hhline{~----}
&
SVRG / SCSG
&
\begin{tabular}{@{}l@{}}
\citep{allen2016variance}
\\
\citep{reddi2016stochastic}
\\
\citep{lei2017non}
\end{tabular}
&
$\ep^{-3.333}$
&
$n + n^{2/3}\ep^{-2}$
\\ \hhline{~----}
&
\cellcolor{orange!25}
\SPIDER-SFO
&
\cellcolor{orange!25}
(this work)
&
\cellcolor{orange!25}
$\ep^{-3}$
&
\cellcolor{orange!25}
$\,\,\,\,$
$n + n^{1/2}\ep^{-2}$
$\,\,^\Delta$
\\ \hline
\multirow{10}{*}{
\begin{tabular}{@{}l@{}}
	First-order\\ Stationary \\ Point \\ \\ 
\footnotesize
(Hessian- \\ 
\footnotesize
Lipschitz \\ 
\footnotesize
Required)
\end{tabular}} 
&
Perturbed GD / SGD
&
\begin{tabular}{@{}l@{}}
\citep{ge2015escaping}
\\
\citep{jin2017escape}
\end{tabular}
&
\begin{tabular}{@{}c@{}}
$poly(d) \ep^{-4}$
\end{tabular}
&
$n\ep^{-2}$
\\ \hhline{~----}
&
\begin{tabular}{@{}l@{}}
\NEON+GD \\ / \NEON+SGD
\end{tabular}
&
\begin{tabular}{@{}l@{}}
\citep{xu2017first}
\\
\citep{allen2017neon2}
\end{tabular}
&
$\ep^{-4}$
&
$n\ep^{-2}$
\\ \hhline{~----}
&
\begin{tabular}{@{}l@{}}
AGD
\end{tabular}
&
\begin{tabular}{@{}l@{}}
\citep{jin2017accelerated}
\end{tabular}
&
N/A
&
$n\ep^{-1.75}$
\\ \hhline{~----}
&
\begin{tabular}{@{}l@{}}
\NEON+SVRG \\ / \NEON+SCSG
\end{tabular}
&
\begin{tabular}{@{}l@{}}
\citep{allen2016variance}
\\
\citep{reddi2016stochastic}
\\
\citep{lei2017non}
\end{tabular}
&
\begin{tabular}{@{}c@{}}
$\ep^{-3.5}$
\\
$(\ep^{-3.333})$
\end{tabular}
&
$n\ep^{-1.5} + n^{2/3}\ep^{-2}$
\\ \hhline{~----}
&
\NEON+FastCubic/CDHS
&
\begin{tabular}{@{}l@{}}
\citep{agarwal2017finding}
\\
\citep{carmon2016accelerated}
\\
\citep{tripuraneni2017stochastic}
\end{tabular}
&
$\ep^{-3.5}$
&
\cellcolor{yellow!25}
$n \ep^{-1.5} + n^{3/4} \ep^{-1.75}$
\\ \hhline{~----}
&
\NEON+Natasha2
&
\begin{tabular}{@{}l@{}}
\citep{allen2017natasha2}
\\
\citep{xu2017first}
\\
\citep{allen2017neon2}
\end{tabular}
&
\begin{tabular}{@{}c@{}}
$\ep^{-3.5}$
\\
$(\ep^{-3.25})$
\end{tabular}
&
$n\ep^{-1.5} + n^{2/3}\ep^{-2}$
\\ \hhline{~----}
&
\cellcolor{orange!25}
\SPIDER-SFO\textsuperscript{+}
&
\cellcolor{orange!25}
(this work)
&
\cellcolor{orange!25}
$\ep^{-3}$
&
\cellcolor{orange!25}
$n^{1/2} \epsilon^{-2}$
$\,\,^\Theta$
\\ \hline
\end{tabular}}
\caption{
Comparable results on the gradient cost for nonconvex optimization algorithms that use only individual (or stochastic) gradients.
Note that the gradient cost hides a poly-logarithmic factors of $d$, $n$, $\ep$.
For clarity and brevity purposes, we record for most algorithms the gradient cost for finding an $(\ep, \cO(\ep^{0.5}))$-approximate second-order stationary point.
For some algorithms we added in a bracket underneath the best gradient cost for finding an $(\ep, \cO(\ep^\alpha))$-approximate second-order stationary point among $\alpha \in (0,1]$, for the fairness of comparison.
\\
{
\footnotesize $^\Delta$: we provide lower bound for this gradient cost entry.
}
\\
{
\footnotesize $^\Theta$: this entry is for $n\ge \Omega(\ep^{-1})$ only, in which case \SPIDER-SFO\textsuperscript{+} outperforms \NEON+FastCubic/CDHS.
}
}
\label{tab:comparison}

\end{table}

\setlength{\textfloatsep}{10pt plus 1.0pt minus 2.0pt}

This subsection compares our \SPIDER\ algorithms with concurrent works.
In special, we detail our main result for applying \SPIDER\ to first-order methods in the list below:

\begin{enumerate}[(i)]
\item
For the problem of finding an $\epsilon$-approximate \FSP, under Assumption \ref{assu:main} our results indicate a gradient cost of $\cO(\min(\ep^{-3}, n^{1/2}\ep^{-2}) )$ which supersedes the best-known convergence rate results for stochastic optimization problem \eqref{opt_eq} [Theorems \ref{theo:FSOone} and \ref{theo:FSOtwo}].
Before this work, the best-known result is $\cO\left(\min(\ep^{-3.333}, n^{2/3}\ep^{-2})\right)$, achieved by \citet{allen2016variance,reddi2016stochastic} in the finite-sum case and \citet{lei2017non} in the on-line case, separately.
Moreover, such a gradient cost achieves the algorithmic lower bound for the finite-sum setting [Theorem \ref{theo:lowerbdd}].

\item
For the problem of finding $(\ep,\delta)$-approximate \SSP\ $x$, under both Assumptions \ref{assu:main} and \ref{assu:SSP}, the gradient cost is $\tilde{\cO}(\ep^{-3} + \epsilon^{-2}\delta^{-2} + \delta^{-5}  )$ in the on-line case and $\tilde{\cO}(
n^{1/2}\epsilon^{-2}
 +n^{1/2} \epsilon^{-1} \delta^{-2} + \epsilon^{-1} \delta^{-3}
 +\delta^{-5} +n
)$ in the finite-sum case [Theorem \ref{theo:SSPT}].
In the classical definition of \SSP\ where $\delta = \cO(\ep^{0.5})$, such gradient cost is simply $\cO(\epsilon^{-3})$ in the on-line case. 
In comparison, to the best of our knowledge the best-known results only achieve a gradient cost of $\cO(\ep^{-3.5})$ under similar assumptions
\citep{reddi2018generic, tripuraneni2017stochastic, allen2017natasha2, allen2017neon2, zhou2018finding}.
\end{enumerate}

We summarize the comparison with concurrent works that solve \eqref{opt_eq} under similar assumptions in Table \ref{tab:comparison}.
In addition, we provide Figure \ref{fig:comp} which draws the gradient cost against the magnitude of $n$ for both an approximate stationary point.%
\footnote{%
One of the results not included in this table is \citet{carmon2017convex}, which finds an $\ep$-approximate first-order stationary point in $\cO(n\ep^{-1.75})$ gradient evaluations. 
However, their result relies on a more stringent Hessian-Lipschitz condition, in which case a second-order stationary point can be found in similar gradient cost \citep{jin2017accelerated}.
}
For simplicity, we leave out the complexities of the algorithms that has Hessian-vector product access and only record algorithms that use stochastic gradients only.%
\footnote{%
Due to the \NEON\ method \citep{xu2017first,allen2017neon2}, nearly all existing Hessian-vector product based algorithms in stochastic optimization can be converted to ones that use stochastic gradients only.}
Specifically, the yellow-boxed complexity $\cO(n \ep^{-1.5} + n^{3/4} \ep^{-1.75})$ in Table \ref{tab:comparison}, which was achieved by \NEON+FastCubic/CDHS \citep{allen2017neon2,jin2017accelerated} for finding an approximate second-order stationary point in the finite-sum case using momentum technique, are the only results that have \textit{not} been outperformed by our \SPIDER-SFO\textsuperscript{+} algorithm in certain parameter regimes ($n \le \cO(\epsilon^{-1})$ in this case).

\begin{figure}
\centering
\includegraphics[height=2.2in]{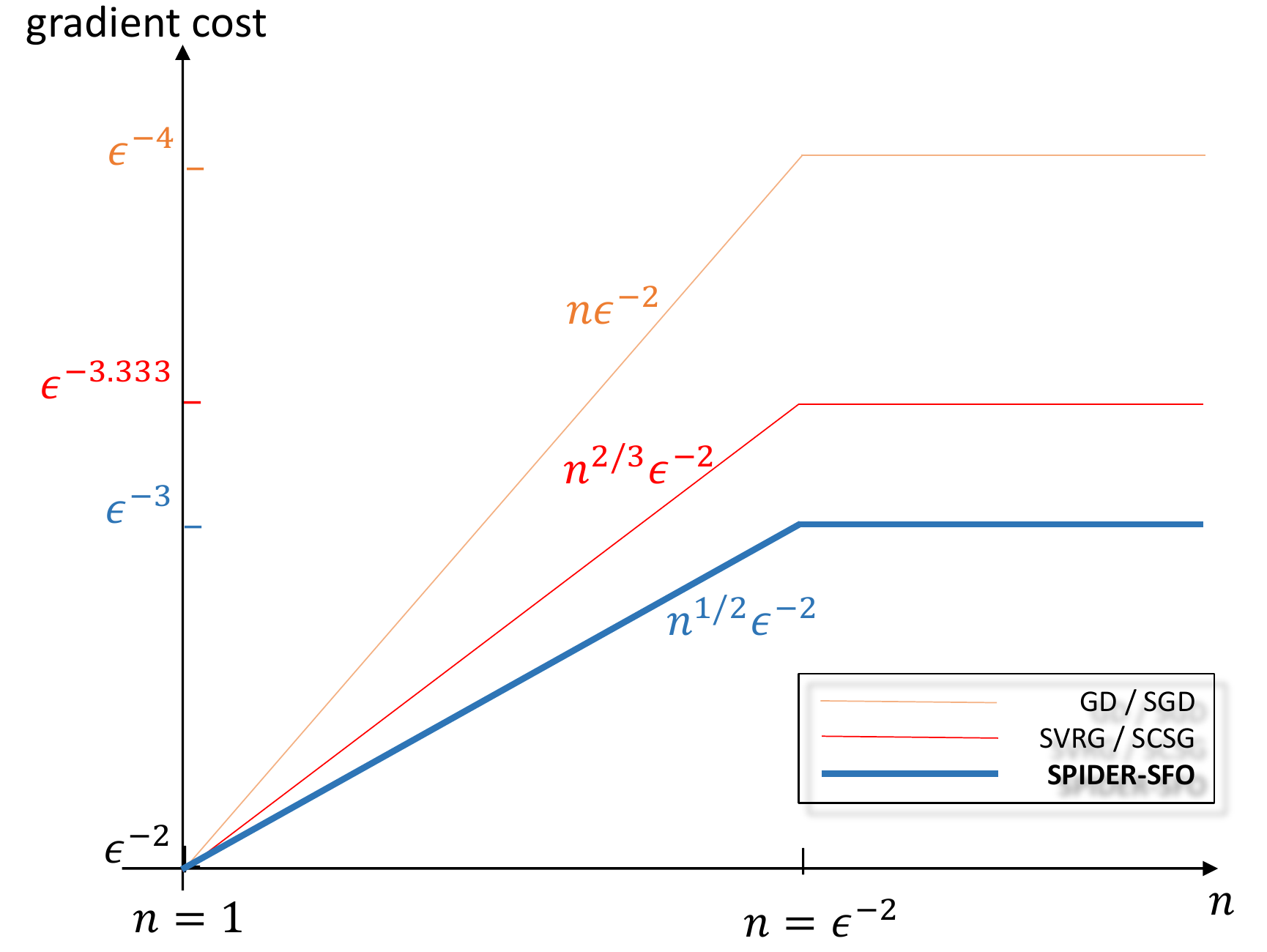}
\quad
\includegraphics[height=2.2in]{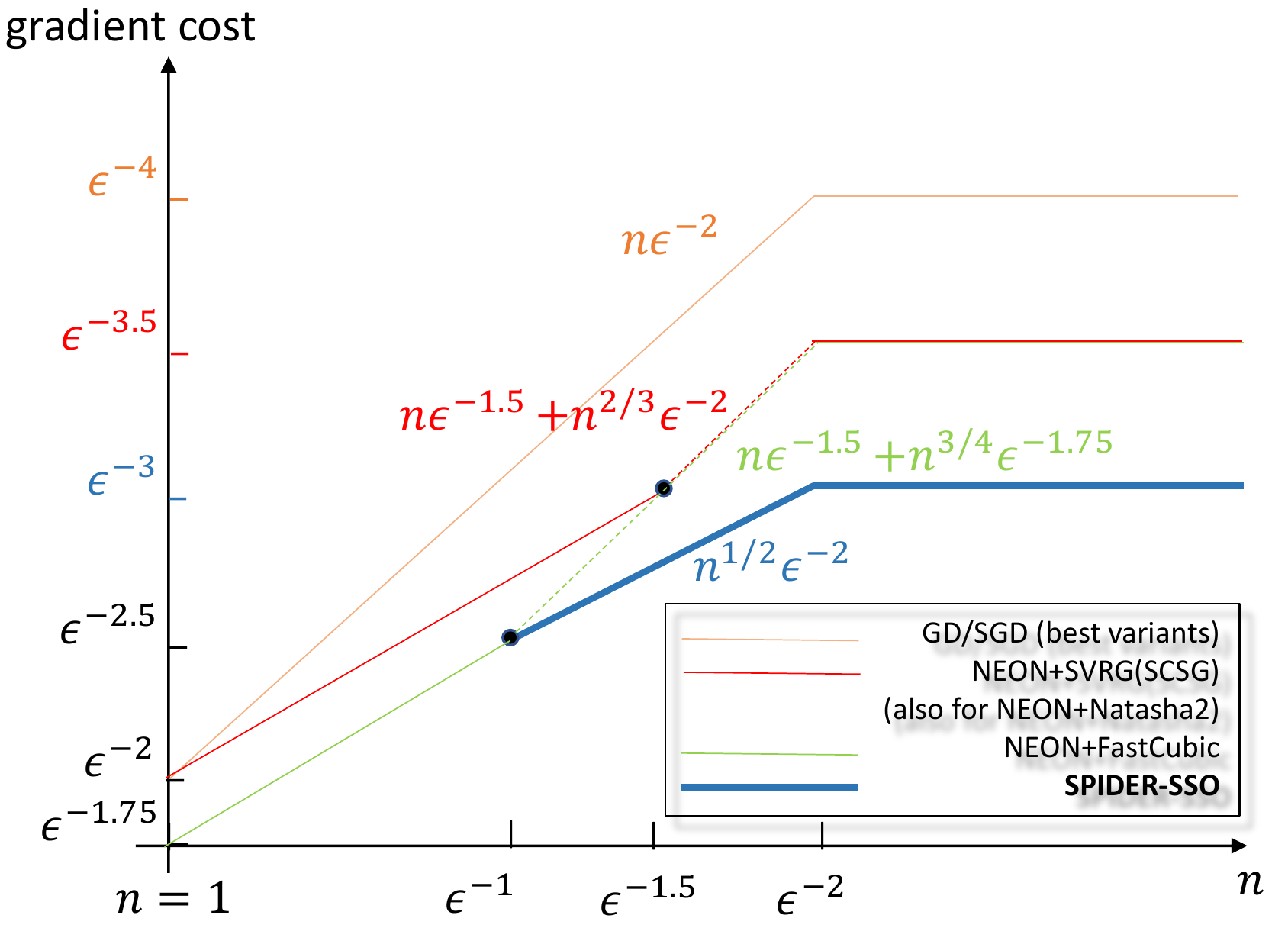}
\caption{
Left panel:
gradient cost comparison for finding an $\ep$-approximate first-order stationary point.
Right panel:
gradient cost comparison for finding an $(\ep,\cO(\ep^{0.5}))$-approximate second-order stationary points (note we assume Hessian Lipschitz condition). 
Both axes are on the logarithmic scale of $\ep^{-1}$. 
}
\label{fig:comp}
\end{figure}

\setlength{\textfloatsep}{10pt plus 1.0pt minus 2.0pt}

\def\IC{online case}
\def\FC{finite-sum case}

\section{SPIDER for Stochastic Zeroth-Order Method}\label{sec:SZO}
\begin{algorithm}[tb]
\caption{\SPIDER-SZO: Input $\x^0$,  $S_1$, $S_2$, $q$, $u$, $\epsilon$ 
(For finding first-order stationary point)}
\label{algo:SPIDER-ZERO}
\begin{algorithmic}[1]
\FOR{$k=0$ to $K$}
\IF{$\mod(k,q)=0$}	
       \STATE   Draw $S_1'=S_1/d$  training samples, for each dimension $j\in [d]$, compute  {\hfill   ($\diamond$ with $2S_1$  total  IZO costs) } 
              \[
              v^k_{j} = \dfrac{1}{S_1'} \displaystyle\sum_{i\in \cS_1'}\frac{f_i(\x^k +\mu \e_j) - f_i (\x^k)}{\mu}                   
              \]         
        where $\e_j$ denotes the vector with $j$-th natural unit basis vector.
\ELSE
		\STATE   Draw $S_2$ sample  pairs $(i, \uu)$, where $i\in [n]$ and $\uu \sim N(\mathbf{0},\Ib_d)$ with $i$ and $\mu$ being independent.
		\STATE Update
		\[
		\vv^k =  \dfrac{1}{S_2}
		\displaystyle\sum_{(i, \uu)\in \cS_2} \left( 
		\frac{f_{i}(\x^k+\mu \u)- f_i(\x^k)}{\mu}\uu - \frac{f_{i}(\x^{k-1}+\mu \u)- f_i(\x^{k-1})}{\mu}\uu 
		\right) +   \vv^{k-1}
		\]
\ENDIF
  \STATE   $\x^{k+1} = \x^k - \eta^k \vv^k$ \text{ where}$\quad \eta^k = \min \left( \frac{\epsilon}{Ln_0\|\vv^k\|}, \frac{1}{2L n_0} \right)$ 	  {\hfill $\diamond$ for convergence rates in expectation}
\ENDFOR
\STATE  Return $\tilde{\x}$ chosen uniformly at random from $\{\x^k \}_{k=0}^{K-1}$
\end{algorithmic}
\end{algorithm}

For SZO algorithms,  \eqref{Q0bdd} can be solved only from the Incremental Zeroth-Order Oracle (IZO)\citep{nesterov2011random},  which is defined as:

\begin{definition}
	An IZO takes an index $i\in [n]$ and a point $\x \in \RR^d$, and returns the $f_i(\x)$. 
\end{definition}

We use Assumption \ref{assu:main2} (including (ii'))  for convergence analysis which  is standard for  SZO\citep{nesterov2011random,ghadimi2013stochastic} algorithms.
Because the true gradient are not allowed to obtain for SZO. Most works \citep{nesterov2011random,ghadimi2013stochastic,shamir2017optimal} use the gradient of a smoothed version of the objective function through a two-point feedback in a stochastic setting.  Following \citep{nesterov2011random}, we consider the typical Gaussian distribution  in the convolution to smooth the function.  
Define
\begin{eqnarray}
\hf (\x) = \frac{1}{(2 \pi)^{\frac{d}{2}}} \int f(\x+\mu \uu)e ^{-\frac{1}{2}\|\uu\|^2}d\uu = \E_\uu [f(\x +\mu \uu)], 
\end{eqnarray}
where $\x \in \RR^{d}$. 
From \citep{nesterov2011random},  the following properties holds :

\begin{enumerate}[(i)]
	\item The gradient of $\hf$ satisfies:
	\begin{eqnarray}
	\nabla \hf (\x) =  \frac{1}{(2 \pi)^{\frac{d}{2}}} \int \frac{f(\x+\mu \uu) -f(\x)  }{\mu} \uu e ^{-\frac{1}{2}\|\uu\|^2}d\uu. 
	\end{eqnarray}
	\item For any $\x \in \RR^d$,  $f(\x)$ has Lipschitz continuous gradients, we have 
	\begin{eqnarray}\label{324}
	\| \nabla  \hf(\x) - \nabla f(\x) \|\leq \frac{\mu}{2}L(d+3)^{\frac{3}{2}}.
	\end{eqnarray}
	\item For any $\x \in \RR^d$, $f(\x)$ has Lipschitz continuous gradients, we have 
	\begin{eqnarray}\label{325}
	\E_{\uu}\left[\frac{1}{\mu^2} \left(f(\x+\mu \uu) -f(\x)\right)^2\| \uu\|^2    \right]\leq  \frac{\mu^2}{2} L^2 (d+6)^3 +2(d+4) \| \nabla f(\x)\|^2.
	\end{eqnarray} 
\end{enumerate}

From the (1), suppose  $\uu\sim N(\mathbf{0}, \mathbf{I}_d)$, and $i\in [n]$, with $\uu $ and $i$ being independent, we have
\begin{eqnarray}\label{226}
&&\E_{i, \uu} \frac{f_{i}(\x^k +\mu \uu_i) - f_{i}(\x^k)}{\mu}\uu = \frac{1}{(2 \pi)^{\frac{d}{2}}} \E_i \left(\int \frac{f_i(\x+\mu \uu) -f_i(\x)  }{\mu} \uu e ^{-\frac{1}{2}\|\uu\|^2}d\uu\right)\notag\\
&=&\frac{1}{(2 \pi)^{\frac{d}{2}}}  \left(\int \frac{f(\x+\mu \uu) -f(\x)  }{\mu} \uu e ^{-\frac{1}{2}\|\uu\|^2}d\uu\right) =  \nabla \hf(\x^k).
\end{eqnarray}
Also 
\begin{eqnarray}\label{335}
\E_{i,\uu} \left[\frac{f_{i}(\x^k +\mu \uu) - f_{i}(\x^k)}{\mu} \uu - \left(\frac{f_{i}(\x^{k-1} +\mu \uu) - f_{i}(\x^{k-1})}{\mu}\uu \right)\right]   =\nabla  \hf(\x^{k}) -\nabla\hf(\x^{k-1}).
\end{eqnarray}

 For non-convex case, the best known result is $\cO(d\epsilon^{-4})$ from \cite{ghadimi2013stochastic}. We has not found a work that applying  Variance Reduction technique to significantly reduce the complexity of IZO. This might because that even in \FC,  the  full gradient is not available (with noise).  In this paper, we give a  stronger results by  \SPIDER\ technique, directly reducing the IZO from $\cO(d\epsilon^{-4})$ to $\cO( \min(dn^{1/2}\epsilon^{-2}, d\epsilon^{-3}))$.

From \eqref{335}, we can  integrate the two-point feed-back to track $\nabla \hf(\x)$.   The algorithm is shown in Algorithm \ref{algo:SPIDER-ZERO}. Then the following lemma shows that  $\vv^k$ is a high accurate estimator of $\| \nabla \hf(\x^k)\|$:
\begin{lemma}\label{zero1}
Under the Assumption \ref{assu:main2}, suppose $i$ is random number of the function index, ($i\in [n]$) and $\uu$ is a  standard Gaussian random vector, i.e. $\uu\sim N(\mathbf{0}, \mathbf{I}_d)$, we have
	\begin{eqnarray}\label{288}
	\E_{ i, \uu} \left\|    \left[\frac{f_{i}(\x +\mu \uu) - f_{i}(\x)}{\mu} \uu - \left(\frac{f_{i}(\y +\mu \uu) - f_{i}(\y)}{\mu}\uu \right)\right]  \right\|^2 \leq 2 (d+4)L^2\|\x - \y\|^2+ 2\mu^2 (d+6)^3L^2.
	\end{eqnarray}
\end{lemma}

 From  \eqref{324}, by setting a smaller $\mu$,  the smoothed gradient $\nabla \hf(\x)$  approximates $\nabla f(\x)$, which ensures sufficient function descent in each iteration.  For simpleness, we only give expectation result, shown in Theorem \ref{zero3}.
\begin{theorem}\label{zero3}
		Under the Assumption \ref{assu:main2} (including (ii')). For infinite case, set	$\mu  =  \min\left( \frac{\epsilon}{2\sqrt{6} L\sqrt{d}}, \frac{\epsilon}{\sqrt{6} n_0 L(d+6)^{3/2}}   \right)$, $S_1 = \frac{96d\sigma^2}{\epsilon^2}$, $S_2 = \frac{30(2d+9)\sigma}{\epsilon n_0}$,  $q =\frac{5n_0\sigma}{\epsilon}$, where $n_0 \in [1, \frac{30(2d+9)\sigma}{\epsilon}]$. 	In the finite-sum case, set the parameters $S_2 = \frac{(2d+9)n^{1/2}}{n_0}$, and  $q=  \frac{n_0n^{1/2}}{6}$, let $S_1/d = [n]$, i.e. $v^k_j = f(\x^k +\mu \e_j) - f(\x^k)/\mu$ with $j\in [d]$, where $n_0 \in [1, \frac{n^{1/2}}{6}]$.		 Then with $\eta^k  = \min (\frac{1}{2Ln_0}, \frac{\epsilon}{Ln_0\|\vv^k\|})$,
		 $K = \left\lfloor(4L \Delta n_0)\epsilon^{-2}\right\rfloor+1$, for  
		  Algorithm \ref{algo:SPIDER-ZERO} we have
		 \begin{eqnarray}
		 \E \left[\|\nabla f(\tx) \|\right]\leq 6\epsilon.
		 \end{eqnarray}
 The IZO calls are $\cO\left( d\min(n^{1/2}\epsilon^{-2},  \epsilon^{-3})\right)$.
\end{theorem}

\section{Summary and Future Directions}\label{sec:summary}

We propose in this work the \SPIDER\ method for non-convex optimization.
Our \SPIDER-type algorithms for first-order and zeroth-order optimization have update rules that are reasonably simple and achieve excellent convergence properties.
However, there are still some important questions left.
For example, the lower bound results for finding a second-order stationary point are \textit{not} complete.
Specially, it is \textit{not} yet clear if our $\tilde\cO(\ep^{-3})$ for the on-line case and $\tilde\cO(n^{1/2} \epsilon^{-2})$ for the finite-sum case gradient cost upper bound for finding a second-order stationary point (when $n\ge \Omega(\ep^{-1})$) is \textit{optimal} or the gradient cost can be further improved, assuming both Lipschitz gradient and Lipschitz Hessian.
We leave this as a future research direction.

\paragraph{Acknowledgement}
The authors would like to thank NIPS Reviewer 1 to point out a mistake in  the original proof of Theorem \ref{theo:FSOone} and thank Zeyuan Allen-Zhu and Quanquan Gu for relevant discussions and pointing out references \citet{zhou2018stochastic,zhou2018finding}, also Jianqiao Wangni for pointing out references \citet{nguyen2017sarah,nguyen2017stochastic}, and Zebang Shen,  Ruoyu Sun, Haishan Ye, Pan Zhou for very helpful discussions and comments.
Zhouchen Lin is supported by National Basic Research Program of China (973 Program) (grant no. 2015CB352502), National Natural Science Foundation (NSF) of China (grant nos. 61625301 and 61731018), and Microsoft Research Asia.

\bibliographystyle{apalike2}
\bibliography{SmileSGD}

\pagebreak\appendix
\section{Vector-Martingale Concentration Inequality}\label{sec:pinelis}

In this and next section, we sometimes denote for brevity that $\EE_{k}[\cdot] = \EE[\cdot \mid x_{0:k}]$, the expectation operator conditional on $x_{0:k}$, for an arbitrary $k\ge 0$.

\paragraph{Concentration Inequality for Vector-valued Martingales}
We apply a result by \citet{pinelis1994optimum} and conclude Proposition \ref{proh} which is an Azuma-Hoeffding-type concentration inequality.
See also \citet{kallenberg1991some}, Lemma 4.4 in \citet{zhang2005learning} or Theorem 2.1 in \cite{zhang2005learning} and the references therein.

\begin{proposition}[Theorem 3.5 in \citet{pinelis1994optimum}]\label{proh}
	Let $\bm\ep_{1:K} \in \RR^d$ be a vector-valued martingale difference sequence with respect to $\cF_k$, i.e., for each $k=1,\dots,K$, $\EE[\bm\ep_k \mid \cF_{k-1}] = 0$ and $\|\ep_k\|^2 \leq B_k^2$. We have 
	\beq\label{azumain}
	\PP \left(
	\left\|
	\sum_{k=1}^{K} \bm\ep_k
	\right\|
	\ge \lambda \right)
	\le
	4 \exp \left(
	-\frac{\lambda^2}{4 \sum_{k=1}^{K} B_k^2} 
	\right),
	\eeq
	where $\lambda$ is an arbitrary real positive number.
\end{proposition}

Proposition \ref{proh} is not a straightforward derivation of one-dimensional Azuma's inequality.
The key observation of Proposition \ref{proh} is that, the bound on the right hand of \eqref{azumain} is \textit{dimension-free} (note the Euclidean norm version of $\RR^d$ is $(2,1)$-smooth).
Such dimension-free feature could be found as early as in \citet{kallenberg1991some}, uses the so-called \textit{dimension reduction
	lemma} for Hilbert space which is inspired from its continuum version proved in \citet{kallenberg1991some}.
Now, we are ready to prove Proposition \ref{prop:aggregate}.

\subsection{Proof of Proposition \ref{prop:aggregate}}
\begin{proof}[Proof of Proposition \ref{prop:aggregate}]
It is straightforward to verify from the definition of $\tilde{Q}$ in \eqref{tildeQ} that
\[
\tilde{Q}(\hx_{0:K} ) - Q(\hx^K )
=
\tilde{Q}(\hx^0 ) - Q(\hx^ 0 )
+
\sum_{k=1}^K \bxi_ k (\hx_{0:k} ) - (Q(\hx^k ) - Q(\hx^{k-1} ))
\]
is a martingale, and hence \eqref{mart_var} follows from the property of $L^2$ martingales \citep{DURRETT}.
\end{proof}

\subsection{Proof of Lemma \ref{lemm:aggregate}}
\begin{proof}[Proof of Lemma \ref{lemm:aggregate}]
For any \blue{$k> 0$}, we have from Proposition \ref{prop:aggregate} (by applying $\tilde{Q} = \cV$)
\beq\label{fir}
\E_k \|     \hV^{k} - \hB(\x^k)  \|^2
=
\E_k \|  \hB_{S_*}(\x^{k}) - \hB(\x^k)  - \hB_{S_*}(\x^{k-1})  + \hB(\x^{k-1})\|^2+  \| \hV^{k-1}  - \hB(\x^{k-1})    \|^2
.
\eeq
Then 
\begin{eqnarray}\label{sec}
&&\E_k \|  \hB_{S_*}(\x^{k}) - \hB(\x^k)  - \hB_{S_*}(\x^{k-1})  + \hB(\x^{k-1})\|^2\notag\\
&\overset{a}=&\frac{1}{\cS_*}\Ei \|  \hB_{i}(\x^{k}) - \hB(\x^k)  - \hB_{i}(\x^{k-1})  + \hB(\x^{k-1})\|^2\notag\\
&\overset{b}\leq&\frac{1}{\cS_*}\Ei \|  \hB_{i}(\x^{k}) - \hB_{i}(\x^{k-1})\|^2\notag\\
&\overset{\eqref{hAidist}}{\leq}&\frac{1}{\cS_*}L_\hB^2\Ei \| \x^{k}-\x^{k-1}\|^2\leq \frac{L_\hB^2 \epsilon_1^2}{\cS_*},
\end{eqnarray}
where in $\overset{a}=$ and $\overset{b}\leq$, we use Eq \eqref{hAidist2}, and \blue{$S_*$} are random sampled from $[n]$ with replacement.
Combining \eqref{fir} and \eqref{sec}, we have
\begin{eqnarray}
&&\E_k \|     \hV^{k} - \hB(\x^k)  \|^2 \leq \frac{L_\hB^2 \epsilon_1^2}{\cS_*} + \|   \hV^{k-1} - \hB(\x^{k-1})\|^2.
\end{eqnarray}
Telescoping the above display for $k' = k-1,\dots, 0$ and using the iterated law of expectation, we have
\begin{eqnarray}
\E \| \hV^{k} - \hB(\x^k)    \|^2 \leq \frac{ k L_\hB^2 \epsilon_1^2}{\cS_*} + \E\| \hV^0 - \hB(\x^0) \|^2.
\end{eqnarray}
\end{proof}

\section{Deferred Proofs}

\subsection{Proof of Lemma \ref{111}}
\begin{proof}[Proof of Lemma \ref{111}]  
For $k = k_0$, we have 
\begin{eqnarray}\label{k_01}
&&\E_{k_0} \| \vv^{k_0} -  \nabla f(\x^{k_0})\|^2\notag\\
& = & \E_{k_0}\| \nabla f_{S_1}(\x^{k_0}) - \nabla f(\x^{k_0})\|^2 
 \le
\frac{\sigma^2}{S_1} = \frac{\epsilon^2}{2}.
\end{eqnarray}
From Line \ref{line:update2} of Algorithm \ref{algo:SPIDER-SFO} we have for all $k\geq 0$,
\beq\label{lemm:100}
\|\x^{k+1} - \x^k\|
=
\min\left(\frac{\epsilon}{Ln_0\| \vv^k\|} ,\frac{1}{2Ln_0}\right) \|\vv^k\|\leq  \dfrac{\epsilon}{Ln_0}
. 
\eeq
Applying Lemma \ref{lemm:aggregate} with $\epsilon_1 = \epsilon / (L n_0)$, $S_2 = 2\sigma / (\epsilon n_0)$, $K = k - k_0  \leq q = \sigma n_0/\epsilon$, we have
\begin{eqnarray}
\E_{k_0}\|  \vv^{k} -  \nabla f(\x^{k}) \|^2
 \leq 
\frac{ \sigma n_0 L^2 }{\epsilon} 
\cdot
\frac{\epsilon^2}{L^2 n_0^2} 
\cdot
\frac{\epsilon n_0}{2\sigma}
 + \E_{k_0} \| \vv^{k_0} -  \nabla f(\x^{k_0})\|^2
 \overset{\eqref{k_01}}{=} 
\epsilon^2
,
\end{eqnarray}
completing the proof.
\end{proof}

\subsection{Proof of Expectation Results for FSP}
The rest of this section devotes to the proofs of Theorems \ref{theo:FSOone}, \ref{theo:FSOtwo}.
To prepare for them, we first conclude via standard analysis the following
 
\begin{lemma}\label{444}
Under the Assumption \ref{assu:main},  setting $k_0 = \lfloor k/q \rfloor \cdot q $, we have  
 	\begin{eqnarray}
 	\E_{k_0} \left[  f(\x^{k+1}) - f(\x^k) \right] \leq   -\frac{\epsilon }{4Ln_0}   \E_{k_0} \left\| \vv^k\right\| + \frac{3\epsilon^2}{4n_0 L }.
 	\end{eqnarray}
 \end{lemma}

\begin{proof}[Proof of Lemma \ref{444}]
 From Assumption \ref{assu:main} (\ref{a2}), we have
 \begin{eqnarray}
  \| \nabla f(\x) - \nabla f(\y) \|^2 =   \left\| \E_i \left(\nabla f_i(\x) - \nabla f_i(\y)\right) \right\|^2\leq  \E_i \|\nabla f_i(\x) - \nabla f_i(\y) \|^2 \leq L^2 \|\x - \y\|^2.
 \end{eqnarray} 
So $f(\x)$ has $L$-Lipschitz continuous gradient, then
\begin{eqnarray}\label{L-smooth}
f(\x^{k+1}) \notag
&\leq&
 f(\x^k)+\<\nabla f(\x^k), \x^{k+1}-\x^k \>+\frac{L}{2}\left\|\x^{k+1}-\x^k\right\|^2\\
&=& 
 f(\x^k) - \eta^k \<\nabla f(\x^k), \vv^k \> +\frac{L(\eta^k)^2}{2}\left\|\vv^k\right\|^2\notag\\
&=&
 f(\x^k) -\eta^k \left(1-\frac{\eta^k L}{2}\right)\left\|\vv^k\right\|^2- \eta^k\< \nabla f(\x^k)-\vv^k,\vv^k\>\notag\\
&\overset{a}{\leq}& f(\x^k)  -\eta^k\left(\frac{1}{2}-\frac{\eta^k L}{2}\right)\left\|\vv^k  \right\|^2+\frac{\eta^k}{2}\left\|\vv^k - \nabla f(\x^k)\right\|^2, 
\end{eqnarray}
where in $\overset{a}\leq$, we applied Cauchy-Schwarz inequality.  
Since $\eta^k= \min\left(\frac{\epsilon}{Ln_0\| \vv^k\|} ,\frac{1}{2Ln_0}\right) \leq \frac{1}{2Ln_0}\leq \frac{1}{2L}$, we have
\begin{eqnarray}
\eta^k\left(\frac{1}{2}-\frac{\eta^k L}{2}\right)\left\|\vv^k  \right\|^2
 \geq 
\frac{1}{4} \eta^k \left\| \vv^k\right\|^2= \frac{\epsilon^2}{8n_0L } \min\left( 2\left\| \frac{\vv^k} {\epsilon}\right\|, \left\| \frac{\vv^k}{\epsilon}\right\|^2  \right)\overset{a}{\geq}  \frac{\epsilon\| \vv^k\|- 2\epsilon^2}{4n_0L },
\end{eqnarray}
where in $\overset{a}\geq$, we use $V(x) = \min \left(| x|, \frac{x^2}{2}\right) \geq |x | - 2$ for all $x$. Hence
\begin{eqnarray}
f(\x^{k+1})
& \leq& 
f(\x^k) - \frac{\epsilon\| \vv^k\|}{4Ln_0} + \frac{\epsilon^2}{2n_0L} + \frac{\eta^k}{2}\left\|\vv^k - \nabla f(\x^k)\right\|^2\notag\\
&\overset{\eta^k\leq \frac{1}{2Ln_0} }{\leq}& 
f(\x^k) - \frac{\epsilon\| \vv^k\|}{4Ln_0} + \frac{\epsilon^2}{2n_0L} + \frac{1}{4Ln_0}\left\|\vv^k - \nabla f(\x^k)\right\|^2
.
\end{eqnarray}
Taking expectation on the above display and using Lemma \ref{111}, we have
\begin{eqnarray}
\E_{k_0} f(\x^{k+1}) - \E_{k_0} f(\x^k) \leq   -\frac{\epsilon }{4Ln_0}   \E_{k_0} \left\| \vv^k\right\|+\frac{3\epsilon^2}{4Ln_0}.
\end{eqnarray}
\end{proof}

The proof is done via the following lemma:
\begin{lemma}\label{lemma10u}
	Under Assumption \ref{assu:main}, for all $k\geq 0$,   we have 
	\begin{eqnarray}\label{lemma10}
	\E \|\nabla f(\x^k)\| \leq \E \|\vv^k \|+ \epsilon.
	\end{eqnarray}
	
\end{lemma}

\begin{proof}
By taking the total expectation in Lemma \ref{111}, we have
\begin{eqnarray}
\E\|\vv^{k} -\nabla f(\x^k) \|^2 \leq \epsilon^2.
\end{eqnarray}
Then  by Jensen's inequality
$$
\left(
\E \|\vv^{k} -\nabla f(\x^k)\|
\right)^2
 \le
\E \|\vv^{k} -\nabla f(\x^k)  \|^2
 \leq 
\epsilon^2
. 
$$
So using triangle inequality
\begin{eqnarray}
\E\|  \nabla f(\x^k) \|
& =& \E \| \vv^{k}  - (\vv^{k} -\nabla f(\x^k)  )   \|\notag\\
&\leq&  \E \| \vv^{k}  \| + \E \| \vv^{k} -\nabla f(\x^k)  \|\leq \E \| \vv^{k}  \| + \epsilon.
\end{eqnarray}
This completes our proof.
\end{proof}

Now, we are ready to prove Theorem \ref{theo:FSOone}.

\begin{proof}[Proof of Theorem \ref{theo:FSOone}]
Taking full expectation on Lemma \ref{444}, and telescoping the results from $k=0$ to $K-1$, we have
\begin{eqnarray}\label{end1}
\frac{\epsilon}{4Ln_0}\sum_{k=0}^{K-1}\E \|\vv^k \| \leq f(\x^0) - \E f(\x^K) + \frac{3K\epsilon^2}{4Ln_0}\overset{\E f(\x^K)\geq f^*}{\leq} \Delta + \frac{3K\epsilon^2}{4Ln_0}.
\end{eqnarray}	
Diving  $\frac{4Ln_0}{\epsilon}K$  both sides of \eqref{end1},  and using $K = \lfloor\frac{4L\Delta n_0}{\epsilon^2}\rfloor+1 \geq \frac{4L\Delta n_0}{\epsilon^2}$, we have
\begin{eqnarray}\label{end2}
\frac{1}{K}\sum_{k=0}^{K-1}\E \|\vv^k \| \leq \Delta \cdot  \frac{4Ln_0}{\epsilon}\frac{1}{K} +3\epsilon\leq 4\epsilon.
\end{eqnarray}

Then from the choose of $\tx$, we have
\begin{eqnarray}\label{final}
\E \| \nabla f(\tx)\| =  \frac{1}{K}\sum_{k=0}^{K-1}\E \|\nabla f(\x^k) \| \overset{\eqref{lemma10}}{\leq} \frac{1}{K}\sum_{k=0}^{K-1}\E \| \vv^k \| + \epsilon \overset{\eqref{end2}}{\leq} 5\epsilon.
\end{eqnarray}

To compute the gradient cost, note in each $q$ iterations we access for one time $S_1$ stochastic gradients and for $q$ times of $S_2$ stochastic gradients, and hence the cost is 
\begin{eqnarray}
\left\lceil K \cdot \frac{1}{q}\right\rceil S_1 
+
K S_2 
&\overset{S_1 = qS_2}{\le}&
2K \cdot S_2 + S_1  \notag\\
&\leq& 
2\left(\frac{4Ln_0\Delta}{\epsilon^2}  \right)  \frac{2\sigma}{\epsilon n_0} + \frac{2\sigma^2}{\epsilon^2}+2S_2\notag\\
&=&\frac{16L\sigma \Delta}{\epsilon^3} +\frac{2\sigma^2}{\epsilon^2}+ \frac{4\sigma}{n_0\epsilon}.
\end{eqnarray}
This concludes a gradient cost of $16L\Delta\sigma\epsilon^{-3}+2\sigma^2\epsilon^{-2}+ 4\sigma n_0^{-1}\epsilon^{-1} $.
\end{proof}

\begin{proof}[Proof of Theorem \ref{theo:FSOtwo}]
For Lemma \ref{111}, we have 
\begin{eqnarray}
\E_{k_0} \| \vv^{k_0} -  \nabla f(\x^{k_0})\|^2
=
 \E_{k_0}\| \nabla f(\x^{k_0}) - \nabla f(\x^{k_0})\|^2  = 0
 .
\end{eqnarray}
With the above display, applying Lemma \ref{lemm:aggregate} with $\epsilon_1 =\frac{\epsilon}{L n_0}$, and $S_2 = \frac{n^{1/2}}{\epsilon n_0}$, $K = k -k_0  \leq q = n_0n^{1/2} $, we have
\begin{eqnarray}
\E_{k_0}\|  \vv^{k_0} -  \nabla f(\x^{k_0}) \|^2
 \leq  
n_0n^{1/2} L^2  \cdot\frac{\epsilon^2}{L^2 n_0^2}\cdot\frac{\epsilon n_0}{n^{1/2}}+ \E_{k_0} \| \vv^{k_0} -  \nabla f(\x^{k_0})\|^2
 \overset{\eqref{k_01}}{=} 
\epsilon^2
.
\end{eqnarray}
So  Lemma \ref{111} holds.  Then from the same technique of on-line case, we can obtain \eqref{lemm:100} and \eqref{lemma10u}, and \eqref{final}.
The gradient cost analysis is computed  as:
\begin{eqnarray}
\left\lceil K \cdot \frac{1}{q}\right\rceil S_1 
+
K S_2 
&\overset{S_1 = qS_2}{\le}&
2K + S_1  \notag\\
&\leq& 
2\left(\frac{4Ln_0\Delta}{\epsilon^2}  \right)  \frac{n^{1/2}}{n_0} + n+2 S_2
\notag\\
&=&\frac{8(L\Delta) \cdot n^{1/2} }{\epsilon^2} + n+ \frac{2n^{1/2}}{n_0}.
\end{eqnarray}
This concludes a gradient cost of $n+	8(L\Delta) \cdot n^{1/2} \epsilon^{-2}+  2n^{-1}_0 n^{1/2} $.
\end{proof}

\subsection{Proof of High Probability Results for FSP}
Set $\cK$ be the time when Algorithm \ref{algo:SPIDER-SFO} stops. We have $\cK=0$ if $\|\vv^0 \|<2\ep$, and   $\cK = \inf\{k \ge 0: \|\vv^k \| < 2\ep \}+1$ if $\|\vv^0 \|\geq 2\ep$.
It is a random stopping time. Let $K_0 = 	\lfloor4L\Delta n_0 \epsilon^{-2}\rfloor +2$. We have the following lemma:

\begin{lemma}\label{h111}
	Set the parameters $S_1$, $S_2$, $\eta$, and $q$ as in Theorem \ref{theo:FSOonehp}. 
	Then under the Assumption \ref{assu:main2},   for fixed $K_0$, define the event:
	\[
	\cH_{K_0} =
	\left(
	\| \vv^k -\nabla f(\x^k)  \|^2  \leq  \epsilon \cdot \cep
	,\quad
	\forall k \le \min(\cK, K_0)
	\right)
	.
	\]
we have $\cH_{K_0}$	occurs with probability at least $1-p$.
\end{lemma}

\begin{proof}[Proof of Lemma \ref{h111}]  
	Because when $k \geq \cK$, the algorithm has already stopped. So if $ \cK\leq k \leq K_0$, we can define  a virtual update as $\x^{k+1} = \x^k$, and $\vv^k$ is still generated by Line \ref{line3} and Line  \ref{line5} in Algorithm \ref{algo:SPIDER-SFO}.  
	
	Then let the event $\tH_{k} =  \left(
	\| \vv^k -\nabla f(\x^k)  \|^2  \leq  \epsilon  \cdot \cep \right),$ with $0\leq k\leq K_0$. We want to prove that for any $k$ with  $0\leq k \leq K_0$, $\tH_{k}$ occurs with probability at least $1 -p/(K_0+1)$.  If so,   using  the fact that  
	$$\cH_{K_0} \supseteq  \left(
	\| \vv^k -\nabla f(\x^k)  \|^2  \leq   \epsilon  \cdot \cep
	,\quad
	\forall k \le  K_0
	\right) = \bigcap_{k=0}^{K_0} (\tH_k) ,  $$
	we have $$\PP(\cH_{K_0})\geq  \PP\left(\bigcap_{k=0}^{K_0} (\tH_k) \right) = \PP\left( \left(\bigcup_{k=0}^{K_0} (\tH_k)^c   \right)^c \right) \geq   1 - \sum_{k=0}^{K_0}\PP(\tH_k^c)  =   1 - p.$$ 	
	
	We prove that  $\tH_{k}$ occurs with probability $1 -p/(K_0+1)$ for any $k$ with $0\leq k \leq K_0$.

	Let $\bxi^k$ with $k\geq0$ denote the randomness in maintaining  \SPIDER\ $\vv^k$ at iteration $k$. And $\cF^k = \sigma\{ \bxi^0, \cdots \bxi^{k}\}$, where $\sigma\{\cdot\}$ denotes the sigma field. We know that $\x^k$  and $\vv^{k-1}$  are  measurable on $\cF^{k-1}$.

	Then given $\cF^{k-1}$,  if  $k = \lfloor k/q \rfloor q$, we set 
	$$\bm\epsilon_{k, i} = \frac{1}{S_1}\left(\nabla f_{\mathcal{S}_1(i)}(\x^k) -\nabla f(\x^k)\right) $$
	where $i$ is the index  with $\mathcal{S}_1(i)$ denoting the $i$-th random component function selected at iteration $k$ and $1\leq i \leq S_1$.  We have 
	$$\EE [\bm\epsilon_{k, i}|\cF^{k-1}] = 0, \quad \|\bm\epsilon_{k, i}\|\overset{Assum. \ref{assu:main2} (iii')}{\leq} \frac{\sigma}{S_1} .   $$
	
	Then from  Proposition \ref{proh}, we have
	\begin{eqnarray}
	&&\PP\left( \| \vv^k - \nabla f(\x^k)\|^2 \geq \epsilon\cdot \cep    \mid\cF^{k-1}\right) = \PP\left(  \left\| \sum_{i=1}^{S_1}\bm\epsilon_{k,i}\right\|^2
	\geq \epsilon\cdot \cep \mid \cF^{k-1}\right)\notag\\
	&\leq& 4\exp\left( -\frac{\epsilon \cdot \cep}{4 S_1 \frac{\sigma^2}{S_1^2}} \right)\overset{S_1 =   \frac{2\sigma^2}{\epsilon^2},  ~ \cep = 10\epsilon\log(4(K_0+1)/p)   }{\leq} \frac{p}{K_0+1}.
	\end{eqnarray}
	So $\PP\left( \| \vv^k - \nabla f(\x^k)\|^2 \geq \epsilon \cdot \cep   \right) \leq \frac{p}{K_0+1} $. 
	
	When $k\neq\lfloor k/q \rfloor q $, set $k_0 =\lfloor k/q \rfloor q $, and 
	$$\bm\epsilon_{j, i} = \frac{1}{S_2}\left(\nabla f_{\mathcal{S}_2(i)}(\x^j)  - \nabla f_{\mathcal{S}_2(i)}(\x^{j-1}) -\nabla f(\x^j) + \nabla f(\x^{j-1}) \right)$$
	where $i$ is the index  with $\mathcal{S}_2(i)$ denoting the $i$-th random component function selected at iteration $k$, $1\leq i \leq S_2$ and $k_0\leq j \leq k$.  We have 
	$$\EE [\bm\epsilon_{j, i}|\cF^{j-1}] = 0.$$ 
	For any $\x$ and $\y$, we have
	\begin{eqnarray}\label{lipf}
	&&\|\nabla f(\x) - \nabla f(\y) \| = \left\|\frac{1}{n}\sum_{i=1}^n \left( \nabla f_i(\x) - \nabla f_i(\y)\right)\right\|\notag\\&\leq& \frac{1}{n}\sum_{i=1}^n\|\nabla f_i(\x) -\nabla  f_i(\y) \| \overset{Assum. \ref{assu:main2} ~(ii')}{\leq}  L\| \x - \y\|,
	\end{eqnarray}
	So $f(\x)$ also have $L$-Lipschitz continuous gradient.
	
	Then from the update rule if $k < \cK$, we have 
	$\|\x^{k+1} - \x^k\| = \|\eta (\vv^k/\| \vv^k\|) \| = \eta =\frac{\epsilon}{Ln_0} $, if $k\geq \cK$, we have $\|\x^{k+1} - \x^k\|=0\leq \frac{\epsilon}{Ln_0} $.   We have 
	\begin{eqnarray}\label{xbound}
	&&\|\epsilon_{j,i} \|\notag\\
	&\leq& \frac{1}{S_2} \left(\left\| \nabla f_i(\x^j) - \nabla f_i(\x^{j-1})\right\| + \left\| \nabla f(\x^j) - \nabla f(\x^{j-1}) \right\|\right)\notag\\
	& \overset{ \eqref{lipf}, ~ Assum. \ref{assu:main2} ~  (ii')}{\leq}& \frac{2L}{S_2}\| \x^j - \x^{j-1}\| \leq \frac{2\epsilon}{S_2 n_0},
	\end{eqnarray}
	for all $k_0<j\leq k$ and $1\leq i \leq S_2$. On the other hand, we have
	\begin{eqnarray}\label{fbound}
	&&\|\vv^k  - \nabla f(\x^k) \|\notag\\
	& = &\| \nabla f_{S_2} (\x^k) -\nabla f_{S_2} (\x^{k-1})   - \nabla f(\x^k) + \nabla f(\x^{k-1})  +  (\vv^{k-1} - \nabla f(\x^{k-1}))  \|\notag\\
	&=&\left\| \sum_{j = k_0+1}^k   \left(\nabla f_{S_2} (\x^k) -\nabla f_{S_2} (\x^{k-1})   - \nabla f(\x^k) + \nabla f(\x^{k-1})\right) + \nabla f_{S_1} (\x^{k_0})- \nabla f(\x^{k_0})   \right\|\notag\\
	&=&\left\| \sum_{j = k_0+1}^k \sum_{i=1}^{S_1} \bm\epsilon_{j,i} +  \sum_{i=1}^{S_2} \bm\epsilon_{k_0,i} \right\|.   
	\end{eqnarray}
	Plugging \eqref{xbound}   and \eqref{fbound} together,  and using   Proposition \ref{proh}, we have
	\begin{eqnarray}
	&&\PP\left( \| \vv^k - \nabla f(\x^k)\|^2 \geq \epsilon  \cdot \cep \mid \cF^{k_0-1}\right )\notag\\
	&\leq& 4\exp\left( -\frac{\epsilon \cdot \cep}{4 S_1 \frac{\sigma^2}{S_1^2} +  4 S_2 (k -k_0) \frac{4\epsilon^2}{S_2^2 n_0^2}} \right)\notag\\
	&\leq&4\exp\left( -\frac{\epsilon \cdot \cep}{4 S_1 \frac{\sigma^2}{S_1^2} +  4 S_2 q \frac{4\epsilon^2}{S_2^2 n_0^2}} \right)\notag\\
	&\overset{a}=&4\exp\left( -\frac{\epsilon^2  10\log( 4(K_0+1)/p)  }{4\sigma^2 \frac{\epsilon^2}{2\sigma^2}  +   \frac{4\epsilon n_0}{2\sigma} \frac{\sigma n_0}{\epsilon} \frac{4\epsilon^2}{ n_0^2}} \right)\leq \frac{p}{K_0+1},
	\end{eqnarray}
	where in $\overset{a}=$, we use $S_1 =  \frac{2\sigma^2}{\epsilon^2} $, $S_2 = \frac{2\sigma}{\epsilon n_0} $,  and $q =\frac{ \sigma n_0}{\epsilon}$. So $\PP\left( \| \vv^k - \nabla f(\x^k)\|^2 \geq \epsilon \cdot \cep  \right) \leq \frac{p}{K_0+1} $, which completes the proof.

\end{proof}

\begin{lemma}\label{h444}
	 Under Assumption \ref{assu:main2}, we have that on $\cH_{K_0} \cap (\cK > K_0)$, for all $0\le k \le K_0$,
	\begin{eqnarray}
	f(\x^{k+1}) - f(\x^k) 
	\leq 
	-\dfrac{\epsilon\cdot \cep}{4Ln_0}.
	\end{eqnarray}
	
	and hence
	$$
	f(\x^{K_0+1}) - f(\x^0)    \leq     -\dfrac{\epsilon\cdot \cep}{4Ln_0}\cdot (K_0)
	.
	$$
	
\end{lemma}

\begin{proof}[Proof of Lemma \ref{h444}]
	Let $\eta^k := \eta / \|\vv^k\|.$
	Since $f$ has $L$-Lipschitz continuous gradient from \eqref{lipf}, we have
	\begin{eqnarray}\label{hL-smooth}
	f(\x^{k+1})
	&\overset{\eqref{L-smooth}}{\leq}& f(\x^k)  -\eta^k\left(\frac{1}{2}-\frac{\eta^k L}{2}\right)\left\|\vv^k  \right\|^2+\frac{\eta^k}{2}\left\|\vv^k - \nabla f(\x^k)\right\|^2.
	\end{eqnarray}

	Because we are on the event $\cH_{K_0} \cap (\cK >K_0)$, so $\cK-1\geq K_0$, then  for all $0\le k\le K_0$, we have $\| \vv^k\|\geq 2\epsilon$, thus
	$$
	\eta^k= \frac{\epsilon}{Ln_0}\frac{1}{\| \vv^k\|} \overset{\|\vv^k\| \geq 2 \cep \geq2 \epsilon}{\leq} \frac{1}{2Ln_0}\leq \frac{1}{2L}
	, 
	$$
	we have
	\begin{eqnarray}\label{vgeq}
	\eta^k\left(\frac{1}{2}-\frac{\eta^k L}{2}\right)\left\|\vv^k  \right\|^2
	\geq 
	\frac{1}{4}\cdot
	\frac{\epsilon}{L n_0 \| \vv^k\|}\| \vv^k\|^2
	\overset{\| \vv^k\|\geq 2\cep}{\geq} \frac{\epsilon \cdot \cep }{2Ln_0}
	,
	\end{eqnarray}
	and for $\cH_{K_0}$ happens, we  also have
	$$
	\frac{\eta^k}{2}\left\|\vv^k - \nabla f(\x^k)\right\|^2
	\overset{\eta^k\leq \frac{1}{2Ln_0}}{\leq}
	\frac{\epsilon \cdot \cep}{4Ln_0}.
	$$
	Hence
	\begin{eqnarray}\label{fdescent}
	f(\x^{k+1})
	\leq 
	f(\x^k) - \frac{\epsilon \cdot \cep}{2Ln_0} + \frac{\eta^k}{2}\left\|\vv^k - \nabla f(\x^k)\right\|^2
	\leq 
	f(\x^k) - \frac{\epsilon \cdot \cep}{4Ln_0},
	\end{eqnarray}
	By telescoping \eqref{fdescent} from $0$ to $K_0$, we have	$$f(\x^{K_0+1}) - f(\x^0)    \leq     -\dfrac{\epsilon \cdot \cep}{4Ln_0}\cdot (K_0).$$
\end{proof}

Now, we are ready to prove Theorem \ref{theo:FSOonehp}.

\begin{proof}[Proof of Theorem \ref{theo:FSOonehp}]

	We only want to prove $(\cK \le K_0) \supseteq \cH_{K_0}$,
	so if $\cH_{K_0}$ occurs,  we have $\cK\leq K_0$, and   $\|\vv^{\cK}\| \le 2\cep$. Because $\|\vv^{\cK}-\nabla f(\x^{\cK}) \|\leq \sqrt{\epsilon\cdot \cep}\leq \cep$ occurs in $\cH_{K_0}$, so  $\|\nabla f(\x^{\cK}) \| \le 3\cep$.

	If $(\cK >  K_0 )$ and $\cH_{K_0}$ occur, plugging in $K_0 =  \lfloor\frac{4L\Delta n_0}{\epsilon^2}\rfloor+2\geq \frac{4L\Delta n_0}{\epsilon^2}+ 1\geq \frac{4L\Delta n_0}{\epsilon \cdot \cep}+ 1$, then from Lemma \ref{h444} at each iteration the function value descends by at least $\epsilon \cdot \cep/ (4Ln_0)$,
	We thus have
	$$  
	-\Delta
	\leq 
	f^* -f(\x^0) 
	\leq 
	f(\x^{K_0})  -  f (\x^0)
	\leq 
	- \left(\Delta+\frac{\epsilon \cdot \cep}{4Ln_0}\right)
	,
	$$ 
	contradicting the fact that $-\Delta
	> - \left(\Delta+\frac{\epsilon \cdot \cep}{4Ln_0}\right)$. This indicates $(\cK \le K_0) \supseteq \cH_{K_0}$.  From Lemma \ref{h111},   with probability $1-p$, $\cH_{K_0}$ occurs, and then  $\|\vv^{\cK}\| \le 2\cep$ and $\|\nabla f(\x^{\cK}) \| \le 3\cep$.

	Then gradient cost can be bounded by the same way in  Theorem \ref{111} as:
	\begin{eqnarray}
	\left\lceil K_0 \cdot \frac{1}{q}\right\rceil S_1 + K_0 S_2&\overset{S_1 = qS_2}{\le}&
	2K_0\cdot S_2 + S_1  \notag\\
	&\leq&2\left(\frac{\Delta}{\epsilon^2 / (4Ln_0) }\right) \cdot S_2 + S_1 + 4S_2  \notag\\
	&\leq&2\left(\frac{4Ln_0\Delta}{\epsilon^2}  \right)  \frac{2\sigma}{\epsilon n_0} + \frac{2\sigma^2}{\epsilon^2} + \frac{8\sigma}{\epsilon n_0}\notag\\
	&=&\frac{16L\sigma \Delta}{\epsilon^3} +\frac{2\sigma^2}{\epsilon^2}+ \frac{8\sigma}{\epsilon n_0}.
	\end{eqnarray}

\end{proof}

\begin{proof}[Proof of Theorem \ref{theo:FSOtwohp}]
	We first verify that $\tH_{k} =  \left(
	\| \vv^k -\nabla f(\x^k)  \|^2  \leq  \epsilon\cdot \cep\right)$ with $0\leq k\leq K_0$ occurs with probability $1 -\delta/(K_0+1)$ for any $k$.
	
	When $k = \lfloor k/q \rfloor q$, we have $\vv^k = \nabla f(\x^k)$.

	When $k\neq\lfloor k/q \rfloor q $, set $k_0 =\lfloor k/q \rfloor q $, and 
	$$\epsilon_{j, i} = \frac{1}{S_2}\left(\nabla f_{\mathcal{S}_2(i)}(\x^j)  - f_{\mathcal{S}_2(i)}(\x^{j-1}) -\nabla f(\x^j) + \nabla f(\x^{j-1}) \right)$$
	where $i$ is the  index  with $\mathcal{S}_2(i)$ denoting the $i$-th random component function selected at iteration $k$, from \eqref{xbound},  we have 
	$$\EE \left[\bm\epsilon_{j, i}|\cF^{j-1}\right] = 0, \quad  \left\|\bm\epsilon_{j,i}\right\| \leq\frac{2\epsilon}{S_2 n_0} ,$$ 
	for all $k_0<j\leq k$ and $1\leq i \leq S_2$. On the other hand
	\begin{eqnarray}
	&&\left\|\vv^k  - \nabla f(\x^k) \right\|\notag\\
	& = &\left\| \nabla f_{\cS_2} (\x^k) -\nabla f_{\cS_2} (\x^{k-1})   - \nabla f(\x^k) + \nabla f(\x^{k-1})  +  (\vv^{k-1} - \nabla f(\x^{k-1}) \right\|\notag\\
	&=&\left\| \sum_{j = k_0+1}^k   \left(\nabla f_{\cS_2} (\x^k) -\nabla f_{\cS_2} (\x^{k-1})   - \nabla f(\x^k) + \nabla f(\x^{k-1})\right) + \nabla f_{\cS_1} (\x^{k_0})- \nabla f(\x^{k_0})  \right\|\notag\\
	&=&\left\| \sum_{j = k_0+1}^k \sum_{i=1}^{S_1} \bm\epsilon_{j,i}  \right\|.   
	\end{eqnarray}
	Then from  Proposition \ref{proh}, we have
	\begin{eqnarray}
	&&\PP\left( \| \vv^k - \nabla f(\x^k)\|^2 \geq \epsilon\cdot \cep   \mid \cF^{k_0-1}\right )\notag\\
	&\leq& 4\exp\left( -\frac{\epsilon\cdot \cep}{ 4 S_2 (k -k_0) \frac{4\epsilon^2}{S_2^2 n_0^2}} \right)\notag\\
	&\leq&4\exp\left( -\frac{\epsilon\cdot \cep}{  4 S_2 q \frac{4\epsilon^2}{S_2^2 n_0^2}} \right)\notag\\
	&\overset{a}=&4\exp\left( -\frac{\epsilon^2  16\log( 4(K_0+1)/p)  }{  4 n_0 n^{1/2} \frac{n_0}{n^{1/2}} \frac{4\epsilon^2}{n_0^2}} \right)\leq \frac{p}{K_0+1},
	\end{eqnarray}
	where in $\overset{a}=$, we use $S_2 = n^{1/2}/ n_0  $,  and $q =n_0n^{1/2} $. So $\PP\left( \| \vv^k - \nabla f(\x^k)\|^2 \geq \epsilon\cdot \cep   \right) \leq \frac{p}{K_0+1} $, which completes the proof.

	Thus  Lemma \ref{111} holds. Then using the same technique of  Lemma \ref{h444} and Theorem \ref{theo:FSOonehp}, we have  $(\cK \le K_0) \supseteq \cH_{K_0}$.   With      probability at least $1-p$,
	$\cH_{K_0}$ occurs,  and $\|\vv^{\cK}\| \le 2\cep$ and $\|\nabla f(\x^{\cK}) \| \le 3\cep$. 
	\begin{eqnarray}
	\left\lceil K_0 \cdot \frac{1}{q}\right\rceil S_1 + K_0 S_2&\overset{S_1 = qS_2}{\le}&
	2K_0\cdot S_2 + S_1  \notag\\
	&\leq&
	2\frac{\Delta}{\epsilon^2 / (4Ln_0) } \cdot S_2 + S_1 +4S_2  \notag\\
	&=& 
	2\left(\frac{4Ln_0\Delta}{\epsilon^2}  \right)  \frac{n^{1/2}}{n_0} + n + 4 n_0^{-1} n^{1/2}
	\notag\\
	&=&\frac{8(L\Delta) \cdot n^{1/2} }{\epsilon^2} + n+4 n_0^{-1} n^{1/2}
	.
	\end{eqnarray}
\end{proof}

\subsection{Proof of Theorem \ref{theo:SSPT} for SSP}
We first restate the \NEON~ result in \cite{allen2017neon2} for NC-search  in the following Theorem:
\begin{theorem}[Theorem 1 in \cite{allen2017neon2}, \NEON2~(on-line)]\label{neon}
	Under the Assumption \ref{assu:main2} (including (ii')),  for every point $\x_0 \in \mathcal{R}^d$, for every $\delta\in (0, L]$, and every $p \in (0,1)$, the \NEON2 (NC search) output 
	$$\ww =\NEON2 (f, \x_0, \delta, p)$$
	satisfies that, with probability at least $1-p$:
	\begin{enumerate}
		\item  if $\ww = \bot$, then $\nabla f^2(\x_0)\succeq -\delta I$.
		\item  if $\ww \neq \bot$, then $\|\ww\|_2 = 1$, and $\ww^\bT \nabla^2 f(\x_0) \ww \leq \frac{\delta}{2} $.
	\end{enumerate}
	Moreover, the total number of stochastic gradient evaluations are $O\left(\log^2((d/p))L^{2}\delta^{-2}\right)$.
\end{theorem}
One can refer to \cite{allen2017neon2} for more details.

Now we prove Theorem \ref{theo:SSPT}:

 From Algorithm \ref{algo:SPIDER-SFOplus}, we can find that all the randomness in iteration $k$  come from $3$ parts: 1) maintaining   \SPIDER\ $\vv^k$ (Line \ref{sfo1}-\ref{sfo2} and \ref{sfo3}-\ref{sfo4}); 2) to conducting NC-search in Line  \ref{sfo5} (if $\text{mod} (k, \sK )=0$); 3) choosing a random direction to update $\x^k$ in Line \ref{sfo8} (if  Algorithm \ref{algo:SPIDER-SFOplus} performs first-order updates).   We  denote the randomness from the three parts as $\bxi^1_k$, $\bxi^2_k$, $\bxi^3_k$, respectively.   Let $\cF^k$ be the filtration involving  the full information of $\x_{0:k}, \vv_{0:k}$, i.e $\cF^k = \sigma \left\{\bxi^1_{0:k}, \bxi^2_{0:k},  \bxi^3_{0:k-1}\right\}$. So the randomness in iteration $k$ given $\cF^k$ only comes from $\bxi^3_{k}$ (choosing a random direction in Line  \ref{sfo8}).

Let the random index $\cI_k = 1$, if  Algorithm \ref{algo:SPIDER-SFOplus} plans to perform the first-order update, $\cI_k = 2$, if it plans to perform the second-order update,  we know that $\cI_k$ is measurable on $\cF^{ \lfloor k/\sK \rfloor  \sK}$ and also on $\cF^k$.  Because the algorithm shall be stopped if it finds $\vv^k\geq 2\cep$ when it plans to do first-order descent,  we can define a virtual update as $\x^{k+1}  = \x^k$ in Line \ref{sfo6} and \ref{sfo7}, with others unchanged if the algorithm has stopped.  Let $\cH^k_1$ denotes the event that  algorithm has not stopped before $k$, i.e.
$$\cH^k_1  = \bigcap_{i=0}^{k}\left( \left(\| \vv^{k} \|\geq 2\cep \cap  \cI_{k}=1\right)  \bigcup \cI_{k}=2 \right),$$
we have $\cH_1^k \in \cF^k$, and $\cH_1^1 \supseteq \cH_1^2\supseteq \cdots\supseteq \cH_1^k$. Let $\cH^{\sK j}_2$ denotes the event that the NC-search in iteration $\sK j$ runs successfully. And $\cH^k_3$ denotes the event that
$$\cH^k_3 =\left(\bigcap_{i=0}^{k}\left(\| \vv^{i} - \nabla f(\x^{i}) \|^2\leq \epsilon\cdot \cep \right)\right)\bigcap \left(\bigcap_{j=0}^{\lfloor k/ \sK \rfloor}  \cH_2^{j\cdot \sK} \right).$$
We know that $\cH_3^k \in \cF^k$, and $\cH_3^1 \supseteq \cH_3^2\supseteq\cdots\supseteq \cH_3^k$. And if $\cH_3^k$ happens,  all NC-search  before iteration $k$ run successfully and  $\| \vv^{i} - \nabla f(\x^{i}) \|^2\leq \epsilon\cdot \cep $ for all $0\leq i\leq k$. 

\begin{lemma}\label{SSPth}
	With the setting of Theorem \ref{theo:SSPT}, and under the Assumption \ref{assu:SSP},  we have
	\begin{eqnarray}
	\PP\left(\cH^{K_0}_3\right) \geq \frac{7}{8}.
	\end{eqnarray}
\end{lemma}
\begin{proof}
	Let event $\tH^k = \left( \| \vv^k - \nabla f(\x^k)\|^2\leq \epsilon \cdot\cep \right)$, with $0\leq k\leq K_0$.  Once we prove that $\tH^k$ occurs with probability at least $1-\frac{1}{16(K_0+1)}$,  we have $\PP  \left(  \bigcap_{i=0}^{K_0}\left(\| \vv^{i} - \nabla f(\x^{i}) \|^2\leq \epsilon\cdot \cep \right)  \right) \geq 1 - \frac{1}{16}$. On the other hand,  from Theorem \ref{neon}, we know each time the NC-search conducts successfully at probability $1- \frac{1}{16J}$, so $\PP(\cH^{K_0}_{2}) \geq 1-\frac{1}{16}$. Combining the above results, we obtain  $\PP\left(\cH^{K_0}_3\right) \geq \frac{7}{8}$.

	To prove $\PP(\tH^k)\geq 1-\frac{1}{16(K_0+1)}$,  consider the filtration $\cF^k_2 = \sigma  \left\{\bxi^1_{0:k-1}, \bxi^2_{0:k}, \cdots,  \bxi^3_{0:k-1}\right\}$, which involves the full information of $\x_{0:k}$.  We know $\x^k$ is measurable on $\cF^k_2$.  Given $\cF^k_2$,  we have
	$$\E_i \left[\nabla f_i(\x^k) -   \nabla f(\x^k)\mid \cF_2^{k}\right] =\mathbf{0}$$
	when mod$(k, p)=0$.  For  mod$(k, p)\neq0$, we have
	\begin{eqnarray}
	\E_i \left[\nabla f_i(\x^k) - \nabla f_i(\x^{k-1}) - \left(\nabla f(\x^k) - \nabla f(\x^{k-1})\right) \mid \cF^k_2\right] = \mathbf{0}. \notag
	\end{eqnarray}
	Because	 $\x^{k}$ is generated by one of the three ways:
	\begin{enumerate}
		\item  Algorithm \ref{algo:SPIDER-SFOplus} performs First-order descent, we have 
		$\|\x^{k} - \x^{k-1}\| = \|\eta (\vv^{k-1}/\| \vv^{k-1}\|) \| = \eta =\frac{\epsilon}{Ln_0} $.
		\item Algorithm \ref{algo:SPIDER-SFOplus} performs  Second-order descent, we have
		$\|\x^{k} - \x^{k-1}\| =\eta  =\frac{\epsilon}{Ln_0} $.
		\item  Algorithm \ref{algo:SPIDER-SFOplus} has already stopped. 
		$\|\x^{k} - \x^{k-1}\| =0 \leq\frac{\epsilon}{Ln_0} $.
	\end{enumerate}
	So $\vv^k - \nabla f(\x^k)$ is martingale, and the  second moment of its difference is bounded by  $\frac{\epsilon}{Ln_0}$ . We can find that the parameters  $S_1$, $S_2$, $\eta$, with $\cep = 	16\epsilon \log( 64 (K_0+1)  )$ for on-line case and $\cep = 	10\epsilon \log( 64 (K_0+1)  )$  for  off-line case are set as  the  same  in  Lemma \ref{h111} with $p = \frac{1}{16}$. Thus  using the same technique of Lemma \ref{h111},  we can obtain $\PP(\tH^k)\geq 1-\frac{1}{16(K_0+1)}$ for all $0\leq k \leq K_0$.	
\end{proof}

Let $\cH_4^k = \cH_1^k \cap \cH_3^k $.  We show that Theorem \ref{theo:SSPT} is essentially to measure the probability of the  event that $\left(\cH^{K_0}_1\right)^c \bigcap \cH^{K_{0}}_3  $.
\begin{lemma}\label{SSPend}
	If 	$\left(\cH^{K_0}_1\right)^c \bigcap \cH^{K_0}_3  $ happens,  Algorithm \ref{algo:SPIDER-SFOplus} outputs an $\x^k$  satisfying \eqref{ESSP} before $K_0$ iterations.
\end{lemma}
\begin{proof}
	Because $\left(\cH^{K_0}_1\right)^c$ happens, we know that Algorithm \ref{algo:SPIDER-SFOplus} has already stopped before $K_0$ and output $\x^k$ with $\|\vv^k\|\leq 2\cep$.  For $\cH^{K_0}_3$ happens, we have $\| \nabla f(\x^k) -\vv^k\|\leq \sqrt{\epsilon \cdot \cep}\leq \cep$. So $\| \nabla f(\x^k)\|\leq 3\cep$. Set $k_0 = \lfloor k/\sK\rfloor\sK$. Since the NC-search conducts successfully, from Theorem \ref{neon}, we have $\lambda_{\min}\left(\nabla f^2 (\x^{k_0})\right)\geq -2\delta I$.
	From Assumption \ref{111},  we have
	\begin{eqnarray}\label{rhoHess}
\!\!\!\!\!\!\!\!\!\!\!\!	\left\| \nabla f^2 (\x) -\nabla f^2 (\y) \right\|_2 \leq\left\| \frac{1}{n}  \sum_{i=1}^n\nabla f_i^2 (\x) -\nabla f_i^2 (\y) \right\|_2 \leq \frac{1}{n}  \sum_{i=1}^n \left\|\nabla f_i^2 (\x) -\nabla f_i^2 (\y) \right\|_2 \leq \rho \|\x - \y \|.
	\end{eqnarray}
	So $f(\cdot)$ has $\rho$-Lipschitz Hessian.  We have 
	\begin{eqnarray}
	&&\left\| \nabla f^2 (\x^k)-\nabla f^2 (\x^{k_0})\right\|_2   \notag\\
	&\leq&\left\|\sum_{i=k_0}^{k-1 }\left(\nabla f^2 (\x^{i+1}) - \nabla f^2 (\x^{i})\right)  \right\|_2 \notag\\
	&\leq&\sum_{i=k_0}^{k-1 }\rho \left\|\x^{i+1} - \x^{i} \right\|_2\leq \sK \frac{\rho\epsilon}{Ln_0} \overset{\sK = \frac{\delta L n_0}{\rho \epsilon}}{=}  \delta.
	\end{eqnarray}
	Thus $ 	\lambda_{\min}(\nabla f^2 (\x^k)) \geq -3\delta I$.
\end{proof}

Now, we are ready to prove Theorem \ref{theo:SSPT}.

\begin{proof}[Proof of Theorem \ref{theo:SSPT}]
	
	For all iteration $\cK$ with  $\text{mod} (\cK, \sK )=0$, given $\cF^\cK$,  we consider the case when  $\cI_{\cK}=2$ and $\cH_4^\cK$ happens.   Because $f(\cdot)$ has $\rho$-Lipschitz Hessian, we have
	\begin{eqnarray}\label{fci1}
	&&f\left(\x^{\cK+ \sK}\right)\notag\\
	&\leq& f\left(\x^{\cK}\right) + \left[\nabla f\left(\x^{\cK}\right)\right]^\bT\left[\x^{ \cK+ \sK  } - \x^{\cK}\right] + \frac{1}{2} \left[\x^{\cK+ \sK} - \x^{\cK}\right]^\bT \left[  \nabla^2 f\left(\x^{\cK}\right) \right]  \left[\x^{  \cK+ \sK } - \x^{\cK}\right]\notag\\
	&& + \frac{\rho}{6}\left\| \left[\x^{ \cK+ \sK} - \x^{\cK}\right] \right\|^3.
	\end{eqnarray}
	Because $\cH_4^{\cK}$ happens,  and $\cI_{\cK}=2$, we have $\ww_1^\bT \left[\nabla f^2(\x^{\cK})\right]\ww_1 \leq -\delta$, and by taking expectation on the random number of the sign, we have 
	$$  \E \left[\left[\nabla f\left(\x^{\cK}\right)\right]^\bT\left[\x^{\cK+ \sK} - \x^{\cK}\right]\mid \cF^\cK\right] = 0, $$
	thus 
	\begin{eqnarray}\label{fci2}
	&&\E \left[f\left(\x^{\cK+\sK}\right)\mid \cF^\cK\right]\leq f\left(\x^{\cK}\right) -  \frac{\delta^3}{2\rho^2} + \frac{\delta^3}{6\rho^2}= f\left(\x^{\cK}\right) - \frac{\delta^3}{3\rho^2}.
	\end{eqnarray} 
	
	Furthermore, by analyzing the difference of $\left(f\left(\x^{\cK}\right) -f^*\right)\bone_{\cH_4^\cK}$, where $\bone_{\cH_4^\cK}$ is the indication function for the event $\cH_4^\cK$, we have
	\begin{eqnarray}\label{com1}
	&&\E \left[\left(f\left(\x^{\cK+ \sK}\right) -f^*\right)\bone_{\cH^{\cK+ \sK}_4} \mid{\cF^\cK} -   \left(f\left(\x^{\cK}\right) -f^*\right)\bone_{\cH_4^\cK} \mid{\cF^\cK}    \right] \notag\\
	&=& \E \left[\left(f\left(\x^{\cK+ \sK}\right) -f^*\right)\left(\bone_{\cH^{\cK+ \sK}_4} -\bone_{\cH^{\cK}_4} \right) \mid{\cF^{\cK}} \right] +  \E \left[\left(f\left(\x^{\cK+ \sK}\right) -f\left(\x^\cK\right)\right)\bone_{\cH^{\cK}_4} \mid \cF^\cK\right] \notag\\
	&\overset{a}\leq& \PP\left(\cH^{\cK}_4\mid \cF^\cK\right) \E \left[f\left(\x^{\cK+\sK}\right) -f\left(\x^{\cK}\right)|\cH^{\cK}_4\cap \cF^\cK \right] \notag\\
	&\overset{\eqref{fci2}}{\leq}& -\PP\left(\cH^{\cK}_4\mid \cF^\cK\right) \frac{\delta^3}{3\rho^2},
	\end{eqnarray}
	where in $\overset{a}\leq$,  we use that $\cH^{\cK}_4\supseteq\cH^{\cK+\sK}_4$, so  $\bone_{\cH^{\cK+ \sK}_4} -\bone_{\cH^{\cK}_4}\leq 0$ and $f\left(\x^{\cK+ \sK}\right) -f^*\geq 0$, then  $\E \left[\left(f\left(\x^{\cK+ \sK}\right) -f^*\right)\left(\bone_{\cH^{\cK+ \sK}_4} -\bone_{\cH^{\cK}_4} \right) \mid{\cF^{\cK}} \right] \leq 0$.

	On the other hand,  given $\cF^\cK$,  we consider the case when  $\cI_{\cK}=1$ and $\cH_4^\cK$ happens, then for any $k$ satisfying $\cK \leq k< \cK+\sK$, we know  $\cI_{k}=1$. 
	
	Given $\cF^k$ with  $\cK \leq k<\cK+\sK$,  then from  \eqref{hL-smooth} we have
	\begin{eqnarray}
	f\left(\x^{k+1}\right)\leq  f\left(\x^{k}\right)  -\eta^{k}\left(\frac{1}{2}-\frac{\eta^{k} L}{2}\right)\left\|\vv^{k} \right\|^2+\frac{\eta^{k}}{2}\left\|\vv^{k} - \nabla f\left(\x^{k}\right)\right\|^2,
	\end{eqnarray}
	with $\eta^{k} = \eta /\left\| \vv^{k}\right\|$.
	Also $\cH_4^K$ is measurable on $\cF^k$, and 
	if  $\cH_4^K$ happens, we have $\left\|\vv^{k}\right\|\geq 2\cep$,   and $\left\|\vv^{k} - \nabla f\left(\x^{k}\right)\right\|^2\le\ep\cdot \cep $, then from  \eqref{vgeq} and \eqref{fdescent}, we have
	\begin{eqnarray}\label{fdescent2}
	f\left(\x^{k+1}\right) \leq   f\left(\x^{k}\right) - \frac{\epsilon \cdot \cep }{4Ln_0}.
	\end{eqnarray}
	Taking expectation up to  $\cF^\cK$, we have
	\begin{eqnarray}
	\E \left[f\left(\x^{k+1}\right)  - f\left(\x^k\right) \mid \cF^{\cK} \cap  \cH^k_4\right] \leq - \frac{\epsilon \cdot \cep}{4Ln_0}.
	\end{eqnarray}
	
	By analyzing the difference of $\left(f\left(\x^{k}\right) -f^*\right)\bone_{\cH_4^k}$, we have
	\begin{eqnarray}\label{com2}
	&&\E \left[\left(f\left(\x^{k+1}\right) -f^*\right)\bone_{\cH^{k+1}_4} \mid{\cF^\cK} -   \left(f\left(\x^{k}\right) -f^*\right)\bone_{\cH_4^k} \mid{\cF^\cK}    \right] \notag\\
	&=&\E \left[\left(f\left(\x^{k+1}\right) -f^*\right)\left(\bone_{\cH^{k+1}_4}-\bone_{\cH^k_4} \right) \mid{\cF^{\cK}} \right] +  \E \left[\left(f\left(\x^{k+1}\right) -f\left(\x^k\right)\right)\bone_{\cH^{k}_4} \mid \cF^\cK\right] \notag\\
	&\overset{a}\leq&\PP\left(\cH^{k}_4\mid \cF^\cK\right) \E \left[f\left(\x^{k+1}\right) -f\left(\x^{k}\right)\mid\cH^{k}_4\cap \cF^\cK \right] \notag\\
	&\leq& -\PP\left(\cH^{k}_4\mid \cF^\cK\right) \frac{\epsilon \cdot \cep}{4Ln_0},
	\end{eqnarray}
	where in $\overset{a}\leq$, we use $\bone_{\cH^{k+1}_4}-\bone_{\cH^k_4} \leq 0$ and $f\left(\x^{k+1}\right) -f^*\geq 0$.
	
	By telescoping \eqref{com2} with $k$ from $\cK$ to $\cK+\sK - 1$, we have
	\begin{eqnarray}\label{com3}
	&&\E \left[\left(f\left(\x^{\cK+ \sK}\right) -f^*\right)\bone_{\cH^{\cK+ \sK}_4} \mid{\cF^\cK} -   \left(f\left(\x^{\cK}\right) -f^*\right)\bone_{\cH_4^k} \mid{\cF^\cK}    \right] \notag\\
	&\leq&- \frac{\epsilon \cdot \cep}{4Ln_0} \sum_{i =\cK}^{\cK+\sK}\PP\left(\cH^i_4\mid \cF^k\right)\notag\\
	&\overset{a}{\leq}&   - \PP\left(\cH_4^{\cK+\sK}\mid \cF^\cK\right) \frac{\sK\epsilon^2}{4Ln_0}
	\overset{\sK = \frac{\delta L n_0}{\rho \epsilon}}= - \PP\left(\cH_4^{\cK+\sK}\mid \cF^\cK\right) \frac{\delta\epsilon}{4\rho}.
	\end{eqnarray}
	where in $\overset{a}\leq$, we use $\cH^i_4\supseteq \cH^{\cK+\sK}_4$ with $\cK\leq i\le \cK+\sK$, and then $\PP\left(\cH^i_4\mid \cF^\cK\right)\geq \PP\left(\cH_4^{\cK+\sK}\mid\cF^{\cK}\right) $.
	
	Combining \eqref{com1} and \eqref{com3}, using $\PP\left(\cH^{\cK}_4\mid \cF^k\right)\geq \PP\left(\cH_4^{\cK+\sK}\mid\cF^{\cK}\right) $, we have
	\begin{eqnarray}\label{com4}
	&&\E \left[\left(f\left(\x^{\cK+ \sK}\right) -f^*\right)\bone_{\cH^{\cK+ \sK}_4} \mid{\cF^\cK} -   \left(f\left(\x^{\cK}\right) -f^*\right)\bone_{\cH_4^k} \mid{\cF^\cK}    \right] \notag\\
	&\leq& -\PP\left(\cH_4^{\cK+\sK}\mid \cF^\cK\right)\min \left( \frac{\delta\cep}{4\rho}, \frac{\delta^3}{3\rho^2} \right). 
	\end{eqnarray}

	By taking full expectation on \eqref{com4}, and telescoping the results with $\cK = 0, \sK, \cdots, (J-1) \sK$, and using $  \PP\left(\cH_4^{\sK j}\right) \leq \PP\left(\cH_4^{J\sK}\right)$ with $j = 1, \cdots, J$,  we have
	\begin{eqnarray}\label{com5}
	&&\E \left[\left(f\left(\x^{J \sK}\right) -f^*\right)\bone_{\cH^{J \sK}_4} -  \left(f\left(\x^{0}\right) -f^*\right)\bone_{\cH_4^0}     \right] \notag\\
	&\leq&- \PP\left(\cH_4^{J\sK}\right)\min \left( \frac{\delta\cep}{4\rho}, \frac{\delta^3}{3\rho^2} \right)J. 
	\end{eqnarray}
	Substituting   the inequalities $f\left(\x^{J \sK}\right) -f^*\geq 0$,   $f\left(\x^{0}\right) -f^*\leq \Delta$, and $J =4\max\left(\left\lfloor   \frac{3\rho^2 \Delta}{\delta^3}, \frac{4 \Delta \rho}{\delta \epsilon}  \right\rfloor\right)+4 \geq \frac{4\Delta}{\min \left( \frac{\delta\cep}{4\rho}, \frac{\delta^3}{3\rho^2} \right) } $ into \eqref{com5}, we have
	\begin{eqnarray}
	\PP\left(\cH^{K_0}_4\right)\leq \frac{1}{4}.
	\end{eqnarray}
	
	Then by union bound, we have
	\begin{eqnarray}
\!\!\!\!\!\!\!\!\!\!\!	\PP\left(\cH^{K_0}_1\right) = \PP\left(\cH^{K_0}_1\bigcap \cH^{K_0}_3  \right) + \PP\left (\cH^{K_0}_1\bigcap \left(\cH^{K_0}_3\right)^c  \right) \leq  \PP\left(\cH^{K_0}_4\right) + \PP\left( \left(\cH^{K_0}_3\right)^c \right) \overset{\text{Lemma \ref{SSPth}}}{\leq}  \frac{1}{4} + \frac{1}{8}.
	\end{eqnarray}
	
	Then our proof is completed by obtaining
	\begin{eqnarray}
\!\!\!\!\!\!\!\!\!\!\!	\PP\left(\left(\cH^{K_0}_1\right)^c \bigcap   \cH^{K_0}_3\right)   = 1 -\PP\left(\left(\cH^{K_0}_1\right) \bigcup   \left(\cH^{K_0}_3\right)^c  \right)\geq 1 -\PP\left(\left(\cH^{K_0}_1\right)  \right)-\PP\left(  \left(\cH^{K_0}_3\right)^c  \right)\overset{\text{Lemma \ref{SSPth}}}{\geq} \frac{1}{2}.
	\end{eqnarray}
	From Lemma \ref{SSPend},  with probability $1/2$,  the algorithm shall be terminated before $\sK J$ iterations, and output a $\x^k$ satisfying \eqref{ESSP}.

	The total stochastic gradient complexity consists of two parts: the \SPIDER\ maintenance cost and NC-Search cost.  We  estimate them as follows:
	\begin{enumerate}
		\item    With probability $1/2$, the algorithm ends in at most $K_0$ iterations, thus the number of stochastic gradient accesses to maintain \SPIDER\ can be bounded by 
		\begin{eqnarray}\label{compute1}
		\lfloor K_0 /q  \rfloor q S_1 + S_2\notag &\overset{S_1 = qS_2}{\leq}&  2K_0 S_2 +S_1\notag\\
		&\leq& \left(4\max \left\lfloor    \frac{3\rho^2\Delta }{\delta^3}, \frac{4\Delta \rho}{ \epsilon \delta}\right\rfloor+4\right)  \left(2S_2\sK\right) +S_1.
		\end{eqnarray}
		\item  With probability $1/2$, the algorithm  ends in $K_0$ iterations, thus  there are at most  $J$ times of NC search.  From Theorem \ref{neon},  suppose the NC-search costs $\tC L^2\delta^{-2}$, where $\tC$ hides a polylogarithmic factor of $d$.	 The  stochastic gradient access for NC-Search is less than:
		\begin{eqnarray}\label{compute2}
		J L^2 \tC \delta^{-2} = \left(4\max \left\lfloor    \frac{3\rho^2\Delta }{\delta^3}, \frac{4\Delta \rho}{ \epsilon \delta}\right\rfloor+4\right) \tC L^2 \delta^{-2},
		\end{eqnarray}
		
	\end{enumerate}
	By summing  \eqref{compute1} and \eqref{compute2}, using $\max\lfloor a,b\rfloor\leq a+b$ with $a\geq0$ and $b\geq0$,  we have that  the total stochastic gradient complexity can be bounded:
	$$   4\left( \frac{3\rho^2\Delta }{\delta^3}+  \frac{4\Delta\rho}{ \epsilon \delta}+2\right)\left(2S_2\sK + \tC L^2\delta^{-2}\right)+S_1.   $$
	For the on-line case, plugging into $\sK = \frac{\delta Ln_0}{\rho \epsilon}$, $S_1 = \frac{2\sigma}{\epsilon^2}$, and $S_2= \frac{2\sigma}{n_0 \epsilon}$, the stochastic gradient complexity can be bounded:
	\[
	\frac{64\Delta L\sigma}{\epsilon^3}+\frac{48\Delta \sigma L \rho}{\epsilon^2\delta^2}+\frac{12\tC\Delta L^2 \rho^2}{\delta^5}  + \frac{16\tC \Delta L^2 \rho}{\epsilon \delta^3}
	+\frac{2\sigma^2}{\epsilon^2} +\frac{8\tC L^2}{\delta^2}  + \frac{ 32L \sigma \delta}{\rho \epsilon^2}  
	.
	\]
	
	For the off-line case,  plugging into $S_2 = \frac{n^{1/2}}{n_0}$, we obtain the stochastic gradient complexity is  bounded by:
	\[
	\frac{32\Delta Ln^{1/2}}{\epsilon^2} +\frac{12\Delta \rho L n^{1/2}}{\epsilon\delta^2}+\frac{12\tC\Delta L^2 \rho^2}{\delta^5}+ \frac{16\tC\Delta L^2 \rho}{\epsilon \delta^3} + n +\frac{8\tC L^2}{\delta^2}+ \frac{ 16L n^{1/2} \delta}{\rho \epsilon}
	.
	\]
\end{proof}

\subsection{Proof  for SZO}\label{sec:proof,sec:SZO}
\begin{proof}[Proof of Lemma \ref{zero1}]
	We have that
	\begin{eqnarray}
		&& \E_{ i, \uu} \left\|    \left[\frac{f_{i}(\x +\mu \uu) - f_{i}(\x)}{\mu} \uu - \left(\frac{f_{i}(\y +\mu \uu) - f_{i}(\y)}{\mu}\uu \right)\right]  \right\|^2\notag\\
		& = &  \E_{ i, \uu} \left\|  \< \nabla f_i (\x) - \nabla  f_i (\y) , \uu \>\uu + \frac{f_{i}(\x +\mu \uu) - f_{i}(\x) -\<\nabla f_i(\x), \mu \uu \>}{\mu} \uu  \right. \notag\\
		&& \left. - \left(\frac{f_{i}(\y +\mu \uu) - f_{i}(\y) -\<\nabla f_i(\y), \mu \uu \>  }{\mu}\uu \right)  \right \|^2 \notag\\
		& \leq &  2\E_{ i, \uu} \left\|  \< \nabla f_i (\x) -  \nabla f_i (\y) , \uu \>\uu\right\|^2 \notag\\
		&& +2 \E_{ i, \uu} \left\|\frac{f_{i}(\x +\mu \uu) - f_{i}(\x) -\<\nabla f_i(\x), \mu \uu \>}{\mu} \uu - \left(\frac{f_{i}(\y +\mu \uu) - f_{i}(\y) -\<\nabla f_i(\y), \mu \uu \>  }{\mu}\uu \right)   \right \|^2 \notag\\
		& {\leq} &  2\E_{  i, \uu} \left\|  \< \nabla f_i (\x) -  \nabla f_i (\y) , \uu \>\uu\right\|^2  +4\E_{ i, \uu} \left(\left|\frac{f_{i}(\x +\mu \uu) - f_{i}(\x) -\<\nabla f_i(\x), \mu \uu \>}{\mu}\right|^2 \left\| \uu\right\|^2\right)\notag\\
		&& +4\E_{ i, \uu}\left(\left| \frac{f_{i}(\y +\mu \uu) - f_{i}(\y) -\<\nabla f_i(\y), \mu \uu \>  }{\mu}\right|^2\left\|\uu   \right \|^2\right) \notag\\
		&\overset{a} \leq &  2\E_{ i, \uu} \left\|  \< \nabla f_i (\x) -  \nabla f_i (\y) , \uu \>\uu\right\|^2  +8 \frac{\mu^2L^2}{4}\E_\uu \| \uu\|^6 \notag\\
		&\overset{b}\leq& 2 (d+4)\E_{i} \left\|  \nabla f_i (\x) -  \nabla f_i (\y)\right\|^2+ 2 \mu^2 L^2 \E_\uu \| \uu\|^6  \notag\\
		&\leq&2 (d+4)\E_{i} \left\|  \nabla f_i (\x) -  \nabla f_i (\y)\right\|^2+  2 \mu^2L^2  \E_\uu \| \uu\|^6  \notag\\
		&\overset{c}\leq&2 (d+4)\E_{i} \left\|  \nabla f_i (\x) -  \nabla f_i (\y)\right\|^2+ 2\mu^2 (d+6)^3L^2 \notag\\
		&\leq&2 (d+4)L^2\|\x - \y\|^2+ 2\mu^2 (d+6)^3L^2,
	\end{eqnarray}
	where in  $\overset{a}{\leq}$ we use 
	$$ | f_i(\a) -f_i(\b) -\<\nabla f_i(\a), \a-\b\> |\leq \frac{L}{2}\|\a -\b\|^2,$$
	because  $f_i$ has $L$-Lipschitz continuous gradient ((6) in \citep{nesterov2011random}); $\overset{b}{\leq}$, we use
	$$
	\E_\uu \|\<\a, \uu \>\uu\|^2 \leq (d+4)\| \a\|^2
	; 
	$$
	obtained  from the same technique of (33) in \citep{nesterov2011random};
	in $\overset{c} =$, $\E_\uu \|\uu\|^6\leq (d+6)^3$ in (17) of \citep{nesterov2011random}.
\end{proof}

\begin{lemma}\label{zerok0}
	Under the Assumption \ref{assu:main2} (including (ii')), if $\lfloor k/q \rfloor q = k$, given $\x^k$, we have 
	\begin{eqnarray}
		\E \|\vv^k -  \nabla \hf(\x^k)\|^2 \leq \frac{\epsilon^2}{4}.
	\end{eqnarray}
\end{lemma}
\begin{proof}
	Let $\E_{k}$ denote that the expectation is taken  on the random number at iteration $k$ given   the full  information of $\x_{0:k}$. Denote $\nabla_j f(\x)$ as the value in the $j$-th coordinate of $\nabla f(\x)$, we have that
	\begin{eqnarray}\label{zerov-f}
		&&\E_k \| \vv^k -\nabla f(\x^k)\|^2\notag\\
		&=&\E_k \sum_{j\in [d]}  \left| \vv_j^k -\nabla_j f(\x^k)\right|^2\notag\\
		&=&\sum_{j\in [d]} \E_k \left| \frac{1}{S_1'} \sum_{i\in \cS_1'}\frac{f_i(\x^k +\mu \e_j) - f_i (\x^k)}{\mu}  - \nabla_j f(\x^k)   \right|^2 \notag\\
		&\leq& 2\sum_{j\in [d]}\E_k \left| \frac{1}{S_1'} \sum_{i\in \cS_1'}\frac{f_i(\x^k +\mu \e_j) - f_i (\x^k)}{\mu}  - \frac{1}{S_1'}\sum_{i\in \cS_1'}\nabla_j f_i(\x^k)   \right|^2 +  2\sum_{j\in [d]} \E_k \left|  \frac{1}{S_1'}\sum_{i\in \cS_1'}\nabla_j f_i(\x^k)  - \nabla_j f(\x^k)   \right|^2 \notag\\
		&\overset{a}{\leq}& \frac{2}{S_1'}\sum_{j\in [d]}\E_k \left|\frac{f_i(\x^k +\mu \e_j) - f_i (\x^k)}{\mu}  - \nabla_j f_i(\x^k)   \right|^2 +  2\sum_{j\in [d]}\E_k \left|  \frac{1}{S_1'}\sum_{i\in \cS_1'}\nabla_j f_i(\x^k)  - \nabla_j f(\x^k)   \right|^2, 
	\end{eqnarray}
	where  in $\overset{a}\leq$, we use $|a_1 +a_2+\cdots +a_s |^2\leq s |a_1 |^2 + s |a_2|^2+\cdots+ s|a_s|^2$.
	
For the first term in the right hand of \eqref{zerov-f}, because $f_i(\x)$ has $L$-Lipschitz continuous gradient, we have 
	\begin{eqnarray}
		&&\left|\frac{f_i(\x^k +\mu \e_j) - f_i (\x^k)}{\mu}  - \nabla_j f_i(\x^k)\right|\notag\\
		& =& \frac{1}{\mu} \left|f_i(\x^k +\mu \e_j) - f_i (\x^k)  - \<\nabla f_i(\x^k), \mu \e_j\> \right|  \leq \frac{1}{\mu}  \frac{L}{2}\|  \mu \e_j \|^2 = \frac{L \mu}{2}.
	\end{eqnarray}
Thus we have
	\begin{eqnarray}
		&&\E_k \| \vv^k -\nabla f(\x^k)\|^2\\
		&\leq& \frac{dL^2\mu^2}{2} + 2 \sum_{j\in [d]} \E_k\left|  \frac{1}{S_1'}\sum_{i\in \cS_1'}\nabla_j f_i(\x^k)  - \nabla_j f(\x^k)   \right|^2\notag\\
		&=&\frac{dL^2\mu^2}{2} + 2 \E_k \left\|  \frac{1}{S_1'}\sum_{i\in \cS_1'}\nabla f_i(\x^k)  - \nabla f(\x^k)   \right\|^2\notag\\
	\end{eqnarray}
	For the on-line case, due to $\mu \leq \frac{\epsilon}{2\sqrt{6}L \sqrt{d}}$, and $S_1 = \frac{96d\sigma^2}{\epsilon^2}$, we have
	\begin{eqnarray}
		&&\E_k \| \vv^k -\nabla f(\x^k)\|^2\leq  \frac{dL^2 \epsilon^2}{48L^2 d} + \frac{2}{S_1'}\sigma^2\leq\frac{\epsilon^2}{24}.
	\end{eqnarray}
	In finite-sum case, we have $\E_k \left\|  \frac{1}{S_1'}\sum_{i\in \cS_1'}\nabla f_i(\x^k)  - \nabla f(\x^k)   \right\|^2 = 0$, so $\E_k \| \vv^k -\nabla f(\x^k)\|^2\leq \frac{\epsilon^2}{24}$.

	Also from \eqref{324}, and $\mu \leq \frac{\epsilon}{\sqrt{6} n_0 L(d+6)^{3/2}}$,  we have
	\begin{eqnarray}\label{zeroff}
		\left\|\nabla f(\x^k) -\nabla \hf(\x^k) \right\|^2 \leq \frac{\mu^2 L^2 (d+3)^3}{4}\leq \frac{\epsilon^2}{6L^2 (d+6)^3} \frac{L^2 (d+3)^3}{4} \leq\frac{\epsilon^2}{24}.
	\end{eqnarray}
	We have 
	\begin{eqnarray}
		\E_k \| \vv^k -\nabla \hf(\x^k)\|^2\leq 2 \| \vv^k  - \nabla f(\x^k)\|^2 + 2\| \nabla f(\x^k)  - \nabla \hf(\x^k)\|^2\leq \frac{\epsilon^2}{6}.
	\end{eqnarray}
	
\end{proof}

\begin{lemma}\label{zero5}
	From the setting of Theorem \ref{zero3}, and under  the Assumption \ref{assu:main2} (including (ii')),	for $k_0 = \lfloor k/q \rfloor \cdot q $, we have $\E_{k_0} \| \vv^k -\nabla \hf(\x^k)  \|\leq  \epsilon^2  $.
\end{lemma}

\begin{proof}
	For $k = k_0$,   from Lemma \ref{zerok0}, we obtain the result. When $k\geq k_0$, 
	
	from Lemma \ref{zero1}, we have that
	\begin{eqnarray}
		&&\!\!\!\!\!\!\!\!\!\!\!\!\!\!\!\!\E_{\cS_2}\left\|\frac{1}{S_2}\sum_{(i,\uu)\in \cS_2} \left( \frac{f_{i}(\x^k+\mu \uu^k)- f_i(\x^k)}{\mu}\uu - \frac{f_{i}(\x^{k-1}+\mu \uu^{k-1})- f_i(\x^{k-1})}{\mu} \uu\right)  -  (\hf(\x^k) - \hf(\x^{k-1}))\right\|^2\notag\\
		& = &\frac{1}{S_2}\E_{i, \uu}\left\| \left( \frac{f_{i}(\x^k+\mu \u^k)- f_i(\x^k)}{\mu}\uu - \frac{f_{i}(\x^{k-1}+\mu \uu^{k-1})- f_i(\x^{k-1})}{\mu} \uu\right)  -  (\hf(\x^k) - \hf(\x^{k-1}))\right\|^2\notag\\
		&\overset{\eqref{288}}{\leq}& \frac{1}{S_2}\left(2(d+4)L^2\| \x^k - \x^{k-1}\|^2 + 2\mu^2 (d+6)^3L^2\right)\notag\\
		&\leq& \frac{1}{S_2}\left(2(d+4)L^2\| \eta^k \vv^k\|^2 + 2\mu^2 (d+6)^3L^2\right)\notag\\
		&\overset{\eta^k \leq \frac{\epsilon}{Ln_0 \| \vv^k\|}}{\leq}& \frac{1}{S_2}\left(  2(d+4)L^2 \frac{\epsilon^2}{L^2 n_0^2} +2 (d+6)^3L^2\frac{\epsilon^2}{6n_0^2 L^2(d+6)^3} \right)\notag\\
		&=& \frac{1}{S_2}\left((2d+9)\frac{\epsilon^2}{n_0^2}\right).
	\end{eqnarray}
	Using Proposition \ref{lemm:aggregate}, for  on-line case, we have
	\begin{eqnarray}
		\E_{k_0} \| \vv^{k} -  \nabla \hf(\x^{k})\|^2 \overset{S_2 = \frac{30(2d+9)\sigma}{\epsilon n_0}}{\leq}\frac{\epsilon^2}{6}+\sum_{j=k_0}^k \frac{\epsilon^3}{30n_0\sigma}\overset{q=\frac{5n_0\sigma}{\epsilon}}{\leq} \frac{\epsilon^2}{3}.
	\end{eqnarray}
	for finite-sum case, we have
	\begin{eqnarray}
		\E_{k_0} \| \vv^{k} -  \nabla \hf(\x^{k})\|^2 \overset{S_2 = \frac{(2d+9)n^{1/2}}{ n_0}}{\leq}\frac{\epsilon^2}{6}+\sum_{j=k_0}^k \frac{\epsilon^2}{n_0n^{1/2}}\overset{q=\frac{n_0 n^{1/2}}{6}}{\leq} \frac{\epsilon^2}{3}.
	\end{eqnarray}
	
\end{proof}

\begin{proof}[Proof of Theorem \ref{zero3}]
	By taking full expectation on  Lemma \ref{zero5}, we have
	\begin{eqnarray}
		\E\|\vv^{k} -\nabla \hf(\x^k) \|^2 \leq \frac{\epsilon^2}{3}.
	\end{eqnarray}
Thus
	\begin{eqnarray}
		\E_k \| \vv^k -\nabla f(\x^k)\|^2\leq 2 \| \vv^k  - \nabla \hf(\x^k)\|^2 + 2\| \nabla f(\x^k)  - \nabla \hf(\x^k)\|^2 \overset{\eqref{zeroff}}{\leq} \epsilon^2.
	\end{eqnarray}

	By using Lemma \ref{444}, \eqref{end1}, and \eqref{end2}, we have
	\begin{eqnarray}\label{zero12}
		\frac{1}{K}\sum_{k=0}^{K-1}\E \|\vv^k \| \leq \Delta \cdot  \frac{4Ln_0}{\epsilon}\frac{1}{K} +3\epsilon\leq 4\epsilon.
	\end{eqnarray}
	One the other hand,  by Jensen's inequality, we have
 $$(\E \|\vv^{k} -\nabla \hf(\x^k)\|)^2 =  \E \|\vv^{k} -\nabla \hf(\x^k)  \|^2  - \E \|\vv^{k} -\nabla \hf(\x^k) - \E (\vv^{k} -\nabla \hf(\x^k))  \|^2\leq \epsilon^2. $$
	So 
	\begin{eqnarray}\label{zero11}
		&&\E\|  \nabla f(\x^k) \| \notag\\
		& =& \E \| \vv^{k}  - (\vv^{k} -\nabla \hf(\x^k)  ) +  \nabla \hf(\x^k) -\nabla f(\x^k)   \|\notag\\
		&\leq&  \E \| \vv^{k}  \| + \E \| \vv^{k} -\nabla \hf(\x^k)  \|+ \E \| \nabla \hf(\x^k) -\nabla f(\x^k) \|\notag\\
		&\overset{\eqref{zeroff}}{\leq}& \E \| \vv^k\| +\epsilon + \frac{\epsilon}{2\sqrt{6}}\leq  \E \| \vv^k\| +2\epsilon.
	\end{eqnarray}
	
	We have
	\begin{eqnarray}
		\E \|\nabla f(\tilde{\x}) \|  =  \frac{1}{K}\sum_{k=0}^{K-1}  \E\|  \nabla f(\x^k) \| \overset{\eqref{zero11}}{\leq} \frac{1}{K}\sum_{k=0}^{K-1}\E \|\vv^k \|  + 2\epsilon \overset{\eqref{zero12}}{\leq} 6\epsilon.
	\end{eqnarray}
\end{proof}

\subsection{Proof of Theorem \ref{theo:lowerbdd} for Lower Bound}
Our proof is a direct extension of \citet{carmon2017lower}.
Before we drill into the proof of Theorem \ref{theo:lowerbdd}, we  first introduce the hard instance $\tf_K$ \blue{with $K\geq 1$} constructed by \citet{carmon2017lower}.
\blue{
	\begin{eqnarray}
	\hf_K(\x) \coloneqq-\Psi(1)\Phi(x_1) +\sum_{i=2}^K \left[\Psi(-x_{i-1})\Phi(-x_i) - \Psi(x_{i-1})\Phi(x_i) \right],
	\end{eqnarray}}
where the component functions are 
\begin{eqnarray}
\Psi(x)\coloneqq\left\{
\begin{array}{ll}
0&x\leq \frac{1}{2}\\
\exp\left(1- \frac{1}{(2x-1)^2}\right)&x> \frac{1}{2}
\end{array}
\right.
\end{eqnarray}
and 
\begin{eqnarray}
\Phi(x)\coloneqq \sqrt{e} \int_{-\infty}^{x} e^{-\frac{t^2}{2}}, 
\end{eqnarray}
where $x_i$ denote the value of $i$-th coordinate of $\x$, with $i\in  [d]$.
$\hf_K(\x)$ constructed by \citet{carmon2017lower} is a zero-chain function, that is for every $i\in [d]$, $\nabla_i f(\x) =0$ whenever $x_{i-1} = x_i =x_{i+1}$.  So any deterministic algorithm can only recover ``one'' dimension in each iteration \citep{carmon2017lower}.  In addition, it satisfies that : If $|x_i|\leq 1$ for any $i \leq K$, 
\begin{eqnarray}
\left\| \nabla \hf_K(\x)\right\| \geq 1.
\end{eqnarray}

Then to handle random algorithms, \citet{carmon2017lower} further consider the following extensions:
\begin{eqnarray}
\tf_{K,\B^K}(\x)  = \hf_K\left((\B^K)^\bT \rho(\x)\right)+\frac{1}{10}\|\x\|^2 =\hf_K\left( \<\b^{(1)}, \rho(\x)  \>, \ldots, \<\b^{(K)}, \rho(\x) \> \right)+\frac{1}{10}\|\x\|^2,
\end{eqnarray}
where $\rho(\x) = \frac{\x}{\sqrt{1+ \|\x\|^2/R^2}}$ and $R = 230 \sqrt{K}$, $\B^K$ is chosen uniformly at random from the space of orthogonal matrices $\cO(d,K) =\{\D \in \RR^{d\times K}\vert \D^\top \D = I_K \ \}$.
The function $\tf_{K,\B}(\x)$ satisfies the following\blue{:}
\begin{enumerate}[(i)]
	\item 
	\begin{eqnarray}\label{bound}
	\tf_{K,\B^K}(\mathbf{0}) - \inf_{\x} \tf_{K,\B^K}(\x) \leq 12 K.
	\end{eqnarray}
	\item $\tf_{K,\B^K}(\x)$ has  constant $l$ (independent of $K$ and $d$) Lipschitz continuous gradient.
	\item if $d\geq 52 \cdot 230^2 K^2 \log (\frac{2K^2}{p})$, for any algorithm $\cA$ \blue{solving} \eqref{opt_eq}  with $n=1$, and $f(\x) =\tf_{K,\B^K}(\x)$, then with probability $1-p$, 
	\begin{eqnarray}\label{bound1}
	\left\|\nabla \tf_{K,\B^K}(\x^k) \right\| \geq \frac{1}{2}, \quad  \text{for every } k\leq K.
	\end{eqnarray}
\end{enumerate}
The above properties found by \citet{carmon2017lower} is very technical.  One can refer to \citet{carmon2017lower} for more details.

\begin{proof}[Proof of Theorem \ref{theo:lowerbdd}]
	Our lower bound theorem proof is as follows. The proof mirrors Theorem 2 in  \citet{carmon2017lower} by further taking the number of individual function $n$ into account.
	Set 
	\begin{eqnarray}
	f_i(\x) \coloneqq  \frac{ln^{1/2}\epsilon^2}{L}\tf_{K,\B_i^K}(\C_i^\bT\x/b)= \frac{ln^{1/2}\epsilon^2}{L} \left(\hf_{K}\left( (\B_i^K)^\bT\rho(\C^\bT_i\x/b)\right) +\frac{1}{10}\left\|\C^\bT_i \x/b\right\|^2\right), 
	\end{eqnarray}
	and 
	\begin{eqnarray}
	f(\x) = \frac{1}{n}\sum_{i=1}^n f_i(\x).
	\end{eqnarray}
	where $\B^{nK} =  [\B^K_1, \ldots,\B^K_n]$  is chosen uniformly at random from the space of orthogonal matrices $\cO(d,K) =\{\D \in \RR^{(d/n)\times (nK)}\vert \D^\top \D = I_{(nK)} \ \}$, with each $\B^K_i \in \{\D \in \RR^{(d/n)\times (K)}\vert \D^\top \D = I_{(K)} \ \}$, $i\in [n]$,  $\C =  [\C_1, \ldots,\C_n]$   is an arbitrary  orthogonal matrices $\cO(d,K) =\{\D \in \RR^{d\times d}\vert \D^\top \D = I_{d} \ \}$, with each $\C^K_i \in \{\D \in \RR^{(d)\times (d/n)}\vert \D^\top \D = I_{(d/n)} \ \}$, $i\in [n]$.  $K =  \frac{\Delta L}{12ln^{1/2}\epsilon^2}$,  with \blue{$n\leq \frac{144 \Delta^2 L^2}{l^2 \epsilon^4}$ (to ensure $K\geq 1$)}, $b = \frac{l\epsilon}{L}$, and $R = \sqrt{230K}$. We first verify that $f(\x)$ satisfies Assumption \ref{assu:main} (\ref{a1}). 
	For  Assumption \ref{assu:main} (\ref{a1}),  from \eqref{bound}, we have 
	$$ f(\mathbf{0}) -\inf_{\x\in \RR^d} f(\x)\leq \frac{1}{n}\sum_{i=1}^n ( f_i(\mathbf{0}) -\inf_{\x\in \RR^d} f_i(\x)) \leq \frac{ln^{1/2}\epsilon^2}{L} 12K=  \frac{ln^{1/2}\epsilon^2}{L}\frac{12\Delta L}{12ln^{1/2}\epsilon^2} =\Delta\footnote{\text{If $\x^0 \neq \mathbf{0}$, we can simply translate the counter example as $f'(\x) = f(\x - \x^0)$, then $ f'(\x^0) -\inf_{\x\in \RR^d} f'(\x)\leq \Delta$.  }}.$$
	For  Assumption \ref{assu:main}(\ref{a2}),  for any \blue{$i$}, using the $\tf_{K,\B_i^K}$ has $l$-Lipschitz continuous gradient, we have
	\begin{eqnarray}
	\left\|   \nabla \tf_{K,\B_i^K}(\C_i^\bT\x/b)  - \nabla \tf_{K,\B_i^K}(\C_i^\bT\y/b)    \right\|^2 \leq  l^2\left\|\C_i^\bT (\x -\y)/b \right\|^2,
	\end{eqnarray}
	Because $\| \nabla f_i(\x)  -\nabla f_i(\y)\|^2 = \left\|\frac{ln^{1/2}\epsilon^2}{Lb} \C_i\left(  \nabla \tf_{K,\B^K_i}(\C_i^\bT\x/b)  - \nabla \tf_{K,\B^K_i}(\C_i^\bT\y/b)\right)\right\|^2$, and using $\C_i^\top\C_i = I_{d/n}$, we have
	\begin{eqnarray}
	\left\| \nabla f_i(\x)  -\nabla f_i(\y) \right\|^2 \leq \left(\frac{ln^{1/2}\epsilon^2}{L}\right)^2 \frac{l^2}{b^4} \left\| \C_i^\bT (\x -\y)  \right\|^2 = n L^2 \left\|\C_i^\bT (\x -\y) \right\|^2,
	\end{eqnarray}
	where we use $b = \frac{l\epsilon}{L}$.
	Summing $i =1, \ldots, n$ and using each $\C_i$ are orthogonal matrix, we have
	\begin{eqnarray}
	\E \| \nabla f_i(\x)  -\nabla f_i(\y) \|^2 \leq L^2\| \x -\y\|^2.
	\end{eqnarray}
	Then with 
	$$d\geq  2\max(9n^3K^2,12n^2KR^2)\log \left(\frac{2n^3K^2}{p}\right) + n^2K  \sim \cO\left(\frac{n^2 \Delta^2 L^2}{\epsilon^{4}}\log\left(\frac{n^2 \Delta^2 L^2}{\epsilon^{4}p}\right)\right),$$
	from Lemma 2 of \citet{carmon2017lower} (or similarly Lemma 7 of \citet{woodworth2016tight} and Theorem 3 of  \citet{woodworth2017lower},  also refer to Lemma \ref{lowerbound}  in the  end of the paper),  with probability at least $1-p$, after $T = \frac{nK}{2}$ iterations (at the end of iteration $T-1$), for all $I^{T-1}_i$ with $i\in [d]$,   if $I^{T-1}_i< K$,  then for any $j_i \in \{I^{T-1}_i+1, \ldots, K \}$,  we have $ \<\b_{i,j_i},\rho(\C_i^\bT\x/b)   \> \leq \frac{1}{2}$,  where $I^{T-1}_i$  denotes that the algorithm $\cA$ has called individual function $i$ with $I^{T-1}_i$ times ($\sum_{i=1}^n I^{T-1}_i = T$) at the end of iteration $T-1$,  and $\b_{i,j}$ denotes the $j$-th column of $\B^K_{i}$.  However, from \eqref{bound1}, if $ \<\b_{i,j_i},\rho(\C_i^\bT\x/b)   \> \leq \frac{1}{2}$, we will have   $\|\nabla \tf_{K,\B_i^K}(\C^\bT_i\x/b)\|\geq\frac{1}{2} $.  So $f_i$ can be solved only after $K$ times calling it. 
	
	From the above analysis, for any algorithm $\cA$, after running $	T=\frac{nK}{2} = \frac{\Delta Ln^{1/2}}{24l\epsilon^2}$ iterations, at least $\frac{n}{2}$ functions cannot be solved (the worst case is when  $\cA$ exactly solves $\frac{n}{2}$ functions), so 
	\begin{eqnarray}
	&&\left\|\nabla f(\x^{nK/2}) \right\|^2 =  \frac{1}{n^2} \left\|\sum_{i  \text{ not solved}} \frac{ln^{1/2}\epsilon^2}{Lb} \C_i  \nabla \tf_{K,\B_i^K}(\C_i^\bT\x^{nK/2}/b) \right\|^2\notag\\
	&& \overset{a}{=} \frac{1}{n^2} \sum_{i \text{ not solved}}\left\| n^{1/2}  \epsilon \nabla \tf_{K,\B_i^K}(\C_i^\bT\x^{nK/2}/b) \right\|^2\overset{\eqref{bound1}}{\geq} \frac{\epsilon^2}{8},
	\end{eqnarray}
	where in $\overset{a}=$, we use $\C_i^\top\C_j=\mathbf{0}_{d/n}$, when $i\neq j$, and $\C_i^\top\C_i= I_{d/n}$.	
\end{proof}

\begin{lemma}\label{lowerbound}
Let $\{\x\}_{0:T}$ with $T = \frac{nK}{2} $ is informed by a certain algorithm in the form \eqref{algor-need}.  Then when $d\geq 2\max(9n^3K^2,12n^3KR^2)\log (\frac{2n^2K^2}{p}) + n^2K$,  with probability $1-p$, at each iteration $0\leq t\leq T$, $\x^t$ can only recover one  coordinate. 
\end{lemma}
\begin{proof}
	The proof is essentially same to \cite{carmon2017lower} and \cite{woodworth2017lower}.   We give a proof here. Before the poof, we give the following definitions:
	  \begin{enumerate}
	  	\item Let $i^t$ denotes that at iteration $t$, the algorithm choses the $i^t$-th individual function.
	  	\item Let $I^t_i$ denotes the total  times that individual function with index $i$ has been called before iteration $k$. We have  $I^0_i=0$ with $i\in [n]$, $i\neq i^t$, and $I^0_{ i^0}=1$. And  for $t\geq 1$,
	  	\begin{eqnarray}
	   I^t_i=\left\{
	   \begin{aligned}
	  I^{t-1}_i +1, & & \quad i = i_t. \\
	  I^{t-1}_i,&  & \quad \text{otherwise}. \\
	   \end{aligned}
	   \right.
	  	\end{eqnarray} 
	   \item Let $\y^t_i =   \rho(\C_i^\bT \x^t) = \frac{\C_i^\bT \x^t}{\sqrt{R^2+ \|\C_i^\bT \x^t\|^2}}$ with $i \in [n]$. We have $\y^t_i\in \RR^{d/n}$ and $\|\y^t_i\|\leq R$.
	  \item Set $\bcV^t_i$ be the set that $ \left(\bigcup_{i=1}^n\left\{ \b_{i, 1}, \cdots  \b_{i, \min(K,I^t_i)}\right\}\right) \bigcup \left\{\y^0_i,\y^1_i, \cdots, \y^t_i\right\}$, where  $\b_{i,j}$ denotes the $j$-th column of $\B^K_{i}$. 
	  \item Set $\bcU^t_i$ be the set of  $ \left\{  \b_{i, \min(K,I^{t-1}_i+1)}, \cdots,  \b_{i, K }\right\}$  with $i\in [n]$. $\bcU^t = \bigcup_{i=1}^n \bcU^t_i$. And set $\tilde{\bcU}^t_i =\left\{\b_{i, \min(K,1)}, \cdots, \b_{i, \min(K,I^{t-1}_i)}\right\}$. $\tilde{\bcU}^t = \bigcup_{i=1}^n \tilde{\bcU}^t_i$. 
	  \item Let $\cP_i^t\in \mathcal{R}^{(d/n) \times (d/n)}$ denote the projection operator to the span of $\uu\in \bcV_i^t$. And let 
	  $\cP^{t\bot}_i$ denote its orthogonal complement.	 
\end{enumerate}
Because $\cA^t$ performs measurable mapping, the above terms are all measurable on  $\bxi$ and $\B^{nK}$, where $\bxi$ is the random vector in $\cA$.
It is clear that if for all $0\leq t\leq T$ and $i \in [n]$, we have
\begin{eqnarray}\label{lowerend}
\left| \< \u, \y^t_i    \> \right| < \frac{1}{2},  \quad  \text{for all~ } \u \in \bcU^t_i.     
\end{eqnarray}
then  at each iteration,  we can only recover one index, which is our destination. To prove that \eqref{lowerend} holds with probability at least $1-p$,  we consider a more hard event $\bcG^t$ as
\begin{eqnarray}\label{lowerbound G}
\bcG^t  = \left\{     \left|\< \u, \cP^{(t-1) \bot}_i \y^t_i    \>\right|\leq  a \|  \cP^{(t-1) \bot}_i \y_i^t\| \mid  \u \in \bcU^t ~(\text{not~} \bcU^t_i) , ~ i\in [n]\right\}, \quad t\geq 1,
\end{eqnarray}
with $a = \min \left( \frac{1}{3(T+1)}, \frac{1}{ 2(1+\sqrt{3T})R}   \right)$. 
And $G^{\leq t} = \bigcap_{j=0}^t \bcG^j$. 

We first show that if $\bcG^{\leq T}$ happens, then \eqref{lowerend} holds for all $0\leq t\leq T$.      For  $0\leq t\leq T$, and  $i \in [n]$, if $\bcU^t_i = \varnothing$, \eqref{lowerend} is right; otherwise for any $\u \in \bcU^t_i$, we have
\begin{eqnarray}\label{lowerbound uy}
&&\left| \< \u, \y^t_i    \> \right|\notag\\
&\leq&\left| \<\u, \cP^{(t-1)\bot}_i \y^t_i \>          \right| + \left| \<\u, \cP^{(t-1)}_i \y^t_i \>          \right|\notag\\
&\leq& a \| \cP_i^{(t-1)\bot} \y^t_i \|+  \left| \<\u, \cP^{t-1}_i \y^t_i \>          \right|\leq   a R+  R\left\| \cP^{t-1}_i \uu        \right\|,
\end{eqnarray}
where in the last inequality, we use $ \| \cP^{(t-1)\bot}_i \y^t_i \|\leq \|  \y^{(t-1)}_i \|\leq R$. 

If $t=0$, we have $\cP_i^{t-1} = \mathbf{0}_{d/n\times d/n}$, then $\left\| \cP_i^{t-1} \uu         \right\| = 0$, so \eqref{lowerend} holds. When $t\geq 1$, suppose at $t-1$,  $\bcG^{\leq t}$ happens then  \eqref{lowerend} holds for all $0$ to $t-1$. Then we need to prove that $\|\cP^{t-1}_i \uu  \|\leq b = \sqrt{3T}a$ with $\uu\in \bcU^t_i$ and $i\in [n]$. Instead,  we prove   a stronger results: $\|\cP^{t-1}_i \uu  \|\leq b = \sqrt{3T}a$ with all   $\uu\in \bcU^t$ and $i\in [n]$. Again,   When $t = 0$, we have $\|\cP^{t-1}_i \uu  \|= 0$, so it is right, when $t\geq 1$, 
by  Graham-Schmidt procedure on  $\y^0_{i},  \b_{i_0, \min(I^0_{i^0}, K)}, \cdots,  \y^{t-1}_{i}, \b_{i_{t-1}, \min(I^{t-1}_{i^{t-1}}, K)}$, we have
\begin{eqnarray}\label{lowerbound3}
\left\| \cP^{t-1}_i \uu        \right\|^2 = \sum_{z=0}^{t-1}\left| \<\frac{ \cP^{(z-1)\bot}_i \y^z_i}{\|\cP^{(z-1)\bot}_i \y^z_i \|}, \u     \>   \right|^2+ \sum_{z=0, ~I^{z}_{i^{z}}\leq K}^{t-1}\left| \<\frac{ \hat{\cP}^{(z-1)\bot}_i \b_{i_{z}, I^z_{i^z}}}{\|\hat{\cP}^{(z-1)\bot}_i \b_{i_{z}, I^z_{i^z}} \|}, \u     \>   \right|^2,
\end{eqnarray}
where 
 $$ \hat{\cP}^{(z-1)}_i  = \cP^{(z-1)}_i + \frac{\left(\cP^{(z-1)\bot}_i \y^z_i\right)\left(\cP^{(z-1)\bot}_i \y^z_i\right)^\bT}{\left\|\cP^{(z-1)\bot}_i \y^z_i\right\|^2}.     $$
Using $\b_{i_{z}, I^z_{i^z}} \bot \uu$ for all $\uu \in \bcU^t$, we have
 \begin{eqnarray}\label{lowerbound2}
&&\left|\<\hat{\cP}^{(z-1)\bot}_i \b_{i_{z}, I_{i^z}^z}, \u \>\right|\\
&=&\left|0-\<\hat{\cP}^{(z-1)}_i \b_{i_{z}, I_{i^z}^z}, \u \>\right|\notag\\
&\leq& \left|\<\cP^{(z-1)}_i \b_{i_{z}, I^z_{i^z}}, \u \>\right|+ \left|  \< \frac{\cP^{(z-1)\bot}_i   \y^z_i}{\|\cP^{(z-1)\bot}_i   \y^z_i \|}, \b_{i_{z}, I^z_{i^z}} \> \<  \frac{\cP^{(z-1)\bot}_i   \y^z_i}{\|\cP^{(z-1)\bot}_i   \y^z_i \|}, \uu\>      \right|.\notag
 \end{eqnarray}
 
For the first term in the right hand of \eqref{lowerbound2}, by induction, we have
\begin{eqnarray}
\left|\<\cP^{(z-1)}_i \b_{i_{z}, I^z_{i^z}}, \u \>\right| = \left|\<\cP^{(z-1)}_i \b_{i_{z}, I^z_{i^z}},\cP^{(z-1)}_i   \u \>\right|\leq b^2. 
\end{eqnarray} 
For the second  term  in the right hand of \eqref{lowerbound2}, by assumption \eqref{lowerbound G}, we have
\begin{eqnarray}
 \left|  \< \frac{\cP^{(z-1)\bot}_i   \y^z_i}{\|\cP^{(z-1)\bot}_i   \y^z_i \|},  \b_{i_{z}, I^z_{i^z}}\> \<  \frac{\cP^{(z-1)\bot}_i   \y^z_i}{\|\cP^{(z-1)\bot}_i   \y^z_i \|}, \uu\>      \right|\leq a^2.
\end{eqnarray}
 
Also,  we have
 \begin{eqnarray}\label{lowerbound4}
&&\left\|\hat{\cP}^{(z-1)\bot}_i \b_{i_{z}, I^z_{i^z}} \right\|^2\\
&=& \| \b_{i_{z}, I^z_{i^z}}\|^2 - \left\|\hat{\cP}^{(z-1)}_i\b_{i_{z}, I^z_{i^z}} \right\|^2\notag\\
&=&  \| \b_{i_{z}, I^z_{i^z}}\|^2 - \left\|\cP^{(z-1)}_i \b_{i_{z}, I^z_{i^z}} \right\|^2 -\left|  \< \frac{\cP^{(z-1)\bot}_i   \y^z_i}{\|\cP^{(z-1)\bot}_i   \y^z_i \|},  \b_{i_{z}, I^z_{i^z}}\>       \right|^2\notag\\
&\geq&1 - b^2 - a^2.\notag
\end{eqnarray}
 Substituting \eqref{lowerbound2}  and \eqref{lowerbound4} into \eqref{lowerbound3}, for all $\u\in \bcU^t$, we have
 \begin{eqnarray}
\left\| \cP^{t-1}_i \uu        \right\|^2 &\leq& t a^2 + t \frac{ (a^2+b^2)^2}{ 1- (a^2 +b^2)}\notag\\
&\overset{a^2 +b^2 \leq  (3T+1)a^2 \leq a}{\leq}& T a^2+ T \frac{a^2}{ 1-a}\overset{a\leq 1/2}\leq 3T a^2 = b^2.
 \end{eqnarray}
 
 Thus for \eqref{lowerbound uy},  $t\geq 1$, because $\u \in \bcU^t_i\subseteq \bcU^t $,  we have
 \begin{eqnarray}
\left| \< \u, \y^t_i    \> \right|\leq (a+b)R \overset{a\leq \frac{1}{ 2(1+\sqrt{3T})R}}{\leq}\leq \frac{1}{2}.
 \end{eqnarray}
 This shows that  if $\bcG^{\leq T}$ happens,  \eqref{lowerend} holds for all $0\leq t\leq T$.  Then we prove that $\PP(\bcG^{\leq T})\geq 1- p$. We have
 \begin{eqnarray}
\PP\left((\bcG^{\leq T})^c\right) &=& \sum_{t=0}^T \PP\left((\bcG^{\leq t})^c\mid \bcG^{< t} \right) .
  \end{eqnarray}
We give the following definition:
\begin{enumerate}
	\item  Denote $\hat{i}^t$ be the sequence of $i_{0:t-1}$.  Let $\hat{\cS}^t$ be the  set that  contains all possible ways of $\hat{i}^t$ ($|\hat{\cS}^t |\leq n^t $).
	\item Let $\tilde{\UU}^j_{\hat{i}^t} = [\b_{j, 1}, \cdots,  \b_{j, \min(K,I^{t-1}_j)}]$ with $j\in [n]$, and $\tilde{\UU}_{\hat{i}^t} =[\tilde{\UU}^1_{\hat{i}^t},\cdots,\tilde{\UU}^n_{\hat{i}^t}]$. $\tilde{\UU}_{\hat{i}^t}$ is  analogous to $\tilde{\bcU^t}$, but is a matrix.
	\item Let $\UU^j_{\hat{i}^t} = [ \b_{j, \min(K,I^t_j)};\cdots;\b_{j, K}] $ with $j\in [n]$, and $\UU_{\hat{i}^t} =[\UU^1_{\hat{i}^t},\cdots,\UU^n_{\hat{i}^t}]$. $\UU_{\hat{i}^t}$ is  analogous to $\bcU^t$, but is a matrix.  Let $\bar{\UU} = [\tilde{\UU}_{\hat{i}^t}, \UU_{\hat{i}^t}  ] $.
	\end{enumerate}
We have that
 \begin{eqnarray}
 &&\PP\left((\bcG^{\leq t})^c\mid \bcG^{< t} \right)\\
&=&\sum_{\hat{i}^t_0\in  \hat{\cS}^t} \E_{\bxi, \UU_{\hat{i}^t_0}} \left(\PP\left((\bcG^{\leq t})^c\mid \bcG^{< t}, \hat{i}^t=\hat{i}^t_0, \bxi,   \UU_{\hat{i}^t_0} \right)\PP\left(\hat{i}^t = \hat{i}^t_0 \mid  \bcG^{< t},   \bxi,   \UU_{\hat{i}^t_0} \right)\right).\notag
 \end{eqnarray}
For  $\sum_{\hat{i}^t_0\in  \hat{\cS}^t} \E_{\bxi, \UU_{\hat{i}^t_0}}\PP\left(\hat{i}^t = \hat{i}^t_0 \mid  \bcG^{< t},   \bxi,   \UU_{\hat{i}^t_0} \right) =  \sum_{\hat{i}^t_0\in  \hat{\cS}^t} \PP\left(\hat{i}^t = \hat{i}^t_0\mid  \bcG^{< t} \right) = 1$,  
 in the rest, we show that the probability $\PP\left((\bcG^{\leq t})^c\mid \bcG^{< t}, \hat{i}^t= \hat{i}^t_0, \bxi = \bxi_0,  \tilde{\UU}_{\hat{i}^t_0} = \tilde{\UU}_0, \right)$ for all  $\xi_0 ,\tilde{\UU}_0$ is small. By union bound, we have
\begin{eqnarray}
&&\PP\left((\bcG^{\leq t})^c\mid \bcG^{< t}, \hat{i}^t= \hat{i}^t_0, \bxi = \bxi_0,   \tilde{\UU}_{\hat{i}^t_0} = \tilde{\UU}_0 \right)\\
&\leq& \sum_{i=1}^n\sum_{\uu\in \bcU^t}\PP \left(\<\uu, \cP_i^{(t-1)\bot}\y^t_i \>\geq a \| \cP_i^{(t-1)\bot}\y^t_i\|\mid \bcG^{< t}, \hat{i}^t= \hat{i}^t_0, \bxi = \bxi_0,   \tilde{\UU}_{\hat{i}^t_0} = \tilde{\UU}_0\right) \notag. 
\end{eqnarray}
Note that $\hat{i}^t_0$ is a constant. Because given $\bxi$ and $\tilde{\UU}_{\hat{i}^t_0}$, under $G^{\leq t}$,  both  $\cP_i^{(t-1)}$ and $\y^t_i$ are known.  We prove 
 	\begin{eqnarray}
\!\!\!\! \!\!\!\! \!\!\!\! 	\PP\left( \UU_{\hat{i}^t_0} =  \UU_0\mid \bcG^{< t}, \hat{i}^t= \hat{i}^t_0, \bxi = \bxi_0,   \tilde{\UU}_{\hat{i}^t_0} = \tilde{\UU}_0\right) =  	\PP\left(\UU_{\hat{i}^t_0} =  \Z_i\UU_0\mid \bcG^{< t}, \hat{i}^t=\hat{i}^t_0, \bxi = \bxi_0,   \tilde{\UU}_{\hat{i}^t_0} = \tilde{\UU}_0\right),
 	\end{eqnarray}
 	where $\Z_i\in \RR^{d/n\times d/n}$, $\Z_i^\bT\Z_i = \I_d$, and $\Z_i \u  = \u = \Z^\bT_i\u$ for all $\u\in \bcV^{t-1}_i $. In  this way,   $\frac{\cP_i^{(t-1)\bot}\u}{\|\cP_i^{(t-1)\bot}\u\|}$  has uniformed distribution on the unit space.  To prove it, we have
 \begin{eqnarray}
 	&&\PP\left( \UU_{\hat{i}^t_0} = \UU_0 \mid \bcG^{< t}, \hat{i}^t=\hat{i}^t_0, \bxi = \bxi_0,   \tilde{\UU}_{\hat{i}^t_0} = \tilde{\UU}_0\right)\notag\\
 	 &=&  \frac{\PP(\UU_{\hat{i}^t_0} = \UU_0, \bcG^{< t}, \hat{i}^t=\hat{i}^t_0, \bxi = \bxi_0,   \tilde{\UU}_{\hat{i}^t_0} = \tilde{\UU}_0)}{\PP( \bcG^{< t}, \hat{i}^t=\hat{i}^t_0,  \bxi = \bxi_0,   \tilde{\UU}_{\hat{i}^t_0} = \tilde{\UU}_0)}\notag\\
 	 &=& \frac{\PP( \bcG^{< t}, \hat{i}^t=\hat{i}^t_0 \mid \bxi = \bxi_0, \UU_{\hat{i}^t_0} =\UU_0,  \tilde{\UU}_{\hat{i}^t_0} = \tilde{\UU}_0) p(\bxi = \bxi_0, \UU_{\hat{i}^t_0} =\UU_0,  \tilde{\UU}_{\hat{i}^t_0} = \tilde{\UU}_0)}{\PP( \bcG^{< t}, \hat{i}^t=\hat{i}^t_0, \bxi = \bxi_0,   \tilde{\UU}_{\hat{i}^t_0} = \tilde{\UU})},
 \end{eqnarray}
 And 
  \begin{eqnarray}
 	&&\PP\left(\UU_{\hat{i}^t_0} = \Z_i \UU_0 \mid \bcG^{< t}, \hat{i}^t=\hat{i}^t_0, \bxi = \bxi_0,   \tilde{\UU}_{\hat{i}_0} = \tilde{\UU}_0\right)\notag\\
 	&=& \frac{\PP( \bcG^{< t}, \hat{i}^t=\hat{i}^t_0 \mid \bxi = \bxi_0, \UU_{\hat{i}^t_0} =\UU_0,  \tilde{\UU}_{\hat{i}^t_0} = \Z_i \tilde{\UU}_0) p(\bxi = \bxi_0, \UU_{\hat{i}^t_0} =\Z_i \UU_0,  \tilde{\UU}_{\hat{i}^t_0} = \tilde{\UU}_0)}{\PP( \bcG^{< t}, \hat{i}^t=\hat{i}^t_0, \bxi = \bxi_0,   \tilde{\UU}_{\hat{i}^t_0} = \tilde{\UU}_0)}
  \end{eqnarray}
For $\bxi$ and $\bar{\UU}$ are independent. And $p(\bar{\UU}) = p(\Z_i \bar{\UU})$, we have $p(\bxi = \bxi_0, \UU_{\hat{i}^t_0} =\UU_0,  \tilde{\UU}_{\hat{i}^t_0} = \tilde{\UU}_0) =p(\bxi = \bxi_0, \UU_{\hat{i}^t_0} =\Z_i \UU_0,  \tilde{\UU}_{\hat{i}^t_0} = \tilde{\UU}_0) $. Then
  we  prove that if $ \bcG^{< t}$ and $\hat{i}^t=\hat{i}^t_0$ happens under  $\UU_{\hat{i}^t_0} =\UU_0, \bxi = \bxi_0,   \tilde{\UU}_{\hat{i}^t_0} = \tilde{\UU}_0$, if and only if $ \bcG^{< t}$ and $\hat{i}^t=\hat{i}^t_0$ happen under  $\UU_{\hat{i}^t_0} =\Z_i \UU_0, \bxi = \bxi_0,   \tilde{\UU}_{\hat{i}^t_0} = \tilde{\UU}_0$.
  
  Suppose at iteration $l-1$ with $l\leq t$, we have the result. At iteration $l$, suppose  $ \bcG^{<l}$ and $\hat{i}^l = \hat{i}^l_0$ happen, given    $\UU_{\hat{i}^t_0} =\UU_0, \bxi = \bxi_0,   \tilde{\UU}_{\hat{i}^t_0} = \tilde{\UU}_0$. Let $\x'$ and $(\hat{i}')^j$ are generated by  $\bxi = \bxi_0,  \UU_{\hat{i}^t_0} =\Z_i \UU_0, \tilde{\UU}_{\hat{i}^t_0} = \tilde{\UU}_0$. Because $\bcG^{<l}$ happens, thus at each iteration, we can only recover one index until $l-1$.  Then $(\x')^j=\x^j$ and $(\hat{i}')^j = \hat{i}^j$. with $j\leq l$.  By induction, we only need to prove that $\bcG^{l-1'}$ will happen.   Let $\uu \in \bcU^{l-1}$, and $i\in [n]$, we have
 \begin{eqnarray}
  \left|\< \Z_i\u, \frac{\cP^{(l-2) \bot}_i \y^{l-1}_i    }{\|  \cP^{(l-2) \bot}_i\y^{l-1}_i\|}\>\right|=  \left|\< \u, \Z_i^\bT\frac{\cP^{(l-2) \bot}_i \y^{l-1}_i    }{\|  \cP^{(l-2) \bot}_i \y_i^{l-1}\|}\>\right|\overset{a}{=} \left|\< \u, \frac{\cP^{(l-2) \bot}_i \y^{l-1}_i    }{\|  \cP^{(l-2) \bot}_i \y^{l-1}_i\|}\>\right|,
 \end{eqnarray}
where in $\overset{a}=$, we use $\cP^{(l-2) \bot}_i \y^{l-1}_i$ is in the span of $\bcV^l_i\subseteq \bcV^{t-1}_i $.  This shows that  if $ \bcG^{< t}$ and $\hat{i}^t=\hat{i}^t_0$ happen under  $\UU_{\hat{i}^t_0} =\UU_0, \bxi = \bxi_0,   \tilde{\UU}_{\hat{i}^t_0} = \tilde{\UU}_0$, then $\bcG^{< t}$ and $\hat{i}^t=\hat{i}^t$ happen under  $\UU_{\hat{i}^t_0} =\Z_i \UU_0, \bxi = \bxi_0,   \tilde{\UU}_{\hat{i}^t_0} = \tilde{\UU}_0$. In the same way, we can prove the necessity.
Thus for any $\uu\in \UU^t$, if $\| \cP_i^{(t-1)\bot}\y^t_i\|\neq 0$ (otherwise, $ \left|\< \u, \cP^{(t-1) \bot}_i \y^t_i    \>\right|\leq  a \|  \cP^{(t-1) \bot}_i \y_i^t\|$ holds), we have
\begin{eqnarray}
&&\PP \left(\<\uu, \frac{\cP_i^{(t-1)\bot}\y^t_i }{ \| \cP_i^{(t-1)\bot}\y^t_i\|}\>\geq a \mid \bcG^{< t}, \hat{i}^t= \hat{i}^t_0, \bxi = \bxi_0,   \tilde{\UU}_{\hat{i}^t_0} = \tilde{\UU}_0\right)\notag\\
&\overset{a}\leq&\PP \left(\<\frac{\cP_i^{(t-1)\bot}\uu }{ \| \cP_i^{(t-1)\bot}\uu\|}, \frac{\cP_i^{(t-1)\bot}\y^t_i }{ \| \cP_i^{(t-1)\bot}\y^t_i\|}\>\geq a \mid \bcG^{< t}, \hat{i}^t= \hat{i}^t_0, \bxi = \bxi_0,   \tilde{\UU}_{\hat{i}^t_0} = \tilde{\UU}_0\right)\notag\\
&\overset{b}{\leq}& 2e^{\frac{-a^2 (d/n -2T)}{2}},
\end{eqnarray}
where in $\overset{a}{\leq}$, we use $\| \cP_i^{(t-1)\bot}\uu\|\leq 1$; and in $\overset{b}\leq$, we use $\frac{\cP_i^{(t-1)\bot}\y^t_i }{ \| \cP_i^{(t-1)\bot}\y^t_i\|}$ is a known unit vector and $\frac{\cP_i^{(t-1)\bot}\u}{\|\cP_i^{(t-1)\bot}\u\|}$  has uniformed distribution on the unit space. Then by union bound, we have  $\PP\left(\left(\bcG^{\leq t}\right)^c\mid \bcG^{< t} \right) \leq 2(n^2K)e^{\frac{-a^2 (d/n -2T)}{2}} $.
  Thus
\begin{eqnarray}
\PP\left( \left(\bcG^{\leq T}\right)^c\right) &\leq&  2(T+1)n^2K\exp\left(\frac{-a^2 (d/n -2T)}{2}\right)\notag\\
& \overset{T = \frac{nK}{2}}{\leq} & 2(nK)(n^2K)\exp\left(\frac{-a^2 (d/n -2T)}{2}\right). 
\end{eqnarray}
Then by setting 
\begin{eqnarray}
d/n &\geq& 2\max(9n^2K^2,12nKR^2)\log (\frac{2n^3K^2}{p}) + nK \notag\\
 &\geq& 2\max(9 (T+1)^2,2 (2\sqrt{3T})^2R^2)\log (\frac{2n^3K^2}{p}) + 2T\notag\\
 &\geq& 2\max(9 (T+1)^2,2 (1+\sqrt{3T})^2R^2)\log (\frac{2n^3K^2}{p}) + 2T\notag\\
 &\geq& \frac{2}{a^2}\log (\frac{2n^3K^2}{p}) + 2T,
\end{eqnarray}
we have $\PP\left(\left(\bcG^{\leq 	T}\right)^c\right)\leq p$. This ends proof.

\end{proof}

\end{document}